\documentclass[11pt,a4paper]{article}
\usepackage{epsf,epsfig,amsfonts,amsgen,amsmath,amstext,amsbsy,amsopn,amsthm,mathtools}
\usepackage{amsmath,times,mathptmx,needspace}
\usepackage{microtype,enumitem,booktabs,array}
\usepackage{amsfonts}
\usepackage{graphics}
\usepackage{latexsym,bm}
\usepackage{amsfonts,amsthm,amssymb,bbding}
\usepackage{indentfirst}
\usepackage{graphicx}
\usepackage{color}
\usepackage[colorlinks=true,anchorcolor=blue,filecolor=blue,linkcolor=blue,urlcolor=blue,citecolor=blue]{hyperref}
\usepackage{float}
\usepackage{tikz}

\newcommand{\T}{\mathsf T}

\evensidemargin=\oddsidemargin\topmargin=-1.5cm

\newtheorem{thm}{Theorem}[section]

\newtheorem{obs}{Observation}[section]
\newtheorem{lem}{Lemma}[section]
\newtheorem{cor}{Corollary}[section]

\newtheorem{remark}{Remark}[section]

\newtheorem{claim}{Claim}[section]
\newtheorem{definition}{Definition}[section]

\addtocounter{section}{0}

\begin{document}
\title{\bf Spectral extrema on closed surfaces of fixed Euler genus\footnote{Supported by
the National Natural Science Foundation of China (Nos. \!12571369, \!12271162, \!12501471).}}
\vspace{-0.5cm}
\author{{\bf Mingqing Zhai$^a$}, {\bf Longfei Fang$^{b,c}$},
{\bf Huiqiu Lin$^{b}$}\thanks{Corresponding author: huiqiulin@126.com
(H. Lin)}
\\
\footnotesize $^{a}$ School of Mathematics and Statistics, Nanjing University of Science and Technology, \\
\footnotesize   Nanjing, Jiangsu 210094, China\\
\footnotesize $^{b}$ School of Mathematics, East China University of Science and Technology,
\footnotesize  Shanghai 200237, China \\
\footnotesize $^{c}$ School of Mathematics and Finance, Chuzhou University, Chuzhou, Anhui 239012, China
}

\date{}

\vspace{-0.5cm}

\maketitle
\vspace{-10pt}
{\flushleft\large\bf Abstract}
{\small{
We investigate a classical extremal problem in combinatorics and topological graph theory:
determining the maximum spectral radius $spex(n,\gamma)$ of all $n$-vertex graphs embeddable on a closed surface of Euler genus $\gamma$.
Building on seminal contributions by Boots and Royle, Cao and Vince, and Hong in the 1990s;
Ellingham and Zha in 2000; Dvořák and Mohar in 2010; Tait and Tobin in 2017; and Liu, Ning, and Wang in 2026,
we provide a complete characterization of spectral extremal graphs for any fixed Euler genus $\gamma$ and all sufficiently large $n$.
The extremal graph attaining $\operatorname{spex}(n,\gamma)$ is unique,
necessarily edge-extremal, and, roughly speaking,
is constructed from the join $K_2 \nabla P_{n-2}$ by adding exactly $3\gamma$ edges.

As a consequence, we derive asymptotically tight estimates for $\operatorname{spex}(n,\gamma)$,
strengthening a classical bound due to Ellingham and Zha [J. Combin. Theory B, 2000].
We further characterize the extremal structures for the projective plane and torus,
sharpening earlier bounds of Hong [J. Combin. Theory B, 1995].
 For the planar case, we establish an explicit vertex threshold of $4.5\times 10^6$ for the Tait--Tobin theorem.
Our approach combines extremal constructions of cellular embeddings, topological edge-switching techniques,
and walk-counting estimates for large-scale substructures,
revealing a bridge between spectral extremality and surface topology.
}}

\begin{flushleft}
\end{flushleft}
\vspace{-1mm}

\section{Introduction}
The interplay between graph structure and topological embeddings is a central topic in structural graph theory.
The celebrated theorem of Archdeacon--Huneke--
Robertson--Seymour \cite{ARC1,ROB} asserts that
graphs embeddable on closed compact surfaces of fixed Euler genus can be characterized by a finite set of excluded minors,
a result recognized by Lov\'{a}sz as a milestone in graph theory \cite{LL}.
In particular,
the influence of topological constraints arising from surface embeddings on spectral extremal configurations has attracted sustained attention since the 1990s.

Let $\mathbb{G}(n,\gamma)$ denote the family of simple graphs on $n$ vertices that can be embedded on a surface with Euler genus $\gamma$.
The spectral radius of a graph $G$, denoted by $\rho(G)$, is the largest eigenvalue of its adjacency matrix. We define
\[
spex(n,\gamma)=\max\big\{\rho(G): G\in \mathbb{G}(n,\gamma)\big\}.
\]
We also write $EX(n,\gamma)$ and $SPEX(n,\gamma)$ for the families of edge-extremal and spectral-extremal graphs in $\mathbb{G}(n,\gamma)$, respectively.
Let $G \nabla H$ denote the join of two graphs $G$ and $H$.

For planar graphs, the spectral extremal problem is closely related to a classical conjecture.
In the 1990s, Boots and Royle \cite{BOOT} and, independently, Cao and Vince \cite{CAO} conjectured
that $\rho(G)\leq \rho(K_2\nabla P_{n-2})$ for every planar graph $G$ of order $n\geq9$,
with equality if and only if $G\cong K_2\nabla P_{n-2}$.
Since then, this conjecture has attracted considerable attention,
and various related results have been established; see, for example, \cite{CAO,Gui,Hong1,Hong2,Tait}.
In particular, Ellingham and Zha \cite{Elli} proved the general bound $\rho(G)\leq2+\sqrt{2n\!-\!6}$ for every planar graph $G$.
Using a decomposition theorem of Gon\c{c}alves \cite{HS},
Dvo\v{r}\'{a}k and Mohar \cite{DM} established that $\rho(G)\leq\sqrt{8\Delta\!-\!16}+2\sqrt{3}$
for every planar graph $G$ with maximum degree $\Delta(G)\geq2$.
Later, Tait and Tobin \cite{Tait1} confirmed the conjecture for sufficiently large $n$.
Very recently, Liu, Ning, and Wang \cite{LIU} settled the remaining cases.

Beyond the planar case, much less is known about spectral extremal structures on higher-genus surfaces.
Hong \cite{Hong2} initiated the study of $spex(n,\gamma)$ for general Euler genus and obtained the following upper bound:
$$spex(n,\gamma)\leq 1+\sqrt{3n+6\gamma-8}.$$
Subsequently, Ellingham and Zha \cite{Elli} improved this estimate to
$$spex(n,\gamma)\leq2+\sqrt{2n+8\gamma-6}.$$
Earlier, Hong \cite{Hong1} established upper bounds in the special cases of projective planar and toroidal graphs,
showing that $\rho(G)\leq2\sqrt{2}+\sqrt{10(n\!-\!1)/3}$ for every projective planar graph $G$,
and $\rho(G)\leq 2\sqrt{2}+\sqrt{12(2n\!+\!1)/7}$ for every toroidal graph $G$.
Dvo\v{r}\'{a}k and Mohar \cite{DM} proved that $\rho(G)\leq d(\gamma)+\sqrt{8(\Delta\!-\!d(\gamma))}$ for every graph $G$ in $\mathbb{G}(n,\gamma)$
with $\Delta(G)\geq d(\gamma)$, where $d(\gamma)=10$ for $\gamma\leq1$, $d(\gamma)=12$ for $\gamma\in\{2,3\}$,
$d(\gamma)=2\gamma+6$ for $\gamma\in\{4,5\}$ and $d(\gamma)=2\gamma+4$ for $\gamma\geq6$.
However, despite these developments, a complete structural description of spectral extremal graphs on high-genus surfaces remains widely open,
even in the asymptotic regime.

The main contribution of this paper is the resolution of this structural problem within a unified framework.
We prove that for every fixed $\gamma$ and all sufficiently large $n$,
spectral extremal graphs admit a unique global structure.

\begin{thm}\label{thm1.1}
For $n\geq50\times(300+180\gamma+24\gamma^2)^2$,
we have $SPEX(n,\gamma)\subseteq EX(n,\gamma)$, and every graph $G$ in $SPEX(n,\gamma)$
must be obtained from $K_2\nabla P_{n-2}$ by adding exactly $3\gamma$ edges.
Furthermore, if $\gamma\geq1$, then
$$
\rho_0+\frac{3\gamma-1}{n}<spex(n,\gamma)<\rho_0+\frac{3\gamma-0.95}{n},
$$
where $\rho_0=\frac32+\sqrt{2n\!-\!\frac{15}4}$.
\end{thm}

In Theorem \ref{thm1.1}, the asymptotically sharp estimates for $\mathrm{spex}(n,\gamma)$
improve upon the classical bounds obtained by Hong \cite{Hong2} and Ellingham and Zha \cite{Elli}.
When $\gamma=0$,
Theorem \ref{thm1.1} yields an explicit threshold of $n$ for the Tait--Tobin theorem \cite{Tait1}.

\begin{cor}\label{cor1.1}
For $n\geq4.5\times10^6$, the graph
$K_2\nabla P_{n-2}$ uniquely achieves the maximum spectral radius among all planar graphs of order $n$.
\end{cor}

To state our next main theorem, we introduce the following graph definitions.
These graphs serve as the base graphs for the spectral extremal graphs in $SPEX(n,\gamma)$.

\begin{definition}\label{def1.1}
For every $\gamma\geq1$, we define a graph $H_{0,\gamma}$ with a pair of attachment vertices.

\vspace{0.5mm}
(i) $H_{0,1}=K_4$ and $H_{0,2}=K_5$, where we choose any two distinct vertices as attachment vertices.

(ii) $H_{0,3}$ is obtained from $K_5$ on vertices $v_0,\ldots,v_4$ by adding two new vertices $v_5$ and $v_6$,
whose neighborhoods are $\{v_2,v_3,v_4\}$ and $\{v_3,v_4\}$, respectively. Its attachment vertices are $v_0$ and $v_1$.

(iii) $H_{0,4}$ is obtained from $K_6$ on vertices $v_0,\ldots,v_5$
by adding two new vertices $v_6$ and $v_7$ with neighborhoods $\{v_2,v_4\}$ and $\{v_3,v_5\}$, respectively.
Its attachment vertices are $v_0$ and $v_1$.

(iv) For $\gamma\geq5$, construct $R_{k}$ recursively as follows.
Begin with a triangle $R_0=w_0w_2w_3$ and define the initial ordered triple as $(w_0,w_2,w_3)$.
For $k=1,\ldots,\gamma-1$,
add vertices $w_{2k+2},w_{2k+3}$ and a path $w_{2k+1}w_{2k+2}w_{2k+3}w_{2k-2}$ to obtain $R_k$,
and update the ordered triple to $(w_{2k},w_{2k+2},w_{2k+3})$.
Denote by $(v_1,v_2,v_3)$ the final ordered triple.
Define $H_{0,\gamma}=K_1\nabla R_{\gamma-1}$, where $V(K_1)=\{v_0\}$,
and take $v_0, v_1$ as the attachment vertices.
\end{definition}

\begin{remark}\label{rem1.1}
Each recursive step in (iv) of Definition~\ref{def1.1} adds a path of length
three connecting the first and third vertices of the preceding ordered triple.
Every $R_{k}$ has $2k+3$ vertices and $3k+3$ edges.
Induction shows that $R_{k}$ admits a spanning path whose endpoints are the first two vertices of its current ordered triple.
At every stage, exactly the three vertices in the current ordered triple are of degree two,
while all remaining vertices have degree three.
\end{remark}

\begin{remark}\label{rem1.2}
For any positive integer $\gamma$,
$H_{0,\gamma}$ can be obtained by adding $3\gamma$ edges to a spanning path
whose endpoints are its attachment vertices.
\end{remark}

Let $H_{0,\gamma}(a,b)$ be the graph obtained from the base graph $H_{0,\gamma}$
by appending pendant paths of lengths $a$ and $b$ to its two attachment vertices.
For $1\leq\gamma\leq4$, the attachment vertices of $H_{0,\gamma}$ are interchangeable via an automorphism.
For $\gamma\geq5$, we fix the longer path to be attached at $v_0$.

\begin{thm}\label{thm1.2}
For any fixed positive integer $\gamma$ and all sufficiently large $n$, we have
$$SPEX(n,\gamma)=\big\{K_2\nabla H_{0,\gamma}(a,b)\big\},$$
where $a=\big\lceil\frac{n-2-|H_{0,\gamma}|}{2}\big\rceil$ and $b=\big\lfloor\frac{n-2-|H_{0,\gamma}|}{2}\big\rfloor.$
Moreover, $\gamma\big(K_2\nabla H_{0,\gamma}(a,b)\big)=\gamma$.
\end{thm}

Let $K_r^n$ denote the graph formed by attaching two paths to distinct vertices of $K_r$,
with lengths $\lceil\frac{n-r}2\rceil$ and $\lfloor\frac{n-r}2\rfloor$, respectively (see Figure \ref{fig.1}).
As consequences of Theorem \ref{thm1.2},
we derive the extremal graphs for the projective plane and the torus, sharpening Hong’s bounds \cite{Hong1}.

\begin{figure}[H]
\definecolor{ExtInk}{HTML}{263443}
\definecolor{ExtBase}{HTML}{548579}
\definecolor{ExtRed}{HTML}{AD6858}
\definecolor{ExtBlue}{HTML}{476F96}
\tikzset{
  ext base edge/.style={draw=ExtBase,line width=0.34mm},
  ext red edge/.style={draw=ExtRed!85!white,line width=0.20mm},
  ext blue edge/.style={draw=ExtBlue!85!white,line width=0.20mm},
  ext joining edge/.style={draw=ExtRed,line width=0.28mm},
  ext vertex/.style={circle,ball color=ExtBase,draw=none,
    inner sep=0pt,minimum size=1.7mm},
  ext red vertex/.style={ext vertex,ball color=ExtRed,minimum size=2.05mm},
  ext blue vertex/.style={ext vertex,ball color=ExtBlue,minimum size=2.05mm},
  ext dot/.style={circle,fill=ExtBase,draw=none,
    inner sep=0pt,minimum size=0.45mm}}
\centering
\begin{tikzpicture}[scale=0.7,x=1.25mm,y=1mm,
  inner xsep=0pt,inner ysep=0pt,outer xsep=0pt,outer ysep=0pt,
  line cap=round,line join=round,
  every label/.append style={text=ExtInk}]
\coordinate (s3) at (0,10);
\coordinate (s2) at (0,-10);
\coordinate (v1) at (10,0);
\coordinate (u1) at (-10,0);
\coordinate (u2) at (-25,0);
\coordinate (u3) at (-45,0);
\coordinate (u4) at (-60,0);
\coordinate (v2) at (25,0);
\coordinate (v3) at (45,0);
\coordinate (v4) at (60,0);
\coordinate (w1) at (-30,30);
\coordinate (w2) at (30,30);
\foreach \p in {u1,u2,u3,u4,v1,v2,v3,v4,s2,s3}
  \draw[ext red edge] (w1)--(\p);
\foreach \p in {u1,u2,u3,u4,v1,v2,v3,v4,s2,s3}
  \draw[ext blue edge] (w2)--(\p);
\draw[ext joining edge] (w1)--(w2);
\draw[ext base edge] (v1)--(s2);
\draw[ext base edge] (v1)--(s3);
\draw[ext base edge] (v1)--(u1);
\draw[ext base edge] (s2)--(s3);
\draw[ext base edge] (s2)--(u1);
\draw[ext base edge] (s3)--(u1);
\draw[ext base edge] (u1)--(u2);
\draw[ext base edge] (u3)--(u4);
\draw[ext base edge] (v1)--(v2);
\draw[ext base edge] (v3)--(v4);
\node[ext vertex] at (s3) {};
\node[ext vertex] at (s2) {};
\node[ext vertex,label={[label distance=1mm]270:$v_1$}] at (v1) {};
\node[ext vertex,label={[label distance=1mm]270:$u_1$}] at (u1) {};
\node[ext vertex,label={[label distance=1mm]270:$u_2$}] at (u2) {};
\foreach \x in {-31,-35,-39}
  \node[ext dot] at (\x,0) {};
\node[ext vertex] at (u3) {};
\node[ext vertex,
  label={[label distance=1mm]270:$u_{\lfloor\frac{n\!-\!4}2\rfloor}$}] at (u4) {};
\node[ext vertex,label={[label distance=1mm]270:$v_2$}] at (v2) {};
\foreach \x in {31,35,39}
  \node[ext dot] at (\x,0) {};
\node[ext vertex] at (v3) {};
\node[ext vertex,
  label={[label distance=1mm]270:$v_{\lceil\frac{n\!-\!4}{2}\rceil}$}] at (v4) {};
\node[ext red vertex,
  label={[label distance=1mm]90:$w_1$}] at (w1) {};
\node[ext blue vertex,
  label={[label distance=1mm]90:$w_2$}] at (w2) {};
\end{tikzpicture}
\caption{The extremal graph $K_2\nabla K_4^{n\!-\!2}$ embeddable on the projective plane.}
\label{fig.1}
\end{figure}

\begin{cor}\label{cor1.2}
For sufficiently large $n$, any graph $G$ of order $n$ embedded on the projective plane satisfies
$\rho(G)\leq \rho(K_2\nabla K_4^{n\!-\!2})$,
with equality if and only if $G\cong K_2\nabla K_4^{n\!-\!2}$.
\end{cor}

\begin{cor}\label{cor1.3}
For sufficiently large $n$, any graph $G$ of order $n$ embedded on the torus satisfies $\rho(G)\leq \rho(K_2\nabla K_5^{n\!-\!2})$,
with equality if and only if $G\cong K_2\nabla K_5^{n\!-\!2}$.
\end{cor}

Our approach combines extremal constructions of cellular embeddings,
topological edge-switching techniques, and walk-counting arguments on large substructures.

\section{Preliminaries}

A \emph{surface} is a compact connected 2-dimensional manifold,
possibly with boundary.
Besides the plane and the sphere, typical examples include the cylinder, the M\"{o}bius band, and the torus.
A surface is called \emph{closed} if it is compact and has no boundary.
Note that the M\"{o}bius band has a boundary homeomorphic to a circle and is therefore not closed.

Let $S'$ be a sphere with two disjoint discs $D_1$ and $D_2$.
Adding a handle to $S'$ consists of removing the interiors of $D_1$ and $D_2$ and identifying their boundary components.
More generally, a surface obtained by adding $k$ pairwise disjoint handles to a sphere is denoted by $S_k$, and $k$ is called its \emph{genus}.
In particular, the torus is homeomorphic to $S_1$.
Every closed orientable surface is homeomorphic to $S_k$ for some $k\ge 0$.

Similarly, let $N_k$ denote the surface obtained from a sphere by adding $k$ cross-caps, where $k$ is called the \emph{cross-cap number}.
The projective plane is homeomorphic to $N_1$, and every closed nonorientable surface is homeomorphic to $N_k$ for some $k\ge 1$.
The classification theorem states that every closed surface is homeomorphic to either $S_k$ or $N_k$ for some integer $k$.
Moreover, for $k\geq0$ and $\ell\geq1$, a surface obtained by adding $k$ handles and $\ell$ cross-caps to a sphere is homeomorphic to $N_{2k+\ell}$.

An embedding $\widetilde{G}$ of a graph $G$ on a surface $\Sigma$ is \emph{cellular} if every component of $\Sigma \setminus \widetilde{G}$ is homeomorphic to an open disc.
These components are called \emph{faces}, and their number is denoted by $f(\widetilde{G})$.
An embedding is \emph{$k$-representative} if every non-contractible closed curve on the surface intersects $\widetilde{G}$ at least $k$ times.
A cellular embedding is equivalent to a connected embedding that is either on the sphere or is 1-representative.

Since most results on graph embeddings concern cellular embeddings, polygonal representations are often used to visualize embeddings on low-genus surfaces.
For the torus $S_1$, one may use a rectangle with opposite sides identified; identifying one pair yields a cylinder, and identifying the remaining pair yields the torus.
Figure \ref{fig.2} ($a$) shows a cellular embedding of $K_7$ on the torus.
Similarly, nonorientable surfaces can be represented by polygons with appropriate edge identifications.
For example, the projective plane can be represented by a disc with antipodal boundary identification.
Figure \ref{fig.2} ($b$) shows a cellular embedding of $K_6$ on the projective plane.

\begin{figure}[H]
\definecolor{SurfaceInk}{HTML}{263443}
\definecolor{SurfaceBoundary}{HTML}{AC6859}
\definecolor{SurfaceEdge}{HTML}{548579}
\definecolor{SurfaceVertex}{HTML}{436E90}
\definecolor{SurfacePaper}{HTML}{F8FAFC}
\definecolor{SurfaceIce}{HTML}{DFEAF3}
\definecolor{SurfaceShade}{HTML}{B9CFDE}
\definecolor{SurfaceMint}{HTML}{CEE1DA}
\definecolor{SurfaceWarm}{HTML}{EFE1D8}
\tikzset{
  surface edge/.style={draw=SurfaceEdge,line width=0.45mm},
  surface boundary/.style={draw=SurfaceBoundary,line width=0.45mm},
  surface vertex/.style={circle,draw=none,inner sep=0pt,
    minimum size=2.05mm,ball color=SurfaceVertex},
  surface boundary vertex/.style={surface vertex,ball color=SurfaceBoundary},
  surface dot/.style={circle,fill=SurfaceEdge,draw=none,
    inner sep=0pt,minimum size=0.95mm},
  surface label/.style={text=SurfaceInk,anchor=base west,
    font=\fontsize{14.23}{17.07}\selectfont},
  surface facet left/.style={left color=SurfacePaper,
    right color=SurfaceShade,shading angle=15},
  surface facet right/.style={left color=SurfaceShade,
    right color=SurfacePaper,shading angle=-15},
  surface facet mint/.style={top color=SurfacePaper,
    bottom color=SurfaceMint,shading angle=20},
  surface facet ice/.style={top color=SurfaceIce,
    bottom color=SurfacePaper,shading angle=-20}}
\centering
\begin{tikzpicture}[x=1mm,y=1mm,
  inner xsep=0pt,inner ysep=0pt,outer xsep=0pt,outer ysep=0pt,
  line cap=round,line join=round]
\path[line width=0mm] (32.58,52.90) rectangle +(140.42,44.89);
\coordinate (A0) at (40,90);
\coordinate (A1) at (50,90);
\coordinate (A2) at (60,90);
\coordinate (A3) at (70,90);
\coordinate (A4) at (70,80);
\coordinate (A5) at (70,70);
\coordinate (A6) at (70,60);
\coordinate (A7) at (60,60);
\coordinate (A8) at (50,60);
\coordinate (A9) at (40,60);
\coordinate (A10) at (40,70);
\coordinate (A11) at (40,80);
\coordinate (AX) at (50,75);
\coordinate (AY) at (60,75);
\shade[surface facet left] (A0)--(A11)--(AX)--cycle;
\shade[surface facet mint] (A0)--(AX)--(A1)--cycle;
\shade[surface facet right] (A1)--(AX)--(AY)--cycle;
\shade[surface facet ice] (A1)--(AY)--(A4)--cycle;
\shade[surface facet mint] (A1)--(A4)--(A2)--cycle;
\shade[surface facet right] (A2)--(A4)--(A3)--cycle;
\shade[surface facet left] (A4)--(AY)--(A5)--cycle;
\shade[surface facet mint] (A5)--(AY)--(A6)--cycle;
\shade[surface facet right] (A6)--(AY)--(A7)--cycle;
\shade[surface facet mint] (A7)--(AY)--(AX)--cycle;
\shade[surface facet left] (A7)--(AX)--(A10)--cycle;
\shade[surface facet ice] (A7)--(A10)--(A8)--cycle;
\shade[surface facet mint] (A8)--(A10)--(A9)--cycle;
\shade[surface facet ice] (A10)--(AX)--(A11)--cycle;
\draw[surface edge] (A0)--(AX);
\draw[surface edge] (A1)--(AX);
\draw[surface edge] (A11)--(AX);
\draw[surface boundary] (A0)--(A9);
\draw[surface boundary] (A0)--(A3);
\draw[surface boundary] (A6)--(A3);
\draw[surface boundary] (A6)--(A9);
\draw[surface edge] (A10)--(AX);
\draw[surface edge] (AY)--(AX);
\draw[surface edge] (A1)--(AY);
\draw[surface edge] (A1)--(A4);
\draw[surface edge] (AY)--(A4);
\draw[surface edge] (A2)--(A4);
\draw[surface edge] (A10)--(A8);
\draw[surface edge] (A10)--(A7);
\draw[surface edge] (A7)--(AX);
\draw[surface edge] (AY)--(A7);
\draw[surface edge] (AY)--(A6);
\draw[surface edge] (AY)--(A5);
\foreach \p in {A0,A1,A2,A3,A4,A5,A6,A7,A8,A9,A10,A11}
  \node[surface boundary vertex] at (\p) {};
\foreach \p in {AX,AY}
  \node[surface vertex] at (\p) {};
\node[surface label] at (35.5,92.3) {$v_1$};
\node[surface label] at (35.5,55.98) {$v_1$};
\node[surface label] at (47.8,92.3) {$v_2$};
\node[surface label] at (47.8,55.98) {$v_2$};
\node[surface label] at (57.8,92.3) {$v_3$};
\node[surface label] at (57.8,55.98) {$v_3$};
\node[surface label] at (70.47,55.98) {$v_1$};
\node[surface label] at (70.47,92.3) {$v_1$};
\node[surface label] at (34.5,79.5) {$v_4$};
\node[surface label] at (34.5,69) {$v_5$};
\node[surface label] at (71,79.5) {$v_4$};
\node[surface label] at (71,69) {$v_5$};
\node[surface label] at (47.5,70.98) {$v_6$};
\node[surface label] at (56,71.5) {$v_7$};
\node[surface label] at (52,48) {$(a)$};
\shade[inner color=SurfacePaper,outer color=SurfaceWarm]
  (103.95,75.00) circle (15mm);
\shade[surface facet mint]
  (92.53,84.57)--(98.75,72.00)--(103.95,81.00)--cycle;
\shade[surface facet right]
  (115.72,84.57)--(103.95,81.00)--(109.20,72.00)--cycle;
\shade[surface facet mint]
  (98.75,72.00)--(103.95,60.00)--(109.10,72.00)--cycle;
\shade[top color=SurfaceIce,bottom color=SurfaceShade,shading angle=15]
  (98.75,72.00)--(109.20,72.00)--(103.95,81.00)--cycle;
\draw[surface edge] (103.95,90.00)--(103.95,81.00);
\draw[surface edge] (92.53,84.57)--(103.95,81.00);
\draw[surface edge] (115.72,84.57)--(103.95,81.00);
\draw[surface edge] (98.75,72.00)--(92.53,84.57);
\draw[surface edge] (98.75,72.00)--(92.53,65.25);
\draw[surface edge] (98.75,72.00)--(103.95,60.00);
\draw[surface edge] (103.95,60.00)--(109.10,72.00);
\draw[surface edge] (115.72,65.25)--(109.10,72.00);
\draw[surface edge] (115.72,84.57)--(109.10,72.00);
\draw[surface edge] (98.75,72.00)--(109.10,72.00);
\draw[surface edge] (98.75,72.00)--(103.95,81.00);
\draw[surface edge] (109.20,72.00)--(103.95,81.00);
\draw[surface boundary] (103.95,75.00) circle (15mm);
\node[surface vertex] at (98.75,72.00) {};
\node[surface vertex] at (109.20,72.00) {};
\node[surface vertex] at (103.95,81.00) {};
\node[surface boundary vertex] at (103.95,60.00) {};
\node[surface boundary vertex] at (103.95,90.00) {};
\node[surface boundary vertex] at (92.53,84.57) {};
\node[surface boundary vertex] at (115.72,84.57) {};
\node[surface boundary vertex] at (92.53,65.25) {};
\node[surface boundary vertex] at (115.72,65.25) {};
\node[surface label] at (102.3,76.00) {$v_1$};
\node[surface label] at (93,71.00) {$v_2$};
\node[surface label] at (111,71.00) {$v_3$};
\node[surface label] at (102,92.00) {$v_4$};
\node[surface label] at (102,56.00) {$v_4$};
\node[surface label] at (87.53,85.57) {$v_5$};
\node[surface label] at (87.53,63.25) {$v_6$};
\node[surface label] at (116.42,85.57) {$v_6$};
\node[surface label] at (116.72,63.25) {$v_5$};
\node[surface label] at (101.5,48) {$(b)$};
\shade[surface facet left]
  (140,60)--(150,90)--(150,85)--cycle;
\shade[surface facet right]
  (160,60)--(150,90)--(150,85)--cycle;
\shade[left color=SurfacePaper,right color=SurfaceIce,shading angle=10]
  (140,60)--(150,85)--(160,60)--(150,70)--cycle;
\shade[surface facet mint]
  (140,60)--(150,70)--(150,65)--cycle;
\shade[surface facet right]
  (160,60)--(150,70)--(150,65)--cycle;
\shade[top color=SurfaceIce,bottom color=SurfacePaper]
  (140,60)--(150,65)--(160,60)--cycle;
\draw[surface edge] (150.00,90.00)--(140.00,60.00);
\draw[surface edge] (150.00,90.00)--(160.00,60.00);
\draw[surface edge] (140.00,60.00)--(160.00,60.00);
\draw[surface edge] (150.00,85.00)--(140.00,60.00);
\draw[surface edge] (150.00,85.00)--(160.00,60.00);
\draw[surface edge] (150.00,70.00)--(140.00,60.00);
\draw[surface edge] (150.00,65.00)--(140.00,60.00);
\draw[surface edge] (150.00,65.00)--(160.00,60.00);
\draw[surface edge] (150.00,70.00)--(160.00,60.00);
\draw[surface edge] (150.00,90.00)--(150.00,85.00);
\draw[surface edge] (150.00,70.00)--(150.00,65.00);
\node[surface vertex] at (160.00,60.00) {};
\node[surface vertex] at (140.00,60.00) {};
\node[surface vertex] at (150.00,90.00) {};
\node[surface vertex] at (150.00,85.00) {};
\node[surface vertex] at (150.00,65.00) {};
\node[surface vertex] at (150.00,70.00) {};
\node[surface dot] at (150.00,73.75) {};
\node[surface dot] at (150.00,77.50) {};
\node[surface dot] at (150.00,81.25) {};
\node[surface label] at (138.00,56.00) {$w_1$};
\node[surface label] at (158.00,56.00) {$w_2$};
\node[surface label] at (148.00,61.00) {$w_{3}$};
\node[surface label] at (148.00,92.00) {$w_{t+2}$};
\node[surface label] at (148,48) {$(c)$};
\end{tikzpicture}
\caption{Embeddings of graphs on torus, projective plane and sphere.}
\label{fig.2}
\end{figure}

The \emph{(minimum) orientable genus} of a graph $G$, denoted by $\gamma_1(G)$, is the minimum integer $k$ such that $G$ embeds on $S_k$.
Similarly, the \emph{(minimum) nonorientable genus} $\gamma_2(G)$ is the minimum $k$ such that $G$ embeds on $N_k$.
If $G$ is planar, then $\gamma_1(G)=\gamma_2(G)=0$.
The \emph{Euler genus} of $G$ is defined as $\gamma(G)=\min\{2\gamma_1(G),\gamma_2(G)\}$.
On the other hand, the maximum orientable genus of a connected graph is the largest $k$ such that it admits a cellular embedding on $S_k$.
The maximum nonorientable genus and maximum Euler genus are defined analogously.

The Euler genus of a surface is the sum of its cross-caps plus twice the number of handles.
An embedding of $G$ on a surface with Euler genus $\gamma(G)$ is called a \emph{minimal embedding}.
Every minimal embedding of a connected graph is known to be cellular.

The genus of a graph is a fundamental parameter in topological graph theory.
A central problem in this field is to determine the minimum and maximum
orientable genus of a graph,
which together characterize all orientable surfaces on which the graph admits a cellular embedding.
This follows from the interpolation theorem of Duke \cite{DUKE}, which states that if a graph embeds on surfaces of genera $a$ and $b$, then it embeds on all intermediate genera.
An analogous result for nonorientable surfaces was proved by Stahl \cite{Stahl}.

\section{A structural theorem}

In this section, our main objective is to establish a structural theorem that serves as a key lemma for Theorem \ref{thm1.1}.
Before proceeding, we will first present some lemmas.

\begin{lem}\emph{(see \cite{Bondy2008}, Theorem 10.4, P. 247)}\label{lem3.1}
A graph $G$ is embeddable on the plane if and only if it is embeddable on the sphere.
\end{lem}

Let $e(G)$ and $f(G)$ denote, respectively, the numbers of edges and faces of $G$. The well-known Euler's formula states that
$|G|-e(G)+f(G)=2-\gamma$ for every graph $G$ cellularly embedded on a surface with Euler genus $\gamma$.
The following lemma follows from Euler's formula.

\begin{lem}\label{lem3.2}\emph{(\cite{Elli})}
Let $G$ be a pseudograph (a graph where loops and multiple
edges are allowed) cellularly embedded on a surface with Euler genus $\gamma$,
such that each face has degree at least $g$, where $g\geq3$. Then, the following inequality holds:
$$e(G)\leq \frac{g}{g-2}\big(|G|-2+\gamma\big).$$
\end{lem}

We now focus on the family $\mathbb{G}(n,\gamma)$,
where each graph is embedded (not necessarily cellularly) on a surface with Euler genus $\gamma$.
Based on Lemma \ref{lem3.2},
the following result can be readily derived; the first upper bound is also noted in \cite{WEST}.

\begin{lem}\label{lem3.3}
For $n\geq3$ and every graph $G\in \mathbb{G}(n,\gamma)$, we have
$e(G)\leq 3(n-2+\gamma).$
Furthermore, if $G$ is bipartite, then $e(G)\leq 2(n-2+\gamma).$
\end{lem}

Determining $\gamma(G)$ is a challenging task.
To date, the Euler genus has been found only for a few graphs.
The hard part of the Map Coloring Theorem establishes the value of $\gamma(K_r)$
(see \cite{Ring1}), and the Euler genus of $K_{s,t}$
was determined by Ringel \cite{Ring2,Ring3}.

\begin{lem}\emph{(\cite{Ring1,Ring2,Ring3})}\label{lem3.4}
Let $r\geq3$ and $t\geq s\geq2$. Then, we have\\
(i) $\gamma_1(K_r)=\lceil\frac1{12}(r-3)(r-4)\rceil$,
and $\gamma_2(K_r)=\lceil\frac1{6}(r-3)(r-4)\rceil$, except that
$\gamma_2(K_7)=3$;\\
(ii) $\gamma_1(K_{s,t})=\lceil\frac1{4}(s-2)(t-2)\rceil$
and $\gamma_2(K_{s,t})=\lceil\frac1{2}(s-2)(t-2)\rceil$.
\end{lem}

For each $i\geq 1$, as shown in Figure \ref{fig.3}, both $H_{i}$ and $H_{i}'$ are planar.
Thus, by Lemma \ref{lem3.1}, $H_{i}$ and $H_{i}'$ can be embedded on spheres,
which we will denote by $S^{i}$ and ${S^i}'$, respectively.
With the above lemmas in place, we now proceed to the key structural theorem,
which is based on an iterative construction.
The initial step is defined as follows:

\begin{definition}\label{def3.1}
Let $G$ be a simple graph cellularly embedded on a surface $R$, with a 3-face $xyz$.
We construct a new graph $G(x,y,z;i)$ and a new surface $R(x,y,z;i)$ as follows:

\vspace{0.5mm}
\noindent
(i) Cut out the face $xyz$ from $G$ (resp.\,$R$)
and the face $v_{i1}v_{i2}v_{i5}$ from $H_{i}$ (resp.\,$S^{i}$).
Then, identify the boundary of the face $xyz$ with that of $v_{i1}v_{i2}v_{i5}$,
where $x\!=\!v_{i1},y\!=\!v_{i2}$, and $z\!=\!v_{i5}$.
This results in a new graph $G'_{i}$ and a new surface $R'_{i}$.

\vspace{0.5mm}
\noindent
(ii) Cut out the faces $v_{i1}v_{i2}v_{i3}$ and $v_{i4}v_{i5}v_{i6}$ from $G'_{i}$ (resp.\,$R'_i$),
and the faces $v'_{i1}v'_{i2}v'_{i3}$ and $v'_{i4}v'_{i5}v'_{i6}$ from $H'_{i}$ (resp.\,${S^i}'$).
Furthermore, identify the boundary of $v_{i1}v_{i2}v_{i3}$ with
that of $v'_{i1}v'_{i2}v'_{i3}$,
and identify the boundary of $v_{i4}v_{i5}v_{i6}$ with
that of $v'_{i4}v'_{i5}v'_{i6}$,
where $v_{ij}\!=\!v'_{ij}$ for $j\in\{1,\ldots,6\}$.
This results in a graph $G(x,y,z;i)$ and a surface $R(x,y,z;i)$.
\end{definition}

\begin{figure}[H]
\definecolor{PlaneInk}{HTML}{263443}
\definecolor{PlaneRed}{HTML}{AF6555}
\definecolor{PlaneBlue}{HTML}{315D88}
\definecolor{PlaneGreen}{HTML}{608C80}
\definecolor{PlanePaper}{HTML}{F8FAFC}
\definecolor{PlaneWarm}{HTML}{EAD5C9}
\definecolor{PlaneTeal}{HTML}{C4DAD2}
\definecolor{PlaneBlueLight}{HTML}{DFEAF4}
\definecolor{PlaneBlueDark}{HTML}{ABC4DA}
\tikzset{
  plane outer edge/.style={draw=PlaneRed,line width=0.4mm},
  plane inner edge/.style={draw=PlaneBlue,line width=0.4mm},
  plane connector/.style={draw=PlaneGreen,line width=0.3mm},
  plane central connector/.style={draw=PlaneGreen,line width=0.35mm},
  plane vertex/.style={circle,draw=none,inner sep=0pt,minimum size=2mm},
  plane outer vertex/.style={plane vertex,ball color=PlaneRed},
  plane inner vertex/.style={plane vertex,ball color=PlaneBlue},
  plane central vertex/.style={plane vertex,ball color=PlaneGreen},
  plane label/.style={text=PlaneInk,
    font=\fontsize{12.38}{8.65}\selectfont},
  plane graph label/.style={text=PlaneInk,anchor=base west,
    font=\fontsize{14.23}{17.07}\selectfont}}
\centering
\begin{tikzpicture}[x=1mm,y=1mm,
  inner xsep=0pt,inner ysep=0pt,outer xsep=0pt,outer ysep=0pt,
  line cap=round,line join=round]
\coordinate (L1) at (0,0);
\coordinate (L2) at (40,0);
\coordinate (L3) at (20,34.64);
\coordinate (L4) at (15,5);
\coordinate (L5) at (25,5);
\coordinate (L6) at (20,13.66);
\coordinate (L7) at (20,23);
\shade[left color=PlanePaper,right color=PlaneWarm,shading angle=15]
  (L1)--(L3)--(L7)--cycle;
\shade[left color=PlaneWarm,right color=PlanePaper,shading angle=-15]
  (L2)--(L3)--(L7)--cycle;
\shade[top color=PlanePaper,bottom color=PlaneTeal,shading angle=-20]
  (L1)--(L4)--(L7)--cycle;
\shade[top color=PlaneTeal,bottom color=PlanePaper,shading angle=20]
  (L2)--(L5)--(L7)--cycle;
\shade[left color=PlanePaper,right color=PlaneTeal,shading angle=15]
  (L4)--(L6)--(L7)--cycle;
\shade[left color=PlaneBlueDark!65!white,right color=PlanePaper,
  shading angle=-15]
  (L5)--(L6)--(L7)--cycle;
\shade[top color=PlaneBlueLight,bottom color=PlaneBlueDark,shading angle=20]
  (L4)--(L5)--(L6)--cycle;
\shade[top color=PlanePaper,bottom color=PlaneWarm!65!white]
  (L1)--(L2)--(L5)--cycle;
\shade[top color=PlaneTeal!65!white,bottom color=PlanePaper]
  (L1)--(L4)--(L5)--cycle;
\draw[plane outer edge] (L1)--(L2);
\draw[plane outer edge] (L2)--(L3);
\draw[plane outer edge] (L1)--(L3);
\draw[plane inner edge] (L4)--(L5);
\draw[plane inner edge] (L4)--(L6);
\draw[plane inner edge] (L5)--(L6);
\draw[plane connector] (L1)--(L4);
\draw[plane connector] (L2)--(L5);
\draw[plane connector] (L1)--(L5);
\draw[plane connector] (L1)--(L7);
\draw[plane connector] (L2)--(L7);
\draw[plane central connector] (L3)--(L7);
\draw[plane connector] (L4)--(L7);
\draw[plane connector] (L5)--(L7);
\draw[plane central connector] (L6)--(L7);
\foreach \p in {L1,L2,L3}
  \node[plane outer vertex] at (\p) {};
\foreach \p in {L4,L5,L6}
  \node[plane inner vertex] at (\p) {};
\node[plane central vertex] at (L7) {};
\node[plane label] at (0,-3) {$v_{i1}$};
\node[plane label] at (40,-3) {$v_{i2}$};
\node[plane label] at (20,38) {$v_{i3}$};
\node[plane label] at (12,7) {$v_{i4}$};
\node[plane label] at (28,7) {$v_{i5}$};
\node[plane label] at (20,9) {$v_{i6}$};
\node[plane label] at (15.8,22) {$v_{i7}$};
\node[plane graph label] at (20,-10) {$H_{i}$};
\coordinate (R2) at (60,0);
\coordinate (R1) at (100,0);
\coordinate (R3) at (80,34.64);
\coordinate (R5) at (75,5);
\coordinate (R4) at (85,5);
\coordinate (R6) at (80,13.66);
\shade[top color=PlanePaper,bottom color=PlaneWarm,shading angle=-15]
  (R2)--(R1).. controls (93,9) and (85,15)..(R6)
           .. controls (91,7) and (110,2)..(R2)--cycle;
\shade[left color=PlaneWarm,right color=PlanePaper,shading angle=-15]
  (R1)--(R3)--(R6).. controls (85,15) and (93,9)..(R1)--cycle;
\shade[left color=PlanePaper,right color=PlaneWarm,shading angle=15]
  (R2)--(R3).. controls (63,1) and (70,2)..(R4)
           .. controls (75,2) and (65,1)..(R2)--cycle;
\shade[top color=PlanePaper,bottom color=PlaneTeal,shading angle=-20]
  (R3)--(R5)--(R6)--cycle;
\shade[left color=PlanePaper,right color=PlaneTeal,shading angle=15]
  (R3).. controls (63,1) and (70,2)..(R4)--(R5)--cycle;
\shade[left color=PlaneTeal,right color=PlanePaper,shading angle=-15]
  (R2).. controls (65,1) and (75,2)..(R4)--(R6)
      .. controls (91,7) and (110,2)..(R2)--cycle;
\shade[top color=PlaneBlueLight,bottom color=PlaneBlueDark,shading angle=20]
  (R4)--(R5)--(R6)--cycle;
\draw[plane outer edge] (R1)--(R2);
\draw[plane outer edge] (R2)--(R3);
\draw[plane outer edge] (R1)--(R3);
\draw[plane inner edge] (R4)--(R5);
\draw[plane inner edge] (R4)--(R6);
\draw[plane inner edge] (R5)--(R6);
\draw[plane connector] (R3)--(R5);
\draw[plane connector] (R3)--(R6);
\draw[plane connector] (R1).. controls (93,9) and (85,15)..(R6);
\draw[plane connector] (R3).. controls (63,1) and (70,2)..(R4);
\draw[plane connector] (R2).. controls (110,2) and (91,7)..(R6);
\draw[plane connector] (R2).. controls (65,1) and (75,2)..(R4);
\foreach \p in {R1,R2,R3}
  \node[plane outer vertex] at (\p) {};
\foreach \p in {R4,R5,R6}
  \node[plane inner vertex] at (\p) {};
\node[plane label] at (60,-4) {$v'_{i2}$};
\node[plane label] at (100,-4) {$v'_{i1}$};
\node[plane label] at (80,38) {$v'_{i3}$};
\node[plane label] at (72,8) {$v'_{i5}$};
\node[plane label] at (87.5,7) {$v'_{i4}$};
\node[plane label] at (80,9.5) {$v'_{i6}$};
\node[plane graph label] at (80,-10) {$H'_{i}$};
\end{tikzpicture}
\caption{The plane graphs $H_{i}$ and $H'_{i}$, where $i\geq 1$. }
\label{fig.3}
\end{figure}

Since the face $v_{i1}v_{i2}v_{i5}$ can be placed as the outer face of $H_i$,
it is clear that $R'_i$ is homeomorphic to $R$.
Furthermore, we observe that $R(x,y,z;i)$ is obtained from $R'_i$ by adding a handle,
and the subgraph of $G(x,y,z;i)$ induced by $\{v_{ij}~|~j=1,\ldots,7\}$
is isomorphic to $K_7$.
Additionally, the following observation can be easily made.

\begin{obs}\label{obs3.1}
Let $G$ be defined as in Definition \ref{def3.1}. Then:\\
(i) $G(x,y,z;i)$ is a simple graph and cellularly embeds on the surface $R(x,y,z;i).$ \\
(ii) $V(G(x,y,z;i))=V(G)\cup\{v_{i3},v_{i4},v_{i6},v_{i7}\}$ and $e(G(x,y,z;i))=e(G)+18$.\\
(iii) $G(x,y,z;i)$ has a subgraph $K_2\nabla P^{i}$, where $V(K_2)=\{v_{i1},v_{i2}\}$ and $P^{i}=v_{i5}v_{i4}v_{i7}v_{i3}v_{i6}$.\\
(iv) If $R$ is homeomorphic to $S_k$, then $R(x,y,z;i)$ is homeomorphic to $S_{k+1}$;
if $R$ is homeomorphic to $N_k$, then $R(x,y,z;i)$ is homeomorphic to $N_{k+2}$.
\end{obs}

By Lemma \ref{lem3.3}, every graph $G\in \mathbb{G}(n,\gamma)$ satisfies
$e(G)\leq 3(n-2+\gamma).$
Building on Definition \ref{def3.1} and Observation \ref{obs3.1}, we now present the key structural theorem.

\begin{thm}\label{thm3.1}
Let $n\geq 2\gamma+4$.
Then, there always exists a graph $G^*\in EX(n,\gamma)$
such that $G^*$ is obtained from $K_2\nabla P_{n-2}$ by adding $3\gamma$ edges.
\end{thm}

\begin{proof}
When $\gamma=0$,
the surface is the plane,
and clearly,  $K_2 \nabla P_{n-2}$ itself is a planar graph with $3n-6$ edges.
Next, we consider the case $\gamma\geq1$.

Let $G_0$ be a simple graph that is cellularly embedded on a surface $R$ and has a 3-face $w_1w_2w_3$.
Set $G_1=G_0(w_1,w_2,w_3;1)$,
and recursively define $G_{i+1}=G_{i}(v_{i1},v_{i2},v_{i6};i+1)$ for $i\geq 1$.
According to Definition \ref{def3.1}, we observe that $w_3=v_{15}$ in $G_1$,
and $v_{i6}=v_{(i+1)5}$ in each $G_{i+1}$.
Therefore, by recursively applying Observation \ref{obs3.1},
we can directly obtain the following statements:
\vspace{1mm}

(i) $G_i$ is a simple graph;

(ii) $|G_i|=|G_{0}|+4i$ and $e(G_i)=e(G_{0})+18i$;

(iii) $G_i$ has a subgraph $K_2\nabla P^{i}$, where $V(K_2)=\{v_{11},v_{12}\}$ and
$P^{i}$ is a path of length $4i$;

(iv) if $R$ is homeomorphic to $S_k$, then $G_i$ is cellularly embedded on $S_{k+i}$;
and if $R$ is homeomorphic to $N_k$, then $G_i$ is cellularly embedded on $N_{k+2i}$.

\vspace{1mm}
Now we consider two cases.
The first case is when $\gamma$ is even.
Define $G_0=K_2 \nabla P^0$, where $t\geq 1$ and $P^{0}=w_{t+2}w_{t+1}\dots w_3$
(see Figure \ref{fig.2} ($c$)).
Clearly, $G_{0}$ is a planar graph of order $t+2$ and size $3t$.
Thus, $G_{0}\in \mathbb{G}(t+2,0)$.
Set $t=n-2\gamma-2$. Since $n\geq 2\gamma+4$, it follows that $t\geq 2$.
Moreover, $|G_0|=n-2\gamma$ and $e(G_0)=3(n-2\gamma-2)$.
In view of (i), $G_{\gamma/2}$ is simple.
Additionally, in view of (ii), we observe that
$|G_{\gamma/2}|=n$
and $e(G_{\gamma/2})=e(G_0)+9\gamma=3(n-2+\gamma).$
Furthermore, based on (iii) and the structure of $G_0$,
we find that $G_{\gamma/2}$ contains $K_2\nabla P_{n-2}$ as a spanning subgraph,
where $V(K_2)=\{v_{11},v_{12}\}=\{w_1,w_2\}$ and $E(P_{n-2})=E(P^0)\cup E(P^{\gamma/2})$.
Finally, based on (iv),
$G_{\gamma/2}$ can be cellularly embedded on $S_{\gamma/2}$.
Thus, $G_{\gamma/2}\in\mathbb{G}(n,\gamma)$.
Therefore, $G_{\gamma/2}$ is a desired graph.

Next, we consider the case when $\gamma$ is odd.
Let $G_0$ be the graph obtained by embedding the entire graph $K_2 \nabla P^0$
into the 3-face $v_1v_2v_3$ of $K_6$,
such that $w_{1}=v_{1}$, $w_{2}=v_{2}$ and $w_{t+2}=v_{3}$
(see $K_2 \nabla P^0$ in Figure \ref{fig.2} ($c$),
and the embedding of $K_6$ on the projective plane in Figure \ref{fig.2} ($b$)).
Clearly,
$G_0$ can be cellularly embedded on $N_{1}$. More precisely,
we have $G_0\in \mathbb{G}(t+5,1)$ and $e(G_0)=3(t+4)$.
Set $t=n-2\gamma-3$.
Since $n\geq 2\gamma+4$, it follows that $t\geq 1$.
Furthermore, $G_0\in \mathbb{G}(n-2\gamma+2,1)$ and $e(G_0)=3(n-2\gamma+1)$.

If $\gamma=1$, then $t=n-5$. Hence,
$|G_0|=n$ and $e(G_0)=3(n-1)$.
Moreover, $G_0$ has a spanning subgraph $K_2\nabla P_{n-2}$,
where $V(K_2)=\{v_1,v_2\}=\{w_1,w_2\}$ and
$E(P_{n-2})=\{v_6v_5,v_5v_4,v_4v_3\}\cup E(P^0)$,
with $v_3=w_{t+2}$.
Therefore, $G_0$ is a desired graph.

Hence, in the remainder of the odd case, we may assume that
$\gamma\geq 3$. In particular, $(\gamma-1)/2\geq1$.
According to statements (i)--(iv),
we similarly obtain that
$G_{(\gamma-1)/2}\in\mathbb{G}(n,\gamma)$,
$e(G_{(\gamma-1)/2})=3(n-2+\gamma)$,
and that $G_{(\gamma-1)/2}$ contains
$K_2\nabla P_{n-2}$ as a spanning subgraph,
where $V(K_2)=\{v_{11},v_{12}\}$ and
$E(P_{n-2})=\{v_6v_5,v_5v_4,v_4v_3\}\cup E(P^0)\cup E(P^{(\gamma-1)/2})$,
with $v_3=w_{t+2}$ and $w_3=v_{15}$.
Therefore, $G_{(\gamma-1)/2}$ is a desired graph.

This completes the proof of Theorem~\ref{thm3.1}.
\end{proof}

\section{Proof of Theorem \ref{thm1.1}}

We will first introduce several notations.
Given two disjoint subsets $S,T\subseteq V(G)$ and a vertex $v\in V(G)$,
define $N_S(v)$ as the set of neighbors of $v$ in $S$, and let $d_S(v)$
be the degree of $v$ in $S$.
Let $G[S]$ denote the subgraph of $G$ induced by $S$,
and $G[S,T]$ be the bipartite subgraph of $G[S\cup T]$, obtained
by removing all the edges within $S$ and within $T$.
Define $e(S)$ and $e(S,T)$ as the numbers of edges in
$G[S]$ and $G[S,T]$, respectively.

Let $G$ be a graph in $SPEX(n,\gamma)$ with spectral radius $\rho(G)=\rho$.
Then, $G$ is connected. Otherwise, we could add edges between two connected components
such that the resulting graph still embeds on a surface with Euler genus $\gamma$.
However, this would increase the spectral radius,
contradicting the maximality of $\rho(G)$.
Since $G$ is connected, by the well-known Perron-Frobenius theorem, there exists a positive eigenvector
$\mathbf{x}=(x_1,\ldots,x_n)^\mathrm{T}$ corresponding to $\rho(G)$, which is called the \emph{Perron vector} of $G$.
We may assume that $u^*\in V(G)$ with $x_{u^{*}}=\max_{1\leq i\leq n}x_i=1$.

Given any $S\subseteq V(G)$, if $|S|\in \{1,2\}$, then $e(S)\leq 1$.
If $|S|\geq 3$, then Lemma \ref{lem3.3} gives $e(S)\leq 3(|S|-2+\gamma)$.
In both cases, we obtain that
\begin{align}\label{eq1}
e(S)\leq 3(|S|+\gamma).
\end{align}

Since $G\in \mathbb{G}(n,\gamma)$, it is clear that $G[S,T]\in \mathbb{G}(n,\gamma)$
for any two disjoint non-empty vertex subsets $S$ and $T$.
If $|S|+|T|\leq 2$, then $e(S,T)\leq 1$.
If $|S|+|T|\geq 3$,
it follows from Lemma \ref{lem3.3} that $e(S,T)\leq 2(|S|+|T|-2+\gamma).$
In both cases, we get that
\begin{align}\label{eq2}
e(S,T)\leq 2(|S|+|T|+\gamma).
\end{align}

Define a parameter as follows:
\begin{eqnarray}\label{eq3}
\alpha=\frac{1}{300+180\gamma+24\gamma^2}.
\end{eqnarray}
According to the condition of Theorem \ref{thm1.1}, we have $n\geq\frac{50}{\alpha^2}$.
Let $H=K_2 \nabla (n-2)K_1$.
Observe that $H$ is a planar graph.
Hence, $H\in \mathbb{G}(n,0)$, and thus $H\in\mathbb{G}(n,\gamma)$ for any non-negative integer $\gamma$.
Since $G\in SPEX(n,\gamma)$,
one can easily calculate that
\begin{eqnarray}\label{eq4}
\rho\!\geq\!\rho(H)\!=\!\frac12\big(1\!+\!\sqrt{8n-15}\big)\!>\!\sqrt{2n\!+\!8}\!>\!\frac{10}{\alpha}.
\end{eqnarray}
Next, we partition $V(G)$ into two subsets $L$ and $\overline{L}$,
where $L=\{v\in V(G)~|~x_v\geq\alpha\}$.
We will now present a series of claims regarding $L$ and $G$.

\begin{claim}\label{cl4.1}
$\gamma<10^{-3}\alpha n$ and $|L|<0.55\alpha n$.
\end{claim}

\begin{proof}
Since $n\geq\frac{50}{\alpha^2}$, based on \eqref{eq3}, it follows that
$\alpha n\geq \frac{50}{\alpha}>10^3\gamma,$
which implies that $\gamma<10^{-3}\alpha n$.
Next,
we will prove that $|L|<0.55\alpha n$.

For any $u\in L$, we know that $\rho x_u\geq \rho\alpha$.
Based on \eqref{eq4}, we further have $\rho x_u\geq \rho\alpha\geq10.$
On the other hand, by the assumption that $\max_{v\in V(G)}x_v=1$,
we obtain that
$\rho x_u\!=\!\sum_{v\in N_G(u)}x_v\leq d_L(u)\!+\!d_{\overline{L}}(u)\alpha.$
Summing this inequality over all vertices $u\in L$ gives:
\begin{align}\label{eq5}
10|L|\leq\sum_{u\in L}\rho x_u\leq\sum_{u\in L}d_L(u)+\sum_{u\in L}d_{\overline{L}}(u)\alpha
\leq 2e(L)+\alpha\cdot e(L,\overline{L}).
\end{align}
From \eqref{eq1} and \eqref{eq2}, we know that $2e(L)\leq6(|L|+\gamma)$ and
$\alpha\cdot e(L,\overline{L})\leq2\alpha(n+\gamma)$.
Based on \eqref{eq3}, we further have
$6\gamma+2\alpha\gamma\leq 10\gamma<0.2\alpha n.$
Combining \eqref{eq5}, we obtain:
\begin{align*}
10|L|\leq 6(|L|+\gamma)+2\alpha(n+\gamma)<6|L|+2.2\alpha n,
\end{align*}
which implies that $|L|<0.55\alpha n$, as claimed.
\end{proof}

\begin{claim}\label{cl4.2}
For each $u\in L$, we have $d_G(u)>(x_u-5.5\alpha)n$.
\end{claim}

\begin{proof}
Recall that $N_L(u)=N_G(u)\cap L$ and $N_{\overline{L}}(u)=N_G(u)\setminus L$.
Based on inequalities \eqref{eq1} and \eqref{eq2},
we obtain
$\sum_{v\in L}d_{N_L(u)}(v)\leq2e(L)\leq 6(|L|+\gamma)$
and $\sum_{v\in L}d_{N_{\overline{L}}(u)}(v)\leq
e(N_{\overline{L}}(u),L)$
$\leq 2(d_G(u)+|L|+\gamma),$
where $|L|+\gamma<0.6\alpha n$ follows from Claim \ref{cl4.1}.
Combining these two inequalities, we obtain
\begin{eqnarray}\label{eq6}
\sum_{v\in L}d_{N_G(u)}(v)x_v\leq\!\sum_{v\in L}d_{N_L(u)}(v)\!+\!\sum_{v\in L}d_{N_{\overline{L}}(u)}(v)
<2d_G(u)+4.8\alpha n.
\end{eqnarray}

On the other hand, in view of \eqref{eq1}, we have $e(G)\leq3(n+\gamma)$.
Note that $x_v<\alpha$ for any $v\in\overline{L}$, and by Claim \ref{cl4.1},
we see that $6\alpha \gamma\leq \gamma<10^{-3}\alpha n$.
Thus, we have
\begin{eqnarray}\label{eq7}
\sum_{v\in \overline{L}}\!d_{N_G(u)}(v)x_v
\leq\!\!\sum_{v\in V(G)}\!\!\!\!d_{G}(v)\alpha
\leq 2\alpha\cdot e(G)
\leq 6\alpha(n+\gamma)<6.2\alpha n.
\end{eqnarray}
As we know, $\rho^2x_u=\sum_{v\in V(G)}d_{N_G(u)}(v)x_v,$
and by \eqref{eq4}, we have $\rho^2>2n$.
Thus,
\begin{align}\label{eq8}
2nx_u<\rho^2x_u=\!\sum_{v\in L}d_{N_G(u)}(v)x_v
\!+\!\sum_{v\in \overline{L}}d_{N_G(u)}(v)x_v.
\end{align}
Combining \eqref{eq8} with inequalities \eqref{eq6} and \eqref{eq7}, we obtain
$2nx_u\!<\!2d_G(u)\!+\!11\alpha n$.
Therefore, we have $d_G(u)>(x_u\!-\!5.5\alpha)n$, as claimed.
\end{proof}

\begin{claim}\label{cl4.3}
Let $u^{**}\in V(G)$ such that $x_{u^{**}}=\max_{u\in V(G)\setminus \{u^*\}}x_u$.
Then $x_{u^{**}}>1\!-\!11\alpha$.
\end{claim}

\begin{proof}
Recall that $x_{u^*}=1>\alpha$. Thus, $u^*\in L$.
In view of \eqref{eq6},
we obtain $\sum_{v\in L}d_{N_G(u^*)}(v)\leq 2d_G(u^*)+4.8\alpha n$.
Notice that $d_{N_G(u^*)}(u^*)=d_G(u^*)<n$. It follows that
\begin{align*}
\sum\limits_{v\in L\setminus\{u^*\}}\!\!\!\!d_{N_G(u^*)}(v)
=\sum_{v\in L}d_{N_G(u^*)}(v)\!-\!d_{N_G(u^*)}(u^*)\leq (1\!+\!4.8\alpha)n,
\end{align*}
which further implies that
\begin{align}\label{eq9}
\sum\limits_{v\in L}\!\!d_{N_G(u^*)}\!(v)x_v
\leq d_{N_G(u^*)}\!(u^*)\!+\!\!\!\!\!\sum\limits_{v\in L\setminus\{u^*\}}\!\!\!\!\!d_{N_G(u^*)}\!(v)x_{u^{**}}
\leq n\!+\!(1\!+\!4.8\alpha)nx_{u^{**}}.
\end{align}
From \eqref{eq8}, we know that $2nx_{u^*}<\sum_{v\in L}d_{N_G(u^*)}(v)x_v
+\sum_{v\in \overline{L}}d_{N_G(u^*)}(v)x_v.$
By setting $u=u^*$ in \eqref{eq7} and summing it with \eqref{eq9}, we obtain:
\begin{align*}
2n=2nx_{u^*}<(1\!+\!6.2\alpha)n\!+\!(1\!+\!4.8\alpha)nx_{u^{**}},
\end{align*}
which yields that
$x_{u^{**}}\geq\frac{1-6.2\alpha}{1+4.8\alpha}>1-11\alpha$.
Therefore, the claim holds.
\end{proof}

By Claim \ref{cl4.3}, we have $x_{u^{**}}\!>\!1\!-\!11\alpha\!>\!\alpha$,
which implies that $u^{**}\!\in L$.
Denote $L^*\!=\!\{u^*\!,u^{**}\}$.
Then, $L^*\!\subseteq L$. Furthermore, by Claims \ref{cl4.2} and \ref{cl4.3}, we deduce that
$d_G(u^*)\!>\!(1\!-\!5.5\alpha)n$ and $d_G(u^{**})\!>\!(1\!-\!16.5\alpha)n$.
We now partition $V(G)$ into $L^*\cup R\cup R_1$,
where $R\!=\!N_G(u^*)\cap N_G(u^{**})$.
Then, $|R|\geq d_G(u^*)\!+\!d_G(u^{**})-n$.
It follows that
\begin{eqnarray}\label{eq10}
|R|\!>\!(1\!-\!22\alpha)n~~~\text{and}~~~|R_1|\!<\!22\alpha n.
\end{eqnarray}

By Lemma \ref{lem3.4}, we see that
$\gamma_1(K_{3,2\gamma+3})\!=\!\lceil\frac{2\gamma+1}{4}\rceil\!\geq\!\frac{\gamma+1}{2}$
and $\gamma_2(K_{3,2\gamma+3})\!=\!\lceil\frac{2\gamma+1}{2}\rceil\!=\!\gamma+1$.
It follows that $$\gamma(K_{3,2\gamma+3})\!=\!\min\big\{2\gamma_1(K_{3,2\gamma+3}),\gamma_2(K_{3,2\gamma+3})\big\}\!=\!\gamma\!+\!1.$$
This implies that $G$ is $K_{3,2\gamma+3}$-free,
since $G$ embeds on a surface with Euler genus $\gamma$.

\begin{claim}\label{cl4.4}
For each $u\in R\cup R_1$, we have $x_u\!<\!(41\!+\!8\gamma)\alpha.$
\end{claim}

\begin{proof}
Let $u$ be an arbitrary vertex in $R\cup R_1$. We first show that
\begin{eqnarray}\label{eq11}
d_{R}(u)\leq 2+2\gamma.
\end{eqnarray}
To prove this, we assume the contrary: $d_{R}(u)\geq 3+2\gamma$.
Under this assumption, the subgraph of $G$ induced by $N_R(u)\cup\{u\}\cup L^*$ contains a copy of $K_{3,3+2\gamma}$,
a contradiction.

Let $u_0\in R\cup R_1$ such that $x_{u_0}=\max_{u\in R\cup R_1}x_u$.
It suffices to show that $x_{u_0}\!<\!(41\!+\!8\gamma)\alpha.$
Note that $d_G(u)\!=\!d_{L^*\cup R\cup R_1}(u)$ and $d_{L^*}(u)\leq1$ for $u\in R_1$.
Based on \eqref{eq11}, we have
\begin{align}\label{eq12}
\sum_{u\in R_1}\!\!\rho x_u\!\leq\!\!\sum_{u\in R_1}\!\!\!\Big(x_{u^*}\!+\!\big(2\!+\!2\gamma\!+\!d_{R_1}(u)\big)x_{u_0}\!\Big)
\!=\!|R_1|\!+\!\Big(\big(2\!+2\!\gamma\big)|R_1|\!+\!2e(R_1)\!\Big)x_{u_0},
\end{align}
where $2e(R_1)\!\leq\!6(|R_1|\!+\!\gamma)$ by inequality \eqref{eq1}.
Recall that $|R_1|\!<\!22\alpha n$ by inequality \eqref{eq10} and $\gamma<10^{-3}\alpha n$ by Claim \ref{cl4.1}.
Since $x_{u_0}\leq1$, inequality \eqref{eq12} can be simplified to
$$\sum_{u\in R_1}\rho x_u\!\leq\!(9\!+\!2\gamma)|R_1|\!+\!6\gamma\!<\!\!(9\!+\!2\gamma)22.1\alpha n.$$

Observe that $N_{R_1}(u_0)\subseteq R_1$ and $L^*=\{u^*,u^{**}\}$.
Combining the above inequality with \eqref{eq11}, we deduce that
\begin{align}\label{eq13}
\rho x_{u_0}
=\!\!\!\!\!\sum_{v\in N_{L^*\cup R}(u_0)}\!\!\!\!\!x_v
+\!\!\!\!\!\sum_{v\in N_{R_1}(u_0)}\!\!\!\!\!x_v
\leq \big(4+2\gamma\big)\!+\!\frac{1}{\rho}\big(9\!+\!2\gamma\big)22.1\alpha n.
\end{align}
Dividing both sides of \eqref{eq13} by $\rho$, and noting that $\rho>\sqrt{2n}>\frac{10}{\alpha}$ based on \eqref{eq4},
we obtain:
\begin{align*}
x_{u_0}\!<\!\big(4\!+\!2\gamma\big)\frac\alpha{10}\!+\!\big(9\!+\!2\gamma\big)\frac{22.1\alpha n}{2n}
\!<\!\big(100\!+\!23\gamma\big)\alpha.
\end{align*}
In view of \eqref{eq3} and the inequality derived above, we observe that $x_{u_0}<\frac13.$
We now refine inequality \eqref{eq12} as follows:
\begin{align*}
\sum_{u\in R_1}\!\!\rho x_u\!\leq\!|R_1|\!+\!\frac13\big((8\!+\!2\gamma)|R_1|\!+\!6\gamma\big)
\!<\!\frac13(11\!+\!2\gamma)22.1\alpha n.
\end{align*}
Using the above inequality, we can further refine inequality \eqref{eq13} as follows:
\begin{align*}
\rho x_{u_0}
\!=\!\!\!\sum_{v\in N_{L^*\cup R}(u_0)}\!\!\!x_v
+\!\!\!\sum_{v\in N_{R_1}(u_0)}\!\!\!x_v
\!\leq\!(4+2\gamma)\!+\!\frac{1}{3\rho}(11\!+\!2\gamma)22.1\alpha n.
\end{align*}
Recall that $\rho>\sqrt{2n}>\frac{10}{\alpha}$.\
This simplifies to $x_{u_0}\!<\!(41\!+\!8\gamma)\alpha,$ as desired.
\end{proof}

Next, we will proceed to prove a {\bf key} structural claim.
In order to ensure the clarity and readability of the proof,
we have to transform the original graph $G$ into a series of
embedded graphs: $\widetilde{G}$, $\widetilde{G'}$, and $\widetilde{G''}$.
Each of these transformations will gradually introduce more structural information, which is essential for the subsequent proof.

Assume that the graph $G$ has Euler genus $\gamma(G)=\gamma_0.$
Since $G\in SPEX(n,\gamma)$,
it is clear that $\gamma_0\leq \gamma$.
Let $\widetilde{G}$ be a minimal embedding of $G$ on a surface $\Sigma_0$ with Euler genus $\gamma_0$.
Note that $G$ is connected. Thus, $\widetilde{G}$ is a cellular embedding.

Let $N_{\widetilde{G}}(u^*)=\{u_1,\dots,u_{d^*}\}$, where $u_1,\dots,u_{d^*}$
surround $u^*$ in clockwise order.
For each $i\in \{1,\dots,d^*\}$, let $F^*_i$ denote the face that
extends clockwise from the edge $u^*u_i$ to $u^*u_{i+1}$ in $\widetilde{G}$, where $u_{d^*+1}=u_1$.
We can construct a new embedded graph $\widetilde{G'}$ from $\widetilde{G}$ as follows (see Figure \ref{fig.4}):
\vspace{1mm}

($a$) if $u_iu_{i+1}\in E(F^*_i)$, we do nothing;

($b$) if $u_iu_{i+1}\in E(\widetilde{G})\setminus E(F^*_i)$,
we move the edge $u_iu_{i+1}$ to cross the face $F^*_i$;

($c$) if $u_iu_{i+1}\notin E(\widetilde{G})$, we add the edge $u_iu_{i+1}$ and make it cross the face $F^*_i$.

\begin{figure}[H]
\definecolor{OpInk}{HTML}{263443}
\definecolor{OpEdge}{HTML}{608682}
\definecolor{OpCurve}{HTML}{315D88}
\definecolor{OpVertex}{HTML}{42627B}
\definecolor{OpPaper}{HTML}{F6F9FC}
\definecolor{OpBlueLight}{HTML}{E4EDF5}
\definecolor{OpBlueDark}{HTML}{B1C9DE}
\definecolor{OpTealLight}{HTML}{EDF5F2}
\definecolor{OpTealDark}{HTML}{B8D1C9}
\tikzset{
  op edge/.style={draw=OpEdge,line width=0.4mm},
  op curve/.style={draw=OpCurve,line width=0.5mm},
  op vertex/.style={circle,ball color=OpVertex,draw=none,
    inner sep=0pt,minimum size=2mm},
  op label/.style={text=OpInk,
    font=\fontsize{12.38}{8.65}\selectfont},
  op face label/.style={text=OpInk,
    font=\fontsize{14.23}{17.07}\selectfont},
  op operation label/.style={text=OpInk,
    font=\fontsize{12.23}{17.07}\selectfont},
  op arrow/.style={text=OpInk!70,
    font=\fontsize{20.23}{17.07}\selectfont}}
\centering
\begin{tikzpicture}[x=1mm,y=1mm,
  inner sep=0pt,outer sep=0pt,line cap=round,line join=round]
\begin{scope}
  \coordinate (AL) at (-12,0);
  \coordinate (AR) at (12,0);
  \coordinate (AB) at (0,-18);
  \shade[left color=OpPaper,right color=OpBlueDark,shading angle=18]
    (AL).. controls (-10,14) and (10,14)..(AR)--(AB)--cycle;
  \shade[top color=OpTealLight,bottom color=OpTealDark,shading angle=15]
    (AL).. controls (-15,25) and (15,25)..(AR)
        .. controls (10,14) and (-10,14)..(AL)--cycle;
  \draw[op edge] (AL)--(AB);
  \draw[op edge] (AR)--(AB);
  \draw[op edge] (AL).. controls (-15,25) and (15,25)..(AR);
  \draw[op curve] (AL).. controls (-10,14) and (10,14)..(AR);
  \foreach \p in {AL,AR,AB}
    \node[op vertex] at (\p) {};
  \node[op label,anchor=east] at (-14.2,0) {$u_{i}$};
  \node[op label,anchor=west] at (14.2,0) {$u_{i+1}$};
  \node[op label,anchor=north] at (0,-20.5) {$u^*$};
  \node[op face label] at (0,-5) {$F^*_{i}$};
\end{scope}
\node[op operation label] at (24,-4) {$(b)$};
\node[op arrow] at (24,-10) {$\Longrightarrow$};
\begin{scope}[xshift=48mm]
  \coordinate (BL) at (-12,0);
  \coordinate (BR) at (12,0);
  \coordinate (BB) at (0,-18);
  \shade[left color=OpPaper,right color=OpBlueDark,shading angle=18]
    (BL)--(BR)--(BB)--cycle;
  \shade[top color=OpTealLight,bottom color=OpTealDark,shading angle=15]
    (BL).. controls (-10,14) and (10,14)..(BR)--cycle;
  \draw[op edge] (BL)--(BR);
  \draw[op edge] (BL)--(BB);
  \draw[op edge] (BR)--(BB);
  \draw[op curve] (BL).. controls (-10,14) and (10,14)..(BR);
  \foreach \p in {BL,BR,BB}
    \node[op vertex] at (\p) {};
  \node[op label,anchor=east] at (-14.2,0) {$u_{i}$};
  \node[op label,anchor=west] at (14.2,0) {$u_{i+1}$};
  \node[op label,anchor=north] at (0,-20.5) {$u^*$};
\end{scope}
\node[op operation label] at (72,-4) {$(c)$};
\node[op arrow] at (72,-10) {$\Longleftarrow$};
\begin{scope}[xshift=96mm]
  \coordinate (CL) at (-12,0);
  \coordinate (CR) at (12,0);
  \coordinate (CB) at (0,-18);
  \shade[left color=OpPaper,right color=OpBlueDark,shading angle=18]
    (CL).. controls (-10,14) and (10,14)..(CR)--(CB)--cycle;
  \draw[op edge] (CL)--(CB);
  \draw[op edge] (CR)--(CB);
  \draw[op curve] (CL).. controls (-10,14) and (10,14)..(CR);
  \foreach \p in {CL,CR,CB}
    \node[op vertex] at (\p) {};
  \node[op label,anchor=east] at (-14.2,0) {$u_{i}$};
  \node[op label,anchor=west] at (14.2,0) {$u_{i+1}$};
  \node[op label,anchor=north] at (0,-20.5) {$u^*$};
  \node[op face label] at (0,-5) {$F^*_{i}$};
\end{scope}
\end{tikzpicture}
\caption{The operations ($b$) and ($c$). }
\label{fig.4}
\end{figure}

\begin{claim}\label{cl4.5}
The following assertions regarding $\widetilde{G'}$ are true.\\
(i) $\widetilde{G'}$ is cellularly embedded on the surface $\Sigma_0$.\\
(ii) The underlying abstract graph of $\widetilde{G'}$ is $G$,
and $u_iu_{i+1}\in E(\widetilde{G})$ for each $i\in \{1,\dots,d^*\}$.\\
(iii) $u^*u^{**}\in E(\widetilde{G})$.
\end{claim}

\begin{proof}
($i$) Let $G'$ be the underlying abstract graph of
$\widetilde{G'}$. By the construction of $\widetilde{G'}$,
we have $G\subseteq G'$, and $G'$ is embedded on the surface $\Sigma_0$.
Hence, $\gamma(G)=\gamma_0\leq \gamma(G')\leq \gamma_0$.
Therefore, $\gamma(G')=\gamma_0$.
Since $G'$ is connected, $\widetilde{G'}$ is a minimal embedding
of $G'$ on $\Sigma_0$, and hence it is cellular.

($ii$) By construction, $G'\in\mathbb{G}(n,\gamma_0)
\subseteq\mathbb{G}(n,\gamma)$ and $G$ is a spanning subgraph of $G'$.
If $G'\neq G$, then $G$ is a proper spanning subgraph of $G'$.
Since $G$ is connected, we have $\rho(G')>\rho(G)$ by the
Perron--Frobenius theorem, contradicting $G\in SPEX(n,\gamma)$.
Hence, $G'=G$, and hence $u_iu_{i+1}\in E(G)$
for every $i\in\{1,\ldots,d^*\}$.

($iii$) Suppose, to the contrary, that $u^*u^{**}\notin E(\widetilde{G'})$.
Notice that in the graph $\widetilde{G'}$,
there are exactly $d^*$ faces $F^*_{11},F^*_{21},\dots,F^*_{d^*1}$ incident to $u^*$,
where $F^*_{i1}$ denotes the 3-face $u^*u_iu_{i+1}$.
We triangulate the remaining faces by adding edges inside them,
without adding vertices or introducing loops; parallel edges
are allowed. We leave all existing $3$-faces unchanged.
Denote the resulting cellular embedding by $\widetilde{G''}$.
Thus every face of $\widetilde{G''}$ is bounded by a $3$-cycle,
and all faces $F^*_{i1}$ remain unchanged.
Clearly, $N_{\widetilde{G''}}(u^{*})=\{u_1,\dots,u_{d^{*}}\}$
and $F^*_{i1}$ is a 3-face in $\widetilde{G''}$ for each $i\in \{1,\ldots,d^*\}$.

In the embedded graph $\widetilde{G''}$,
let $\mathcal{F}$ denote the set of all faces,
$\mathcal{F}^*$ denote the set of faces incident to $u^{*}$
(clearly, $\mathcal{F}^*=\{F^*_{i1}~|~1\leq i \leq d^*\}$),
and $\mathcal{F}^{**}$ denote the set of faces incident to $u^{**}$.
Additionally, let $\mathcal{F}'$ be the set of faces incident to at least one edge in
$\{u_iu_{i+1}~|~1\leq i \leq d^*\}$.
Clearly, $\mathcal{F}^*\subseteq\mathcal{F}'$ and $\mathcal{F}^{*}\cap \mathcal{F}^{**}=\varnothing$.
Moreover, since every face of $\widetilde{G''}$ is a 3-face, it is easy to see that
$$|\mathcal{F}^{**}|=d_{\widetilde{G''}}(u^{**})
\geq d_{\widetilde{G}}(u^{**})\geq |R|.$$
Set $\mathcal{F}^{***}=\mathcal{F}'\setminus \mathcal{F}^{*}$.
Since each face $F$ in $\mathcal{F}^{***}$ is a 3-face,
it follows that $F$ is incident to at most two edges in $\{u_iu_{i+1}~|~1\leq i \leq d^*\}$,
which implies that $|\mathcal{F}^{***}|\geq \frac{1}{2}d^*$.

We now demonstrate that $\mathcal{F}^{**}\cap \mathcal{F}^{***}\neq \varnothing$.
Assume, for the sake of contradiction, that $\mathcal{F}^{**}\cap \mathcal{F}^{***}=\varnothing$.
Then, we have
$$|\mathcal{F}|
\geq|\mathcal{F}^*|+ |\mathcal{F}^{**}|+ |\mathcal{F}^{***}|
\geq d^*+|R|+\frac{1}{2}d^*.$$
Note that $d^*=d_{\widetilde{G}}(u^{*})\geq |R|$, where $|R|>(1-22\alpha)n$ by inequality \eqref{eq10}.
Thus,
$$|\mathcal{F}|\geq \frac{5}{2}|R|>\frac{5}{2}\big(1-22\alpha\big)n>2(n-2+\gamma_0).$$
On the other hand,
since each face of $\widetilde{G''}$ is a 3-face,
by Lemma \ref{lem3.2}, we know that $e(\widetilde{G''})\leq 3(n-2+\gamma_0)$.
By the Handshake Theorem, we also have $3|\mathcal{F}|=2e(\widetilde{G''})$.
Combining this, we get that $|\mathcal{F}|\leq 2(n-2+\gamma_0)$, which leads to a contradiction
since $|\mathcal{F}|>2(n-2+\gamma_0)$.
Therefore, we conclude that
$\mathcal{F}^{**}\cap \mathcal{F}^{***}\neq \varnothing$.

Now, choose an $F''\in \mathcal{F}^{**}\cap \mathcal{F}^{***}$.
Then, there exists some $i_0\in \{1,\dots,d^*\}$ such that
the 3-face $F''$ is incident to both $u^{**}$ and the edge $u_{i_0}u_{i_0+1}$ in $\widetilde{G''}$.
By the construction of $\widetilde{G''}$,
we can see that $F''$ is a subregion of some face $F'$ of $\widetilde{G'}$
(possibly $F''=F'$).

\begin{figure}[H]
\definecolor{LocalInk}{HTML}{263443}
\definecolor{LocalEdge}{HTML}{608682}
\definecolor{LocalCurve}{HTML}{315D88}
\definecolor{LocalVertex}{HTML}{42627B}
\definecolor{LocalTeal}{HTML}{CBDDD7}
\definecolor{LocalBlue}{HTML}{D2E0EF}
\definecolor{LocalBlueSide}{HTML}{B5CCDF}
\definecolor{LocalPaper}{HTML}{F8FAFC}
\tikzset{
  local edge/.style={draw=LocalEdge,line width=0.4mm},
  local curve/.style={draw=LocalCurve,line width=0.4mm},
  local vertex/.style={circle,ball color=LocalVertex,draw=none,
    inner sep=0pt,minimum size=2mm},
  local label/.style={text=LocalInk,
    font=\fontsize{12.38}{8.65}\selectfont},
  local face label/.style={text=LocalInk,
    font=\fontsize{14.23}{17.07}\selectfont},
  local graph label/.style={text=LocalInk,
    font=\fontsize{14.23}{17.07}\selectfont},
  local arrow/.style={text=LocalInk!70,
    font=\fontsize{20.23}{17.07}\selectfont}}
\centering
\begin{tikzpicture}[x=1mm,y=1mm,
  inner sep=0pt,outer sep=0pt,line cap=round,line join=round]
\begin{scope}
  \coordinate (AL) at (-12,0);
  \coordinate (AR) at (12,0);
  \coordinate (AB) at (0,-17.5);
  \coordinate (AT) at (0,17.5);
  \shade[top color=LocalPaper,bottom color=LocalTeal,shading angle=18]
    (AL)--(AT)--(AR)--cycle;
  \shade[top color=LocalPaper,bottom color=LocalBlue,shading angle=0]
    (AL)--(AR)--(AB)--cycle;
  \shade[left color=LocalBlueSide,right color=LocalBlue!35!white]
    (AL).. controls (-14,-8) and (-7,-17.5)..(AB)--cycle;
  \shade[left color=LocalBlue!35!white,right color=LocalBlueSide]
    (AR).. controls (14,-8) and (7,-17.5)..(AB)--cycle;
  \draw[local edge] (AL)--(AR);
  \draw[local edge] (AL)--(AB);
  \draw[local edge] (AR)--(AB);
  \draw[local edge] (AL)--(AT);
  \draw[local edge] (AR)--(AT);
  \draw[local curve]
    (AL).. controls (-14,-8) and (-7,-17.5)..(AB);
  \draw[local curve]
    (AR).. controls (14,-8) and (7,-17.5)..(AB);
  \foreach \p in {AL,AR,AB,AT}
    \node[local vertex] at (\p) {};
  \node[local label,anchor=east] at (-14.2,0) {$u_{i_0}$};
  \node[local label,anchor=west] at (14.2,0) {$u_{i_0+1}$};
  \node[local label,anchor=north] at (0,-20) {$u^{**}$};
  \node[local label,anchor=south] at (0,20) {$u^{*}$};
  \node[local face label] at (0,-6) {$F''$};
  \node[local graph label] at (0,-31) {$\widetilde{G''}$};
\end{scope}
\node[local arrow] at (24,-9) {$\Longleftarrow$};
\begin{scope}[xshift=48mm]
  \coordinate (BL) at (-12,0);
  \coordinate (BR) at (12,0);
  \coordinate (BB) at (0,-17.5);
  \coordinate (BT) at (0,17.5);
  \shade[top color=LocalPaper,bottom color=LocalTeal,shading angle=18]
    (BL)--(BT)--(BR)--cycle;
  \shade[top color=LocalPaper,bottom color=LocalBlue,shading angle=0]
    (BL)--(BR)
      .. controls (14,-8) and (7,-17.5)..(BB)
      .. controls (-7,-17.5) and (-14,-8)..(BL)--cycle;
  \draw[local edge] (BL)--(BR);
  \draw[local edge] (BL)--(BT);
  \draw[local edge] (BR)--(BT);
  \draw[local curve]
    (BL).. controls (-14,-8) and (-7,-17.5)..(BB);
  \draw[local curve]
    (BR).. controls (14,-8) and (7,-17.5)..(BB);
  \foreach \p in {BL,BR,BB,BT}
    \node[local vertex] at (\p) {};
  \node[local label,anchor=east] at (-14.2,0) {$u_{i_0}$};
  \node[local label,anchor=west] at (14.2,0) {$u_{i_0+1}$};
  \node[local label,anchor=north] at (0,-20) {$u^{**}$};
  \node[local label,anchor=south] at (0,20) {$u^{*}$};
  \node[local face label] at (0,-6) {$F'$};
  \node[local graph label] at (0,-31) {$\widetilde{G'}$};
\end{scope}
\node[local arrow] at (72,-9) {$\Longrightarrow$};
\begin{scope}[xshift=96mm]
  \coordinate (CL) at (-12,0);
  \coordinate (CR) at (12,0);
  \coordinate (CB) at (0,-17.5);
  \coordinate (CT) at (0,17.5);
  \shade[left color=LocalPaper,right color=LocalTeal!90!white,
    shading angle=12]
    (CT)--(CL).. controls (-14,-8) and (-7,-17.5)..(CB)--cycle;
  \shade[left color=LocalBlueSide,right color=LocalPaper,
    shading angle=-12]
    (CT)--(CR).. controls (14,-8) and (7,-17.5)..(CB)--cycle;
  \draw[local edge] (CL)--(CT);
  \draw[local edge] (CR)--(CT);
  \draw[local edge] (CB)--(CT);
  \draw[local curve]
    (CL).. controls (-14,-8) and (-7,-17.5)..(CB);
  \draw[local curve]
    (CR).. controls (14,-8) and (7,-17.5)..(CB);
  \foreach \p in {CL,CR,CB,CT}
    \node[local vertex] at (\p) {};
  \node[local label,anchor=east] at (-14.2,0) {$u_{i_0}$};
  \node[local label,anchor=west] at (14.2,0) {$u_{i_0+1}$};
  \node[local label,anchor=north] at (0,-20) {$u^{**}$};
  \node[local label,anchor=south] at (0,20) {$u^{*}$};
  \node[local graph label] at (0,-31) {$\widetilde{G'''}$};
\end{scope}
\end{tikzpicture}
\caption{An illustration of the local structures of $\widetilde{G''}$, $\widetilde{G'}$ and $\widetilde{G'''}$. }
\label{fig.5}
\end{figure}

Let $\widetilde{G'''}$ be the embedding on $\Sigma_0$ obtained from $\widetilde{G'}$ by deleting the edge $u_{i_0}u_{i_0+1}$ and
adding the edge $u^*u^{**}$ (see Figure \ref{fig.5}).
By Claim \ref{cl4.4} and equality \eqref{eq3}, we have $x_u<(41+8\gamma)\alpha<1-11\alpha$ for each $u\in R\cup R_1$.
However, $x_{u^*}=1$, and by Claim \ref{cl4.3}, $x_{u^{**}}>1-11\alpha$.
Consequently,
$$\rho(\widetilde{G'''})\!-\!\rho(\widetilde{G'})
\geq \frac{\mathbf{x}^{\mathrm{T}}(A(\widetilde{G'''})\!-\!A(\widetilde{G'}))\mathbf{x}}{\mathbf{x}^{\mathrm{T}}\mathbf{x}}
=\frac{2}{\mathbf{x}^{\mathrm{T}}\mathbf{x}}\big(x_{u^*}x_{u^{**}}\!-\!x_{u_{i_0}}x_{u_{i_0+1}}\big)
\!>\!0,$$
which contradicts the fact that $\widetilde{G'}\in SPEX(n,\gamma)$.
Therefore, $u^*u^{**}\in E(\widetilde{G'})$.
\end{proof}

\begin{claim}\label{cl4.6}
$R_1$ is empty.
\end{claim}

\begin{proof}
Suppose to the contrary that $|R_1|\geq1$.
We claim that $\delta(\widetilde{G}[R_1])\leq 5+\gamma$.
If $|R_1|\leq6$, this is immediate. Hence, assume $|R_1|\geq7$.
By Lemma \ref{lem3.3},
$e(\widetilde{G}[R_1])\leq 3\big(|R_1|-2+\gamma\big)$.
Therefore,
\[
\delta(\widetilde{G}[R_1])
\leq
\frac{2e(\widetilde{G}[R_1])}{|R_1|}
\leq
6+\frac{6\gamma-12}{|R_1|}
<6+\gamma.
\]
Since $\delta(\widetilde{G}[R_1])$ is an integer, it follows that
$\delta(\widetilde{G}[R_1])\leq5+\gamma$.

Now, we can choose a vertex $u_1\in R_1$ with $d_{R_1}(u_1)\leq5+\gamma$.
Combining (\ref{eq11}), we have $d_{R\cup R_1}(u_1)\leq7+3\gamma$.
Noting that $d_{\{u^*,u^{**}\}}(u_1)\leq1$,
and by Claim \ref{cl4.4}, we obtain that
\begin{eqnarray}\label{eq14}
\sum_{w\in N_{\widetilde{G}}(u_1)}\!\!\!\!\!x_w\leq x_{u^*}
+\!\!\!\!\!\sum_{w\in N_{R\cup R_1}(u_1)}\!\!\!\!\!x_w
\leq 1+\big(7+3\gamma\big)\big(41+8\gamma\big)\alpha.
\end{eqnarray}

Let $\widetilde{G''''}$ be the graph obtained from $\widetilde{G}$ by deleting all edges adjacent to $u_1$,
adding the edges $u_1u^*$ and $u_1u^{**}$,
and making $u_1 u^*,u_1u^{**}$ cross the face incident to the edge $u^*u^{**}$ in $\widetilde{G}$.
Obviously, $\widetilde{G''''}$ can still be embedded on the surface $\Sigma_0$ with Euler genus $\gamma_0\leq \gamma$.
Hence, we also have $\widetilde{G''''}\in\mathbb{G}(n,\gamma)$.
In view of \eqref{eq3}, we observe that $$11\alpha+(7+3\gamma)(41+8\gamma)\alpha=(298+179\gamma+24\gamma^2)\alpha<1.$$
Combining the inequality above with Claim \ref{cl4.3}, we deduce that
$$x_{u^*}+x_{u^{**}}>2-11\alpha>1+(7+3\gamma)(41+8\gamma)\alpha.$$
It follows from \eqref{eq14} that $x_{u^*}+x_{u^{**}}>\sum_{w\in N_{\widetilde{G}}(u_1)}x_w$, and thus
$$\rho(\widetilde{G''''})\!-\!\rho(\widetilde{G})\geq
\frac{2x_{u_1}}{\mathbf{x}^{\mathrm{T}}\mathbf{x}}\Big((x_{u^*}\!+\!x_{u^{**}})
\!-\!\!\!\!\sum_{w\in N_{\widetilde{G}}(u_1)}x_w\Big)\!>\!0,
$$
which contradicts the assumption that $\widetilde{G}\in SPEX(n,\gamma)$.
Hence, $R_1$ is an empty set.
\end{proof}

In the following, we will complete the proof of Theorem \ref{thm1.1}.
To enhance clarity and readability, the proof is divided into two separate theorems.

\begin{thm}\label{thm4.1}
Let $n\geq\frac{50}{\alpha^2}$.
Then $SPEX(n,\gamma)\subseteq EX(n,\gamma)$, and every graph $G$ in $SPEX(n,\gamma)$
is obtained from $K_2\nabla P_{n-2}$ by adding $3\gamma$ edges.
\end{thm}

\begin{proof}
Let $G$ be a graph in $SPEX(n,\gamma)$.
Recall that $V(G)$ has been partitioned into $L^*\cup R\cup R_1$,
where $L^*=\{u^*,u^{**}\}$ and $R=N_G(u^*)\cap N_G(u^{**})$.
By Claim \ref{cl4.6}, $R_1$ is an empty set,
and by Claim \ref{cl4.5}, we further know that $u^*u^{**}\in E(G)$ and $G$ contains $K_2 \nabla P_{n-2}$ as a spanning subgraph.
In the following, we will show that $G\in EX(n,\gamma)$.
Based on Lemma \ref{lem3.3}, to establish that $G\in EX(n,\gamma)$, it suffices to verify that
$e(G)\!=\!3(n\!-\!2\!+\!\gamma)$.

Suppose, for the sake of contradiction, that $e(G)\!<\!3(n\!-\!2\!+\!\gamma)$.
By symmetry, we can observe that $x_{u^{**}}\!=\!x_{u^*}\!=\!1$.
Assume that $R=\{u_1,\dots,u_{n-2}\}$.
For each $u_i\in R$, we have
\begin{eqnarray}\label{eq15}
\rho x_{u_i}= x_{u^*}+x_{u^{**}}+\!\!\sum_{u\in N_R(u_i)}x_{u}.
\end{eqnarray}
In view of \eqref{eq11}, we know that $d_R(u_i)\leq2+2\gamma$
for each $u_i\in R$.
Furthermore,
by Claim \ref{cl4.4}, we obtain
$\sum_{u\in N_R(u_i)}x_{u}\!\leq\!(2\!+\!2\gamma)(41\!+\!8\gamma)\alpha.$
Combining this with \eqref{eq3} gives $\sum_{u\in N_R(u_i)}x_{u}\!<\!1$.
It follows from \eqref{eq15} that
$\rho x_{u_i}<3$ for each $u_i\in R$.

Let $\phi=e(G)-e(K_2\nabla P_{n-2})$.
Then, we also have $d_R(u_i)\leq\phi+2$ for each $u_i\in R$.
Using \eqref{eq15} again, and noting that $x_{u}<\frac{3}{\rho}$ for each $u\in N_R(u_i)$,
we obtain
\begin{eqnarray}\label{eq16}
\frac{2}{\rho}\leq x_{u_i}\!<\!\frac{2}{\rho}\!+\!\frac{3}{\rho^2}\big(\phi\!+\!2\big).
\end{eqnarray}
Since $e(G)\!<\!3(n\!-\!2\!+\!\gamma)$ and $e(K_2\nabla P_{n-2})=3n\!-\!6$,
we have $\phi<3\gamma$.
Furthermore, in light of \eqref{eq3} and \eqref{eq4}, it is evident that
\begin{eqnarray}\label{eq17}
\rho\!>\!\frac{10}{\alpha}\!=\!3000\!+\!1800\gamma\!+\!240\gamma^2\!>\!12(\phi\!+\!1)(\phi\!+\!2).
\end{eqnarray}
Based on \eqref{eq16} and \eqref{eq17},
it follows that for any two vertices $u_i,u_j\in R$,
\begin{eqnarray}\label{eq18}
\frac{4}{\rho^2}\!\leq\!x_{u_i}x_{u_{j}}\!<\!\frac{4}{\rho^2}
\!+\!\frac{12(\phi\!+\!2)}{\rho^3}\!+\!\frac{9(\phi\!+\!2)^2}{\rho^4}
\!<\!\frac{4}{\rho^2}\!+\!\frac{2}{\rho^2(\phi\!+\!1)}.
\end{eqnarray}

By Theorem \ref{thm3.1},
there exists a graph $G^*\in\mathbb{G}(n,\gamma)$,
which is obtained from $K_2\nabla P_{n-2}$ by adding $3\gamma$ edges.
We can modify $G$ by first deleting $\phi$ edges within $R$ to obtain $K_2 \nabla P_{n-2}$,
and then adding $3\gamma$ edges within $R$ to obtain $G^*$.
Based on \eqref{eq18}, we have
$$\rho(G^*)\!-\!\rho(G)
\!\geq\!\frac{\mathbf{x}^{\mathrm{T}}\big(A(G^*)\!-\!A(G)\big)\mathbf{x}}{\mathbf{x}^{\mathrm{T}}\mathbf{x}}
\!>\!\frac{2}{\mathbf{x}^{\mathrm{T}}\mathbf{x}}\Big(\frac{12\gamma}{\rho^2}\!-\!\frac{4\phi\!+\!2}{\rho^2}\Big).$$
Since $\phi<3\gamma$, we have $\phi\leq 3\gamma-1$, and thus $\rho(G^*)>\rho(G)$,
which contradicts the assumption that $G\in SPEX(n,\gamma)$.
Therefore, $e(G)=3(n-2+\gamma)$, as desired.
\end{proof}

Now, let $\mathbb{H}(n,\gamma)$ denote the family of graphs where each member is obtained from
$K_2\nabla P_{n-2}$ by adding $3\gamma$ edges.
By Theorem \ref{thm4.1}, we have
$$SPEX(n,\gamma)\subseteq \mathbb{G}(n,\gamma)\cap \mathbb{H}(n,\gamma).$$

\begin{thm}\label{thm4.2}
Let $n\geq\frac{50}{\alpha^2}$ and $\gamma$ be a positive integer.
Denote $\rho_0=\frac32\!+\!\sqrt{2n\!-\!\frac{15}4}.$
Then, for every graph $G$ in $\mathbb{G}(n,\gamma)\cap\mathbb{H}(n,\gamma)$, we have
\begin{eqnarray*}
\rho_0\!+\!\frac{3\gamma\!-\!1}n<\rho(G)<\rho_0\!+\!\frac{3\gamma\!-\!0.95}n.
\end{eqnarray*}
\end{thm}

\begin{proof}
Let $G_0=K_2\nabla C_{n-2}.$
Based on the high symmetry of the structure of $G_0$,
we can observe that $\rho(G_0)=\rho_0$ through a straightforward calculation.
By the Perron-Frobenius theorem,
there exists a positive eigenvector $\mathbf{y}=(y_1,\ldots,y_n)^\mathrm{T}$ corresponding to
the spectral radius of $G_0.$
Let $u^*$ and $u^{**}$ be the two dominating vertices in $G_0$.
By symmetry, we may assume that $y_{u^{*}}\!=\!y_{u^{**}}\!=\!1$ and $y_{u}\!=\!y_0$
for each $u\in V(G_0)\setminus\{u^*,u^{**}\}$.
Then, we obtain $\rho_0y_0=2y_{u^{*}}+2y_0$, which yields that $y_0=\frac{2}{\rho_0-2}$,
and thus, $$\mathbf{y}^{\mathrm{T}}\mathbf{y}\!=\!2y^2_{u^*}\!+\!(n\!-\!2)y_0^2
\!=\!2\!+\!\frac{4(n\!-\!2)}{(\rho_0\!-\!2)^2}.$$
Observe that $(\rho_0\!-\!2)^2<2n.$
Then, $(\rho_0\!-\!2)^2\mathbf{y}^{\mathrm{T}}\mathbf{y}\!=\!2(\rho_0\!-\!2)^2\!+\!4(n\!-\!2)\!<\!8n$.

Choose an arbitrary graph $G\in\mathbb{G}(n,\gamma)\cap\mathbb{H}(n,\gamma),$
and let $\rho=\rho(G)$. Based on (\ref{eq4}),
we have $\rho\!\geq\!\rho\big(K_2\nabla(n\!-\!2)K_1\big)\!>\!\sqrt{2n\!+\!8}$.
We can modify $G_0$ by deleting an edge to obtain $K_2\nabla P_{n-2}$,
and then adding $3\gamma$ edges within $V(G_0)\setminus\{u^*,u^{**}\}$ to obtain the resulting graph $G$.
Thus,
\begin{eqnarray*}
\rho\!-\!\rho_0
\!\geq\!\frac{\mathbf{y}^{\mathrm{T}}\big(A(G)\!-\!A(G_0)\big)\mathbf{y}}{\mathbf{y}^{\mathrm{T}}\mathbf{y}}
\!=\!\frac{2(3\gamma\!-\!1)y_0^2}{\mathbf{y}^{\mathrm{T}}\mathbf{y}}
\!=\!\frac{8(3\gamma\!-\!1)}{(\rho_0\!-\!2)^2\mathbf{y}^{\mathrm{T}}\mathbf{y}}
\!>\!\frac{3\gamma\!-\!1}{n}.
\end{eqnarray*}
Therefore, we conclude that $\rho>\rho_0\!+\!\frac{1}n(3\gamma\!-\!1).$

Now, let $R=V(G)\setminus\{u^*,u^{**}\},$
and let $\mathbf{x}=(x_1,\ldots,x_n)^\mathrm{T}$
be a positive eigenvector corresponding to $\rho(G).$
Since $u^*$ and $u^{**}$ are both dominating vertices in $G$,
we have $x_{u^*}=x_{u^{**}}$, and we may assume that $x_{u^*}=1$.
Recall that every graph in $\mathbb{G}(n,\gamma)$ is $K_{3,2\gamma+3}$-free.
Thus, for any $u\in R$, we have $d_R(u)\leq2\gamma+2$. Since $\rho x_u=2+\sum_{v\in N_R(u)}x_v$,
it is clear that
\begin{eqnarray}\label{eq19}
\frac2\rho\leq\min_{u\in R}x_u\leq\max_{u\in R}x_u\!\leq\!\frac2{\rho\!\!-\!\!2\gamma\!\!-\!\!2},
\end{eqnarray}
which implies that $\frac4{\rho^2}\leq x_ux_v\leq\frac4{(\rho\!-\!2\gamma\!-\!2)^2}$ for any two vertices $u,v\in R$.
Hence,
\begin{eqnarray}\label{eq20}
\rho_0-\rho
\geq \frac{\mathbf{x}^{\mathrm{T}}\big(A(G_0)\!-\!A(G)\big)\mathbf{x}}{\mathbf{x}^{\mathrm{T}}\mathbf{x}}
\geq\frac{8}{\mathbf{x}^{\mathrm{T}}\mathbf{x}}
\Big(\frac{1}{\rho^2}\!-\!\frac{3\gamma}{(\rho\!\!-\!\!2\gamma\!\!-\!\!2)^2}\Big),
\end{eqnarray}
where
$\mathbf{x}^{\mathrm{T}}\mathbf{x}\!=\!2\!+\!\sum_{u\in R}x_{u}^2\geq2\!+\!\frac4{\rho^2}(n\!-\!2)$.
Recall that $\rho\!>\!\sqrt{2n\!+\!8}$.
It follows that $\rho^2\mathbf{x}^{\mathrm{T}}\mathbf{x}\!\geq\!2\rho^2\!+\!4(n\!-\!2)\!>\!8n$.
Moreover, since $\rho>\sqrt{2n}\geq\frac{10}{\alpha}=\!3000\!+\!1800\gamma\!+\!240\gamma^2$,
a straightforward calculation shows that
$3\gamma/(\rho\!\!-\!\!2\gamma\!\!-\!\!2)^2\!\leq\!(3\gamma\!+\!0.05)/\rho^2.$
Combining (\ref{eq20}) gives: $$\rho_0-\rho
\geq\frac{8}{\rho^2\mathbf{x}^{\mathrm{T}}\mathbf{x}}\big(0.95\!-\!3\gamma\big)>\frac{1}{n}\big(0.95\!-\!3\gamma\big).$$
It follows that $\rho<\rho_0\!+\!\frac{1}n(3\gamma\!-\!0.95),$
completing the proof.
\end{proof}

\section{Common lemmas and claims for Theorem \ref{thm1.2}}

Recall that $\mathbb{G}(n,\gamma)$ is the family of graphs of order $n$ that can be embedded on a surface with Euler genus $\gamma$,
and $\mathbb{H}(n,\gamma)$ represents the family of graphs where each member is obtained from
$K_2\nabla P_{n-2}$ by adding $3\gamma$ edges.
In this section, we will explore more precise structural properties of the graphs in $\mathbb{G}(n,\gamma)\cap \mathbb{H}(n,\gamma),$
under the assumptions that $\gamma\geq1$ and $n$ is sufficiently large.

One natural candidate for maximizing the spectral radius in
$\mathbb{H}(n,\gamma)$ is obtained by making all $3\gamma$ added edges
incident to the same internal vertex of the path $P_{n-2}$.
The graph obtained in this way contains a copy of
$K_{3,3\gamma+2}$.
By Lemma \ref{lem3.4}, $K_{3,3\gamma+2}$ cannot be embedded
on a surface with Euler genus $\gamma$.
Thus, the proof of Theorem \ref{thm1.2} requires a careful
combination of spectral techniques and structural analysis.
To facilitate this, we introduce several definitions and additional lemmas.

Recall that $d_H(u)$ denotes the degree of a vertex $u$ in a graph $H$.
We now introduce the concept of the \emph{$\ell$--degree} $w^{(\ell)}_H(u)$ of a vertex $u$,
which is defined as the number of $\ell$--walks starting from $u$.
For every $u\in V(H)$, we define $w^{(0)}_H\!(u)=1$, and clearly,
we have $w^{(1)}_H\!(u)=d_H(u)$.
Furthermore, it is easy to see that
\begin{eqnarray}\label{eq21}
w^{(\ell+1)}_H\!(u)=\!\!\!\sum_{v\in N_{H}(u)}\!\!\!w^{(\ell)}_H\!(v).
\end{eqnarray}

As early as 2002, Nikiforov \cite{NIKI} employed this parameter to investigate the relationship between the spectral radius and clique number of graphs.
Let $w^{(\ell)}\!(H)$ denote the total number of $\ell$--walks in $H$.
Clearly, $w^{(1)}(H)=2e(H)$. Furthermore, we have
\begin{eqnarray}\label{eq22}
w^{(2)}(H)=\!\!\!\sum_{v\in V(H)}\!\!\!d^2_H\!(v)~~~\mbox{and}
~~~w^{(3)}(H)=\!\!\!\sum_{uv\in E(H)}\!\!\!2d_H\!(u)d_H\!(v).
\end{eqnarray}
In view of (\ref{eq21}), we can deduce that
\begin{eqnarray}\label{eq23}
w^{(4)}(H)=\!\!\!\sum_{v\in V(H)}\!\!\!\big(w^{(2)}_H\!(v)\big)^2,~~~\mbox{where}~~w^{(2)}_H\!(v)=\!\!\!\sum_{u\in N_H\!(v)}\!\!\!d_H\!(u).
\end{eqnarray}

For an integer $r\geq2$, let $K_{n_1,\ldots,n_r}$ be the complete $r$-partite graph,
with partition classes $V_1,\ldots,V_r$ satisfying $|V_i|=n_i$ and $\sum_{i=1}^rn_i=n.$
The following lemma is due to Zhang \cite{Zhang}, and it provides a useful tool for studying spectral extremal problems.

\begin{lem}(\cite{Zhang})\label{lem5.1}
For each $1\leq i\leq r$, let $H_i$ be a graph with $V(H_i)\subseteq V_i$.
Let $G$ be the graph obtained from $K_{n_1,n_2,\ldots,n_r}$ by embedding the edges of $H_i$ into the partition class $V_i$ for $i\in\{1,\ldots,r\}$.
If $x>\rho(H_i)$, then
$\sum_{\ell=1}^\infty\frac{w^{(\ell)}(H_i)}{x^{\ell}}$ is convergent.
Moreover,
\begin{eqnarray*}
\sum\limits_{i=1}^{r}\frac1{1+\frac{n_i}{\rho}+\sum\limits_{\ell=1}^\infty\frac{w^{(\ell)}(H_i)}{\rho^{\ell+1}}}=r-1,\,\, \text{where}\,\,\rho=\rho(G).
\end{eqnarray*}
\end{lem}

Let $G_1$ and $G_2$ be two distinct graphs in $\mathbb{G}(n,\gamma)\cap\mathbb{H}(n,\gamma)$ such that $V(G_1)=V(G_2)$.
For $j\in\{1,2\}$, let $H_j$ be the subgraph of $G_j$ induced by $V(G_j)\setminus\{u^*,u^{**}\}$,
where $u^*,u^{**}$ are the two dominating vertices of $G_j$.
By Lemma \ref{lem3.4}, $G_j$ is $K_{3,2\gamma+3}$-free, and thus $\Delta(H_j)\leq 2\gamma+2$.

\begin{lem}\label{lem5.2}
If there exists a constant $k$ with respect to $n$ such that $w^{(k)}(H_1)>w^{(k)}(H_2)$,
and $w^{{(\ell)}}(H_1)=w^{{(\ell)}}(H_2)$ for all $\ell\leq k-1$, then we have
$\rho(G_1)>\rho(G_2)$.
\end{lem}

\begin{proof}
For $j\in\{1,2\}$,
as noted above, we have $\Delta(H_j)\leq2\gamma+2$.
Moreover, $H_j$ is obtained from $P_{n-2}$ by adding $3\gamma$ edges, denoted as $e_1,\ldots,e_{3\gamma}.$
We first claim that
\begin{eqnarray}\label{eq24}
0\leq w^{(\ell)}(H_j)-w^{(\ell)}(P_{n-2})\leq 6\gamma(4\gamma+4)^{\ell-1}.
\end{eqnarray}

Since $P_{n-2}$ is a spanning subgraph of $H_j$, the left-hand side of (\ref{eq24}) holds trivially.
We now establish the right-hand side of (\ref{eq24}).
For every edge $e_i$, where $i\in\{1,\ldots,3\gamma\}$,
let $W_{H_j}^{{(\ell)}}\!(e_i)$ denote the set of $\ell$--walks in $H_j$ that contain the edge $e_i$
(it is not necessary for these walks to start from $e_i$).
It is clear that
\begin{eqnarray}\label{eq25}
w^{(\ell)}(H_j)\leq w^{(\ell)}(P_{n-2})+\sum_{i=1}^{3\gamma}\big|W_{H_j}^{{(\ell)}}\!(e_i)\big|.
\end{eqnarray}
Furthermore,
every walk in $W_{H_j}^{(\ell)}\!(e_i)$ can be shortened to a walk in $W_{H_j}^{(\ell-1)}\!(e_i)$ by deleting its ending edge if $e_i$ is the starting edge of the walk,
or by deleting its starting edge otherwise. Since $\Delta(H_j) \le 2\gamma + 2$,
every walk in $W_{H_j}^{(\ell-1)}\!(e_i)$ corresponds to at most $4\gamma + 4$ walks in $W_{H_j}^{(\ell)}\!(e_i)$ via the above operation.
Hence,
$|W_{H_j}^{(\ell)}\!(e_i)| \le (4\gamma + 4) \, |W_{H_j}^{(\ell\!-\!1)}\!(e_i)|,$
and
$$\big|W_{H_j}^{(\ell)}\!(e_i)\big|\le(4\gamma+4)^{\ell-1}\big|W_{H_j}^{(1)}\!(e_i)\big|=2(4\gamma+4)^{\ell-1}.$$
Combining this with (\ref{eq25}), we derive the right-hand side of inequality (\ref{eq24}).

Since $\Delta(H_j)\leq 2\gamma+2$ for $j\in\{1,2\}$, we have $\rho(H_j)\leq\Delta(H_j)<\sqrt{2n}$.
By Lemma \ref{lem5.1}, we know that $\sum_{\ell=1}^\infty\frac{w^{(\ell)}(H_j)}{\rho^{\ell}}$ is convergent for $\rho\geq\sqrt{2n}$.
Next, we prove that
\begin{eqnarray}\label{eq26}
\sum_{\ell=1}^\infty\frac{w^{(\ell)}(H_1)-w^{(\ell)}(H_2)}{\rho^{\ell+1}}>0
\end{eqnarray}
for any $\rho\geq\sqrt{2n}$.
By the assumption of the lemma, we know that
\begin{eqnarray}\label{eq27}
\sum_{\ell=1}^k\frac{w^{(\ell)}(H_1)-w^{(\ell)}(H_2)}{\rho^{\ell+1}}\geq\frac1{\rho^{k+1}}.
\end{eqnarray}
In view of (\ref{eq24}), we derive that
$$w^{(\ell)}(H_1)\!-\!w^{(\ell)}(H_2)\geq w^{(\ell)}(P_{n-2})\!-\!w^{(\ell)}(H_2)\geq-6\gamma(4\gamma\!+\!4)^{\ell-1}.$$
Since $\rho\geq\sqrt{2n}$ and $6\gamma(4\gamma\!+\!4)^k$ is constant with respect to $n$, it follows that
\begin{eqnarray*}
\sum_{\ell=k\!+\!1}^\infty\!\!\!\!\frac{w^{(\ell)}\!(H_1)\!-\!w^{(\ell)}\!(H_2)}{\rho^{\ell\!+\!1}}
\geq\!\!\!\sum_{\ell=k\!+\!1}^\infty\!\!\!\!\frac{-6\gamma(4\gamma\!+\!4)^{\ell\!-\!1}}{\rho^{\ell\!+\!1}}
=\frac{-6\gamma(4\gamma\!+\!4)^k}{\rho^{k\!+\!2}}\cdot\frac1{1\!-\!\frac{4\gamma\!+4}{\rho}}
>\frac{-1}{\rho^{k\!+\!1}}.
\end{eqnarray*}
Combining this with inequality (\ref{eq27}), we conclude that inequality (\ref{eq26}) holds.

Note that for $j\in\{1,2\}$, the graph $G_j$ can be constructed from a complete 3-partite graph $K_{1,1,n-2}$ by embedding
a copy of $H_j$ into its third partition.
Define $$f_j(\rho)=\frac2{1+\frac{1}{\rho}}+\frac{1}{1+\frac{n-2}{\rho}+\sum\limits_{\ell=1}^\infty\frac{w^{(\ell)}\!(H_j)}{\rho^{\ell+1}}}-2,$$
where $\rho\geq\sqrt{2n}$ and $j\in\{1,2\}$.
By Lemma \ref{lem5.1}, we have
$f_1\big(\rho(G_1)\big)=f_2\big(\rho(G_2)\big)=0.$
Based on inequality (\ref{eq26}), it follows that $f_2(\rho)>f_1(\rho)$ for any $\rho\geq\sqrt{2n}$.
Since Theorem \ref{thm4.2} indicates that $\rho(G_1)\geq\sqrt{2n}$,
we obtain $$f_2\big(\rho(G_1)\big)>f_1\big(\rho(G_1)\big)=f_2\big(\rho(G_2)\big).$$
Noting that $f_2(\rho)$ is increasing in $\rho$,
we then conclude that $\rho(G_1)>\rho(G_2)$.
\end{proof}

Let $H_0(a,b)$ denote the graph obtained from a base graph $H_0$
by attaching two pendant paths of lengths $a$ and $b$ to two distinct vertices $u,v$ of $H_0$, respectively.

\begin{lem}\label{lem5.3}
Let $a\!+\!b\!+\!r=n\!-\!2$, and let $H_0$ be a graph of order $r\geq3$ satisfying $\delta(H_0)\geq2$,
where $a,b\geq1$ and $r$ is constant with respect to $n$.
Assume $H_0$ admits a spanning path between $u$ and $v$. \\
(i) If $a\geq b+2$, then $\rho\big(K_2\nabla H_0(a\!-\!1,b\!+\!1)\big)>\rho\big(K_2\nabla H_0(a,b)\big).$\\
(ii) If $a=b+1$ and $d_{H_0}(u)>d_{H_0}(v)$, then $\rho\big(K_2\nabla H_0(a,b)\big)>\rho\big(K_2\nabla H_0(b,a)\big).$
\end{lem}

\begin{proof}
Denote $H_1=H_0(a\!-\!1,b\!+\!1)$ and $H_2=H_0(a,b)$.
Since $\rho(K_2\nabla(n\!-\!2)K_1)=\frac12(1\!+\!\sqrt{8n-15})$,
 and $K_2\nabla(n\!-\!2)K_1$ is a spanning subgraph of $K_2\nabla H_i$ for $i\in\{1,2\}$,
it follows that $\rho(K_2\nabla H_i)\geq\frac12(1\!+\!\sqrt{8n-15})\geq\sqrt{2n}$ for sufficiently large $n$.

By an argument analogous to that in the proof of Lemma \ref{lem5.2}, we can see that if
\begin{eqnarray*}
S:=\sum_{\ell=1}^\infty\frac{w^{(\ell)}\!(H_1)-w^{(\ell)}\!(H_2)}{\rho^{\ell+1}}>0
\end{eqnarray*}
for every $\rho\geq\sqrt{2n}$, then $\rho\big(K_2\nabla H_1\big)>\rho\big(K_2\nabla H_2\big).$
For case (i), it suffices to prove $S>0$ for any $\rho\geq\sqrt{2n}$.
For case (ii), we have $H_1=H_0(b,a)$ and $H_2=H_0(a,b)$, so the desired inequality is
$\rho(K_2\nabla H_2)>\rho(K_2\nabla H_1)$.
Thus, it suffices to show $S<0$ for any $\rho\geq\sqrt{2n}$.

Let $P=u_1u_2\ldots u_{n-2}$ be a spanning path of $H_2$, where $u=u_{a+1}$ and $v=u_{a+r}=u_{n-b-2}$.
Then, $H_0$ is the subgraph of $H_2$ induced by $\{u_{a+1},\ldots,u_{a+r}\}$.
Moreover, $H_1$ can be constructed from $H_2$ by deleting the edge $u_1u_2$ and adding the edge $u_{n-2}u_1$.
Let $W^{(\ell)}_{H_1}(e)$ denote the set of $\ell$--walks in $H_1$ that contains the edge $e$,
and let $D_\ell:=|W^{(\ell)}_{H_1}\!(u_{n\!-\!2}u_1)|-|W^{(\ell)}_{H_2}\!(u_1u_2)|$.
Since $H_1-\{u_{n-2}u_1\}\cong H_2-\{u_1u_2\}$, we have
$w^{(\ell)}\!(H_1)-w^{(\ell)}\!(H_2)=D_\ell$, and thus $S=\sum_{\ell=1}^\infty\frac{D_\ell}{\rho^{\ell+1}}$
for any integer $\ell\geq1$.

We first show the following equality:
\begin{eqnarray}\label{eq28}
\sum_{\ell=1}^{b+2}\frac{D_\ell}{\rho^{\ell+1}}=\frac{D_{b+2}}{\rho^{b+3}}
\end{eqnarray}

Recall that $a\geq b\!+\!1$ in cases (i) and (ii).
Observe that for $\ell\leq b+1$,
every $\ell$--walk in $W^{(\ell)}_{H_2}\!(u_1u_2)$ can only use edges from the induced path
$P'_{\ell+1}=u_1u_2\ldots u_{\ell+1}$ in $H_2$.
Similarly, for $\ell\leq b+1$,
every $\ell$--walk in $W^{(\ell)}_{H_1}\!(u_{n-2}u_1)$ can only use edges from the induced path
$P''_{\ell+1}=u_1u_{n-2}\ldots u_{n-\ell-1}$ in $H_1$.
Hence, for $\ell\leq b+1$, we have $P'_{\ell+1}\cong P''_{\ell+1}$,
and thus $D_\ell=0.$
Therefore, (\ref{eq28}) holds.

Now, assume $\ell=b+2$. If $a\geq b+2$,
every $\ell$-walk in $W^{(\ell)}_{H_2}\!(u_1u_2)$ can only utilize edges from the induced path $P'_{b+3}=u_1u_2\ldots u_{b+3}$ in $H_2$;
and if $a=b+1$, every $\ell$-walk in $W^{(\ell)}_{H_2}\!(u_1u_2)$ can only employ edges of the subgraph obtained from the induced path
$P'_{b+2}=u_1u_2\cdots u_{b+2}$ in $H_2$ by attaching $d_{H_0}(u)$ edges to $u=u_{a+1}=u_{b+2}$.
By contrast, every $\ell$-walk in $W^{(\ell)}_{H_1}\!(u_{n-2}u_1)$ can use edges of the subgraph
obtained from the induced path $P''_{b+2}=u_1u_{n-2}\ldots u_{n-b-2}$ in $H_1$ by attaching $d_{H_0}(v)$ edges to $v=u_{a+r}=u_{n-b-2}$.
Recall that $D_\ell=|W^{(\ell)}_{H_1}\!(u_{n\!-\!2}u_1)|-|W^{(\ell)}_{H_2}\!(u_1u_2)|$. Thus
\begin{eqnarray}\label{eq29}
D_{b+2}=~~
\begin{cases}
2\big(d_{H_0}(v)\!-\!1\big)\geq2, & a\geq b+2;\\
2\big(d_{H_0}(v)\!-\!d_{H_0}(u)\big)\leq-2, & a=b+1.
\end{cases}
\end{eqnarray}


Next, we prove that for $\rho\geq\sqrt{2n}$, the following inequalities hold:
\begin{eqnarray}\label{eq30}
\sum_{\ell=b+3}^{\infty}\frac{D_\ell}{\rho^{\ell+1}}
\begin{cases}
>-\frac12\frac{|D_{b+2}|}{\rho^{b+3}}, & a\geq b+2;\\
<\frac12\frac{|D_{b+2}|}{\rho^{b+3}}, & a=b+1.
\end{cases}
\end{eqnarray}
Recall that $H_2$ has a spanning path $P=u_1u_2\ldots u_{n-2}$
and $H_1$ has a spanning path $Q=u_2\ldots u_{n-2}u_1$.
Define $W^{(\ell)}_1\!(u_{n\!-\!2}u_1)=W^{(\ell)}_{H_1}\!(u_{n\!-\!2}u_1)\setminus W^{(\ell)}_{Q}\!(u_{n\!-\!2}u_1)$
and $W^{(\ell)}_2\!(u_1u_2)=W^{(\ell)}_{H_2}\!(u_1u_2)\setminus W^{(\ell)}_{P}\!(u_1u_2)$.
It is easy to see that $|W^{(\ell)}_{P}\!(u_1u_2)|=|W^{(\ell)}_{Q}\!(u_{n\!-\!2}u_1)|$.
Thus for every $\ell\geq 1$, we have
\begin{eqnarray}\label{eq31}
D_\ell=\big|W^{(\ell)}_{H_1}\!(u_{n\!-\!2}u_1)\big|\!-\!\big|W^{(\ell)}_{H_2}\!(u_1u_2)\big|
\!=\!\big|W^{(\ell)}_1\!(u_{n\!-\!2}u_1)\big|\!-\!\big|W^{(\ell)}_2\!(u_1u_2)\big|.
\end{eqnarray}
Let $E_0=E(H_2)\setminus E(P)=E(H_1)\setminus E(Q)$.
Clearly, $E_0\subseteq E(H_0)$, and every walk in $W^{(\ell)}_1\!(u_{n\!-\!2}u_1)\cup W^{(\ell)}_2\!(u_1u_2)$
visits $u_1$ and passes through at least one edge of $E_0$.

We begin with an elementary estimate for walks along a path of length $t.$
For integers $t\ge 1$ and $m\ge 0$, let $p_t(m)$ denote the number of $m$-walks on a path $w_0w_1\cdots w_t$
that start at $w_0$ and arrive at $w_t$ for the first time on their final step.
Clearly, we have $p_t(t)=1$ and $p_t(m)=0$ for $0\leq m<t$. Define
$T_t=\sum_{m=0}^\infty\frac{p_t(m)}{\rho^m}.$ We claim that
\begin{eqnarray}\label{eq32}
\frac1{\rho^t}\leq T_t\leq\frac1{\rho^t}\big(1+\frac1{\rho^2\!-\!4}\big)^{t-1}\leq \frac1{\rho^t}e^{(t-1)/(\rho^2-4)}.
\end{eqnarray}
Indeed, the lower bound in \eqref{eq32} is trivial, while the second upper bound for $T_t$ follows from the inequality $1+x\le e^x$,
which holds for all real $x$.
We now establish the first upper bound for $T_t$.
Consider an arbitrary $m$-walk, and decompose it at its first visits to the vertices $w_1,\dots,w_t$.
The initial segment is simply $w_0w_1$.
Every subsequent segment starts at $w_{j-1}$ and reaches $w_j$ for the first time on its last step; hence its length takes the form $2k+1$ for some integer $k\ge 0$.
When $k\ge1$, removing the final step yields a closed $2k$-walk whose first step is leftward and whose last step is rightward.
The remaining $2k-2$ steps admit at most $2^{2k-2}=4^{k-1}$ possible direction sequences.
Recall $\rho^2\ge 2n>4$. The weighted sum for each of the $t-1$ subsequent segments satisfies
$\frac1\rho+\sum_{k=1}^\infty\frac{4^{k-1}}{\rho^{2k+1}}=\frac1{\rho}(1+\frac1{\rho^2-4}).$
The weight of the initial segment $w_0w_1$ is equal to $\frac1\rho$.
Multiplying these upper bounds over the $t$ segments yields the claimed estimate \eqref{eq32}.

We now partition each $\ell$-walk in $W^{(\ell)}_1(u_{n-2}u_1)$ and $W^{(\ell)}_2(u_1u_2)$ into two disjoint families.

{\bf Type I}: Walks starting at $u_1$, proceeding along the pendant path until first reaching its attachment vertex in $H_0$,
then traversing one edge in $E_0$ and terminating, together with their reversals.

{\bf Type II}: All remaining walks.

The pendant path connecting $u_1$ to $H_0$ has length $b+1$ within $H_1$ and length $a$ within $H_2$,
with corresponding attachment vertices $v$ and $u$, respectively.
Precisely $d_{H_0}(v)-1$ edges of $E_0$ are incident to $v$, and $d_{H_0}(u)-1$ edges of $E_0$ are incident to $u$.
Let $r_{1,\ell}$ and $r_{2,\ell}$ denote the counts of Type II walks in $W^{(\ell)}_1(u_{n-2}u_1)$ and $W^{(\ell)}_2(u_1u_2)$, respectively.
Then for $\ell\geq1$, we obtain
\begin{equation}\label{eq33}
\begin{aligned}
\big|W^{(\ell)}_1(u_{n-2}u_1)\big|
&=2\big(d_{H_0}(v)-1\big)p_{b+1}(\ell-1)+r_{1,\ell},\\
\big|W^{(\ell)}_2(u_1u_2)\big|
&=2\big(d_{H_0}(u)-1\big)p_a(\ell-1)+r_{2,\ell}.
\end{aligned}
\end{equation}

To estimate the Type II walks, fix either graph $H_i$ and let $t$ be the length of its pendant path from $u_1$ to $H_0$.
Then $t=b+1$ for $H_1$, and $t=a$ for $H_2$.
Every such walk visits $u_1$ and traverses an edge in $E_0$.
After reversing the walk if needed, select a visit to $u_1$ that occurs prior to some traversal of an edge of $E_0$.
Decompose the walk into three parts: a prefix terminating at this visit to $u_1$, a middle segment ending upon its first subsequent visit to $H_0$, and the remaining suffix.

Let $j$ and $k$ denote the lengths of the prefix and suffix, respectively.
Then $j\ge 0$, $k\ge 1$, and $j+k\ge 2$; if this fails, the walk would belong to Type I.
The middle segment is counted by $p_t(m)$, where $m$ stands for its length.
Since $\max\{\Delta(H_1),\Delta(H_2)\}\le r$, there are at most $r^j$ possible prefixes and at most $r^k$ possible suffixes.
Recall that $T_t=\sum_{m=0}^\infty\frac{p_t(m)}{\rho^m}.$
For $i\in\{1,2\}$ and $\rho\ge 2r$, we have
\begin{equation}\label{eq34}
\begin{aligned}
\sum_{\ell=1}^{\infty}\frac{r_{i,\ell}}{\rho^{\ell+1}}
\leq \frac{2T_t}{\rho}\!\!\sum_{\substack{j\ge 0,\ k\ge 1\\ j+k\ge 2}}\!\!\big(\frac{r}{\rho}\big)^{j+k}
=\frac{2T_t}{\rho}
\sum_{s=2}^{\infty}\sum_{\substack{j\ge 0,\ k\ge 1\\ j+k=s}}\!\!\big(\frac{r}{\rho}\big)^s
=\frac{2T_t}{\rho}
\sum_{s=2}^{\infty} s\big(\frac{r}{\rho}\big)^s.
\end{aligned}
\end{equation}
Observe that $\sum_{s=2}^\infty sx^s=\frac{x^2(2-x)}{(1-x)^2} \le 6x^2$ holds for any real number $x\in[0,\frac12]$.
It follows from \eqref{eq34} that $\sum_{\ell=1}^{\infty}\frac{r_{i,\ell}}{\rho^{\ell+1}}\le \frac{12 r^2 T_t}{\rho^3}.$
Since $t\in\{a,b+1\}<n$ and $\rho^2\ge 2n$, we have $\rho^2-4\geq2(t-1)$.
Combining these two inequalities with the second upper bound for $T_t$ in \eqref{eq32}, we obtain
\begin{equation}\label{eq35}
\sum_{\ell=1}^{\infty}\frac{r_{1,\ell}}{\rho^{\ell+1}}
\le \frac{C_r}{\rho^{b+4}}
\quad \mbox{and} \quad
\sum_{\ell=1}^{\infty}\frac{r_{2,\ell}}{\rho^{\ell+1}}
\le \frac{C_r}{\rho^{a+3}},
\end{equation}
where the constant $C_r=12r^2\sqrt{e}$.

If $a\geq b+2$, then by inequality \eqref{eq31}, we have
$D_\ell\geq-|W^{(\ell)}_2(u_1u_2)|$ for any $\ell\geq1$.
Combining this with the second equality in \eqref{eq33}, we deduce that
\begin{eqnarray*}
\sum_{\ell=b+3}^{\infty}\frac{D_\ell}{\rho^{\ell+1}}
\geq -\sum_{\ell=1}^{\infty}
\frac{|W^{(\ell)}_2(u_1u_2)|}{\rho^{\ell+1}}=-\frac{2\big(d_{H_0}(u)-1\big)}{\rho^2}\sum_{\ell=1}^{\infty}\frac{p_a(\ell-1)}{\rho^{\ell-1}}
-\sum_{\ell=1}^{\infty}\frac{r_{2,\ell}}{\rho^{\ell+1}}.
\end{eqnarray*}
Recall that $|H_0|\leq r$ and $T_t=\sum_{m=0}^\infty\frac{p_t(m)}{\rho^m}.$
It follows that
$\sum_{\ell=b+3}^{\infty}\frac{D_\ell}{\rho^{\ell+1}}
\geq-\frac{2r}{\rho^2}T_a-\sum_{\ell=1}^{\infty}\frac{r_{2,\ell}}{\rho^{\ell+1}}$.
Recall that $\rho^2-4\geq2(t-1)$ for $t=a$ and $\rho\geq\sqrt{2n}$.
Combining \eqref{eq32} and \eqref{eq35} with \eqref{eq29}, we derive
$$\sum_{\ell=b+3}^{\infty}\frac{D_\ell}{\rho^{\ell+1}}
\geq-\frac{2r\sqrt{e}}{\rho^{a+2}}-\frac{C_r}{\rho^{a+3}}
>-\frac12\frac{|D_{b+2}|}{\rho^{b+3}}.$$
Therefore, \eqref{eq30} holds when $a\geq b+2$.

If $a=b+1$, then combining \eqref{eq31}, \eqref{eq33} with \eqref{eq29} yields
$D_\ell=D_{b+2}p_{b+1}(\ell-1)+(r_{1,\ell}-r_{2,\ell}).$
By \eqref{eq29}, we have $D_{b+2}\le -2<0$. Moreover, since $p_{b+1}(\ell-1)\ge0$ and $r_{2,\ell}\ge0$, we have $D_\ell\le r_{1,\ell}$. Combining this with the first inequality in \eqref{eq35}, we obtain
\begin{equation*}
\sum_{\ell=b+3}^{\infty}\frac{D_\ell}{\rho^{\ell+1}}
\leq
\sum_{\ell=b+3}^{\infty}\frac{r_{1,\ell}}{\rho^{\ell+1}}
\leq
\frac{C_r}{\rho^{b+4}}
<
\frac12\frac{|D_{b+2}|}{\rho^{b+3}}.
\end{equation*}
Here, the last inequality holds for all sufficiently large $n$, because \eqref{eq29} gives $|D_{b+2}|\ge2$, while $\rho^2\ge2n$ implies $\rho\to\infty$ as $n\to\infty$.
Therefore, \eqref{eq30} holds when $a=b+1$.

Combining \eqref{eq28} and \eqref{eq30} with \eqref{eq29}, we obtain
$S>\frac12\frac{D_{b+2}}{\rho^{b+3}}>0$
when $a\geq b+2$, while
$S<-\frac12\frac{|D_{b+2}|}{\rho^{b+3}}<0$
when $a=b+1$.
Therefore, the lemma is established.
\end{proof}

%

Choose an extremal graph $G^*\in SPEX(n,\gamma)$.
By Theorem \ref{thm1.1}, we have $G^*\in\mathbb{G}(n,\gamma)\cap \mathbb{H}(n,\gamma)$.
Let $\widetilde{G}^*$ denote an embedding of $G^*$ on a surface $\Sigma^*$ with Euler genus $\gamma$,
and recall that $\gamma(G^*)$ stands for the Euler genus of $G^*$.
Based on the aforementioned three lemmas,
we now state some structural claims for $\widetilde{G}^*$,
which will be crucial to the proof of Theorem \ref{thm1.2}.

\begin{claim}\label{cl5.1}
We have $\gamma(G^*)=\gamma$ and $e(G^*)=3(n-2+\gamma)$.
Moreover, $\widetilde{G}^*$ is a triangulation of $\Sigma^*$, and thus for any vertex $u\in V(\widetilde{G}^*)$,
there are exactly $d_{\widetilde{G}^*}(u)$ faces incident to $u$,
which together form a wheel $W(u)$ of order $d_{\widetilde{G}^*}(u)+1$ centered at $u$.
\end{claim}

\begin{proof}
Since $G^*\in\mathbb{G}(n,\gamma)\cap \mathbb{H}(n,\gamma)$, we have $\gamma(G^*)\leq\gamma$ and $e(G^*)=3(n-2+\gamma)$.
On the other hand, applying Lemma \ref{lem3.3} to a minimal embedding of $G^*$ gives
$e(G^*)\leq3\big(n-2+\gamma(G^*)\big)$.
Hence, $\gamma\leq\gamma(G^*)$, and thus $\gamma(G^*)=\gamma$.
Thus, $\widetilde{G}^*$ is a minimal embedding of $G^*$ and,
since $G^*$ is connected, $\widetilde{G}^*$ is cellular.
By Euler's formula,
$f(\widetilde{G}^*)=2-n+e(G^*)-\gamma=2(n-2+\gamma)$.
Since every face has degree at least $3$ and
$2e(G^*)=6(n-2+\gamma)=3f(\widetilde{G}^*)$,
 every face has degree exactly $3$.
Hence, $\widetilde{G}^*$ is a triangulation.
Therefore, the final statement follows.
\end{proof}

Recall that $G^*$ is obtained from $K_2\nabla P_{n-2}$ by adding $3\gamma$ edges.
Let $u^*$ and $u^{**}$ be the two dominating vertices of $G^*$.
By Claim \ref{cl5.1}, $d_{\widetilde{G}^*}(u^*)=n-1$ implies that
$\widetilde{G}^*$ admits a spanning wheel centered at $u^*$.
We may assume without loss of generality that
the outer cycle of this wheel is
\begin{equation}\label{eq36}
C_{u^*}:=u^{**}u_1u_2\ldots u_{n-2}u^{**}.
\end{equation}
Let $H^*=G^*-\{u^*,u^{**}\}$.
Clearly, $H^*$ contains a spanning path $u_1u_2\ldots u_{n-2}$.

A vertex of degree $d$ in $H^*$ is called a \emph{$d$-vertex}.
Every $d$-vertex with $d\geq3$ is referred to as a \emph{fork}.
Along the spanning path $u_1u_2\ldots u_{n-2}$ of $H^*$,
a vertex $u_k$ is called an \emph{internal vertex} if there exist two forks $u_i$ and $u_j$ such that $i<k<j$.

\begin{claim}\label{cl5.2}
$d_{H^*}\!(u_1)=d_{H^*}\!(u_{n-2})=1$.
\end{claim}

\begin{figure}[H]
\definecolor{SwitchInk}{HTML}{263443}
\definecolor{SwitchBase}{HTML}{648B8A}
\definecolor{SwitchBlue}{HTML}{315D88}
\definecolor{SwitchRust}{HTML}{B16B4C}
\definecolor{SwitchVertex}{HTML}{42627B}
\definecolor{SwitchCool}{HTML}{C6D9E8}
\definecolor{SwitchWarm}{HTML}{ECD0BD}
\definecolor{SwitchTeal}{HTML}{CEE2DE}
\tikzset{
  switch edge/.style={draw=SwitchBase,line width=0.4mm},
  switch before/.style={draw=SwitchBlue,line width=0.5mm},
  switch after/.style={draw=SwitchRust,line width=0.5mm},
  switch vertex/.style={circle,ball color=SwitchVertex,draw=none,
    inner sep=0pt,minimum size=1.5mm},
  switch dot/.style={circle,fill=SwitchBase,draw=none,
    inner sep=0pt,minimum size=0.5mm},
  switch label/.style={text=SwitchInk,anchor=center,
    font=\fontsize{12.38}{8.65}\selectfont},
  switch small label/.style={switch label,
    font=\fontsize{8.38}{8.65}\selectfont}}
\centering
\begin{tikzpicture}[scale=0.95,x=1mm,y=1mm,
  inner xsep=0pt,inner ysep=0pt,outer xsep=0pt,outer ysep=0pt]
\coordinate (Ltop) at (30,20);
\coordinate (Lone) at (0,0);
\coordinate (Li) at (10,0);
\coordinate (Lip) at (20,0);
\coordinate (Ljm) at (30,0);
\coordinate (Lj) at (40,0);
\coordinate (Ln) at (50,0);
\coordinate (Ldouble) at (60,0);
\shade[top color=white,bottom color=SwitchCool,shading angle=18]
  (Ltop)--(Li)--(Lip)--cycle;
\shade[top color=white,bottom color=SwitchCool!88!white,shading angle=-18]
  (Ltop)--(Ljm)--(Lj)--cycle;
\shade[top color=white,bottom color=SwitchTeal!70!white,shading angle=-18]
  (Ltop)--(Ln)--(Ldouble)--cycle;
\draw[switch edge] (Ln)--(Ldouble);
\foreach \p in {Lone,Li,Lip,Ljm,Lj,Ln,Ldouble}
  \draw[switch edge] (Ltop)--(\p);
\draw[switch before] (Li)--(Lip);
\draw[switch before] (Ljm)--(Lj);
\foreach \p in {Ltop,Lone,Li,Lip,Ljm,Lj,Ln,Ldouble}
  \node[switch vertex] at (\p) {};
\foreach \x in {7,5,3,23,25,27,43,45,47}
  \node[switch dot] at (\x,0) {};
\node[switch label] at (31,25) {$u^{*}$};
\node[switch label] at (0,-5) {$u_{1}$};
\node[switch label] at (10,-5) {$u_{i}$};
\node[switch label] at (20,-5) {$u_{i+1}$};
\node[switch label] at (30,-5) {$u_{j-1}$};
\node[switch label] at (40,-5) {$u_{j}$};
\node[switch label] at (50,-5) {$u_{n-2}$};
\node[switch label] at (62,-4) {$u^{**}$};
\node[text=SwitchInk!70,anchor=base west,
  font=\fontsize{20.23}{17.07}\selectfont]  at (70,12) {$\Longrightarrow$};
\coordinate (Rtop) at (135,20);
\coordinate (Rone) at (90,0);
\coordinate (Ri) at (100,0);
\coordinate (Rj) at (110,0);
\coordinate (Rn) at (120,0);
\coordinate (Rip) at (130,5);
\coordinate (Rjm) at (140,5);
\coordinate (Rdouble) at (150,0);
\shade[top color=white,bottom color=SwitchWarm,shading angle=18]
  (Rtop)--(Ri)--(Rj)--cycle;
\shade[top color=white,bottom color=SwitchWarm!90!white,shading angle=18]
  (Rtop)--(Rn)--(Rip)--cycle;
\shade[top color=SwitchTeal!35!white,bottom color=SwitchTeal,
  shading angle=0]
  (Rn)--(Rip)--(Rdouble)--cycle;
\shade[top color=white,bottom color=SwitchTeal!85!white,shading angle=-18]
  (Rtop)--(Rjm)--(Rdouble)--cycle;
\draw[switch edge] (Rn)--(Rdouble);
\foreach \p in {Rone,Ri,Rj,Rn,Rip,Rjm,Rdouble}
  \draw[switch edge] (Rtop)--(\p);
\draw[switch edge] (Rdouble)--(Rip);
\draw[switch edge] (Rdouble)--(Rjm);
\draw[switch after] (Ri)--(Rj);
\draw[switch after] (Rip)--(Rn);
\foreach \p in {Rtop,Rone,Ri,Rj,Rn,Rip,Rjm,Rdouble}
  \node[switch vertex] at (\p) {};
\foreach \x in {93,95,97,113,115,117}
  \node[switch dot] at (\x,0) {};
\foreach \x in {133,135,137}
  \node[switch dot] at (\x,5.1) {};
\node[switch label] at (136,25) {$u^{*}$};
\node[switch label] at (90,-5) {$u_{1}$};
\node[switch label] at (100,-5) {$u_{i}$};
\node[switch label] at (110,-5) {$u_{j}$};
\node[switch label] at (120,-5) {$u_{n-2}$};
\node[switch small label] at (131,2.5) {$u_{i+1}$};
\node[switch small label] at (136,7.25) {$u_{j-1}$};
\node[switch label] at (150,-4) {$u^{**}$};
\end{tikzpicture}
\caption{The local edge-switching from $H^*$ to $H$.}
\label{fig.6}
\end{figure}

\begin{proof}
Note that $H^*$ is obtained from a spanning path $u_1u_2\ldots u_{n-2}$ by adding $3\gamma$ edges.
Then, there are at most $6\gamma$ forks in $H^*$.
For sufficiently large $n$,
there exist consecutive 2-vertices $u_{i},u_{i+1},\ldots,u_{j-1}$ in $H^*$,
with $i\geq 2$ and $i+2\leq j\leq n-3$.

By symmetry, it suffices to show $d_{H^*}\!(u_{n-2})=1$.
We now define a graph $G$ as follows: $G=G^*\!-\!\{u_iu_{i+1},u_{j-1}u_j\}\!+\!\{u_iu_j,u_{i+1}u_{n-2}\}$,
as illustrated in Figure \ref{fig.6}, where the case $u_{i+1}=u_{j-1}$ is also allowed.
Let $H$ be the subgraph of $G$ induced by $V(G)\setminus\{u^*,u^{**}\}$.
Then, $H$ admits a spanning path $u_1\ldots u_iu_j\ldots u_{n-2}u_{i+1}\ldots u_{j-1}$.
Hence, $G\in \mathbb{H}(n,\gamma)$.

Since $\widetilde{G}^*$ is a triangulation of $\Sigma^*$, every edge in $\widetilde{G}^*$ is incident to two 3-faces.
Recall that the local embedding round at $u^*$ forms a wheel with an outer cycle $u^{**}u_1u_2\ldots u_{n\!-\!2}u^{**}$.
Since $u_i,u_{i\!+\!1},u_{i\!+\!2},\ldots,u_{j\!-\!1}$ are 2-vertices in $H^*$, for every $k\in\{i,i\!+\!1,\ldots,j\!-\!1\}$,
$u_{k}$ and $u_{k\!+\!1}$ share exactly two common neighbors: $u^*$ and $u^{**}$.
Thus, $\{u^{**}\!,u_i,u_{i\!+\!1}\}$, $\ldots,\{u^{**}\!,u_{j\!-\!1},u_j\}$
each forms a 3-face.
We now embed $G$ on the surface $\Sigma^*$ as follows (see Figure \ref{fig.6}):

\vspace{1mm}
(a) Replace the $j-i$ faces $\{u^{*}\!,u_i,u_{i\!+\!1}\}, \ldots, \{u^{*}\!,u_{j\!-\!1},u_j\}$ with a face $\{u^*\!,u_i,u_j\}$.

(b) Replace the $j-i$ faces $\{u^{**}\!,u_i,u_{i\!+\!1}\}, \ldots, \{u^{**}\!,u_{j\!-\!1},u_j\}$ with a face $\{u^{**}\!,u_i,u_j\}$.

(c) Replace the face $\{u^*\!,u_{n\!-\!2},u^{**}\}$ with $2j\!-\!2i\!-\!1$ faces inside the triangle $\{u^*\!,u_{n\!-\!2},u^{**}\}$,
while keeping all other faces unchanged.

\vspace{1mm}
Thus, we obtain a cellular embedding of $G$ on $\Sigma^*$,
and hence, $G\in\mathbb{G}(n,\gamma)\cap\mathbb{H}(n,\gamma)$.
Note that $d_{H}\!(u_{n\!-\!2})\!=\!d_{H^*}\!(u_{n\!-\!2})\!+\!1$, $d_{H}\!(u_{j\!-\!1})\!=\!d_{H^*}\!(u_{j\!-\!1})\!-\!1$,
and $d_{H}\!(u_k)\!=\!d_{H^*}\!(u_k)$ for all other $u_k\in V(H)$.
Then, $w^{(1)}\!(H)=w^{(1)}\!(H^*)$, and in view of (\ref{eq22}), we deduce that
\begin{eqnarray}\label{eq37}
w^{(2)}\!(H)\!-\!w^{(2)}\!(H^*)\!=\!\!\sum\limits_{u\in V(H)}\!\!\!\!\big(d^2_H\!(u)\!-\!d^2_{H^*}\!(u)\big)
\!=\!2d_{H^*}\!(u_{n\!-\!2})\!-\!2d_{H^*}\!(u_{j\!-\!1})\!+\!2.
\end{eqnarray}
Recall that $d_{H^*}\!(u_{j\!-\!1})=2$. If $d_{H^*}\!(u_{n\!-\!2})\geq2$, then
$w^{(2)}\!(H)>w^{(2)}\!(H^*)$.
By Lemma \ref{lem5.2}, this would derive that $\rho(G)>\rho(G^*)$, a contradiction.
Thus, we have $d_{H^*}\!(u_{n-2})=1$.
\end{proof}

\begin{claim}\label{cl5.3}
There are no consecutive internal 2-vertices in $H^*$.
\end{claim}

\begin{proof}
Suppose, for the sake of contradiction, that $H^*$ contains a sequence of consecutive internal 2-vertices
$u_{i},u_{i+1},\ldots,u_{j-1}$, with $j\geq i+2$ and $u_{i-1}$ and $u_j$ being two forks.
Then, by Claim \ref{cl5.2}, it follows that $i\geq 3$ and $j\leq n-3$.

Define $G$ and its subgraph $H$ as described in the proof of Claim \ref{cl5.2}
(see Figure \ref{fig.6}).
As shown above, we know that
$w^{(1)}\!(H)=w^{(1)}\!(H^*)$ and $G\in \mathbb{G}(n,\gamma)\cap\mathbb{H}(n,\gamma)$.
Furthermore, by Claim \ref{cl5.2} and equality (\ref{eq37}),
we conclude that $w^{(2)}\!(H)=w^{(2)}\!(H^*)$.

Note that $E(H)\setminus\{u_iu_j,u_{n-2}u_{i+1}\}\!=\!E(H^*)\setminus\{u_iu_{i+1},u_{j-1}u_j\}.$
Moreover, $d_{H}\!(u_{n\!-\!2})\!=\!d_{H^*}\!(u_{n\!-\!2})\!+\!1=2$, $d_{H}\!(u_{j\!-\!1})\!=\!d_{H^*}\!(u_{j\!-\!1})\!-\!1=1$,
and $d_{H}\!(u_k)\!=\!d_{H^*}\!(u_k)$ for all other $u_k\in V(H)$.
From (\ref{eq22}), we know that $w^{(3)}(H)=\sum_{uv\in E(H)}2d_H\!(u)d_H\!(v).$ If $j\geq i+3$, then
\begin{eqnarray*}
&& w^{(3)}\!(H)\!-\!w^{(3)}\!(H^*)\nonumber\\
&=& 2d_{H}\!(u_i)d_{H}\!(u_j)\!+\!2d_{H}\!(u_{j\!-\!1})d_{H}\!(u_{j\!-\!2})\!+\!2d_{H}\!(u_{n\!-\!2})\big(d_{H}\!(u_{i\!+\!1})\!+\!d_{H}\!(u_{n\!-\!3})\big)\nonumber\\
&& \!-\!2d_{H^*}\!(u_i)d_{H^*}\!(u_{i\!+\!1})\!-\!2d_{H^*}\!(u_{j\!-\!1})\big(d_{H^*}\!(u_{j\!-\!2})\!+\!d_{H^*}\!(u_j)\big)\!-\!2d_{H^*}\!(u_{n\!-\!2})d_{H^*}\!(u_{n\!-\!3}).
\end{eqnarray*}
If $j=i+2$, then $u_{i+1}=u_{j-1}$. In this case, $u_iu_{i+1}=u_{j-1}u_{j-2}\in E(H^*)\setminus E(H)$, and
\begin{eqnarray*}
&& w^{(3)}\!(H)\!-\!w^{(3)}\!(H^*)\nonumber\\
&=& 2d_{H}\!(u_i)d_{H}\!(u_j)\!+\!2d_{H}\!(u_{n\!-\!2})\big(d_{H}\!(u_{j\!-\!1})\!+\!d_{H}\!(u_{n\!-\!3})\big)\nonumber\\
&& \!-\!2d_{H^*}\!(u_i)d_{H^*}\!(u_{i\!+\!1})\!-\!2d_{H^*}\!(u_{j\!-\!1})d_{H^*}\!(u_j)\!-\!2d_{H^*}\!(u_{n\!-\!2})d_{H^*}\!(u_{n\!-\!3}).
\end{eqnarray*}
Recall that $d_{H^*}\!(u_{i})\!=\!\cdots\!=\!d_{H^*}\!(u_{j\!-\!1})=2$.
In both cases, we derive that
\begin{eqnarray*}
w^{(3)}\!(H)\!-\!w^{(3)}\!(H^*)\!=\!
2\big(d_{H^*}\!(u_i)\!-\!2\big)\big(d_{H^*}\!(u_j)\!-\!2\big)\!+\!2(d_{H^*}\!(u_{n\!-\!3})\!-\!2\big).
\end{eqnarray*}
If $d_{H^*}\!(u_{n\!-\!3})\geq3$, then $w^{(3)}\!(H)\!>\!w^{(3)}\!(H^*)$. By Lemma \ref{lem5.2}, we have $\rho(G)>\rho(G^*)$,
a contradiction.
Hence, $w^{(3)}\!(H)\!=\!w^{(3)}\!(H^*)$ and $d_{H^*}\!(u_{n\!-\!3})=2$.
Therefore, the fork $u_j$ must have $j\leq n-4$.

In light of (\ref{eq23}), we know that
$w^{(4)}(H)=\sum_{v\in V(H)}\big(w^{(2)}_H\!(v)\big)^2,$ where $w^{(2)}_H\!(v)$ is the sum of the degrees of the neighbors of $v$ in $H$.
If $j\geq i+3$, there are exactly five vertices $u_k$ such that $w^{(2)}_H\!(u_k)\neq w^{(2)}_{H^*}\!(u_k)$,
namely those with indices $k\in\{i,j\!-\!2,j\!-\!1,n\!-\!3,n\!-\!2\}$.
Recall that both  $u_{i-1}$ and $u_j$ are forks. A direct calculation shows that
\begin{eqnarray}\label{eq38}
w^{(4)}\!(H)\!-\!w^{(4)}\!(H^*)\!=\!2\big(d_{H^*}\!(u_{i-1})\!-\!2\big)\big(d_{H^*}\!(u_j)\!-\!2\big)\!+\!2\big(d_{H^*}(u_{n-4})\!-\!2\big)>0.
\end{eqnarray}
If $j=i+2$, then $u_{i}=u_{j-2}$, and there are four vertices $u_k$ such that $w^{(2)}_H\!(u_k)\neq w^{(2)}_{H^*}\!(u_k)$,
namely those with indices $k\in\{i,j\!-\!1,n\!-\!3,n\!-\!2\}$.
A similar calculation yields the same inequality as in (\ref{eq38}).
Therefore, we conclude that $w^{(4)}\!(H)\!>\!w^{(4)}\!(H^*).$
By Lemma \ref{lem5.2}, it follows that $\rho(G)>\rho(G^*)$, a contradiction.
This establishes the claim.
\end{proof}

\begin{claim}\label{cl5.4}
If $u_{i+1}$ is an internal 2-vertex,
then $u_i$ and $u_{i+2}$ are adjacent forks.
\end{claim}

\begin{figure}[H]
\definecolor{SwitchInk}{HTML}{263443}
\definecolor{SwitchBase}{HTML}{648B8A}
\definecolor{SwitchBlue}{HTML}{315D88}
\definecolor{SwitchRust}{HTML}{B16B4C}
\definecolor{SwitchBlueFill}{HTML}{C6D9E8}
\definecolor{SwitchRustFill}{HTML}{ECD0BD}
\definecolor{SwitchTealFill}{HTML}{CEE2DE}
\definecolor{SwitchVertex}{HTML}{42627B}
\tikzset{
  switch edge/.style={draw=SwitchBase,line width=0.32mm},
  switch before/.style={draw=SwitchBlue,line width=0.52mm},
  switch after/.style={draw=SwitchRust,line width=0.52mm},
  switch vertex/.style={circle,ball color=SwitchVertex,draw=none,
    inner sep=0pt,minimum size=1.65mm},
  switch dot/.style={circle,fill=SwitchBase,draw=none,
    inner sep=0pt,minimum size=0.5mm},
  switch label/.style={text=SwitchInk,inner sep=0pt,
    font=\fontsize{12}{14.4}\selectfont}}
\centering
\begin{tikzpicture}[
  x=1mm,y=1mm,
  inner sep=0pt,outer sep=0pt,
  line cap=round,line join=round]
\coordinate (Ltop) at (30,20);
\coordinate (Lone) at (0,0);
\coordinate (Li) at (10,0);
\coordinate (Lip) at (20,0);
\coordinate (Lipp) at (30,0);
\coordinate (Ln) at (40,0);
\coordinate (Ldouble) at (50,0);
\shade[top color=white,bottom color=SwitchBlueFill,shading angle=18]
  (Ltop)--(Li)--(Lip)--cycle;
\shade[top color=white,bottom color=SwitchBlueFill!88!white,
  shading angle=-18]
  (Ltop)--(Lip)--(Lipp)--cycle;
\shade[top color=white,bottom color=SwitchTealFill!70!white,
  shading angle=-18]
  (Ltop)--(Ln)--(Ldouble)--cycle;
\draw[switch edge] (Ln)--(Ldouble);
\foreach \p in {Lone,Li,Lip,Lipp,Ln,Ldouble}
  \draw[switch edge] (Ltop)--(\p);
\draw[switch before] (Li)--(Lip)--(Lipp);
\foreach \p in {Ltop,Lone,Li,Lip,Lipp,Ln,Ldouble}
  \node[switch vertex] at (\p) {};
\foreach \x in {3,5,7,33,35,37}
  \node[switch dot] at (\x,0) {};
\node[switch label] at (31,25) {$u^{*}$};
\node[switch label] at (0,-5) {$u_{1}$};
\node[switch label] at (10,-5) {$u_{i}$};
\node[switch label] at (19,-5) {$u_{i+1}$};
\node[switch label] at (30,-5) {$u_{i+2}$};
\node[switch label] at (42,-5) {$u_{n-2}$};
\node[switch label] at (53,-4) {$u^{**}$};
\node[text=SwitchInk!70,anchor=base west,
  font=\huge]
  at (65,10) {$\Longrightarrow$};
\coordinate (Rtop) at (115,20);
\coordinate (Rone) at (90,0);
\coordinate (Ri) at (100,0);
\coordinate (Ripp) at (110,0);
\coordinate (Rn) at (120,0);
\coordinate (Rip) at (127,5);
\coordinate (Rdouble) at (140,0);
\shade[top color=white,bottom color=SwitchRustFill,shading angle=18]
  (Rtop)--(Ri)--(Ripp)--cycle;
\shade[top color=white,bottom color=SwitchRustFill!90!white,
  shading angle=18]
  (Rtop)--(Rn)--(Rip)--cycle;
\shade[top color=white,bottom color=SwitchTealFill!85!white,
  shading angle=-18]
  (Rtop)--(Rip)--(Rdouble)--cycle;
\shade[top color=SwitchTealFill!35!white,bottom color=SwitchTealFill]
  (Rn)--(Rip)--(Rdouble)--cycle;
\draw[switch edge] (Rn)--(Rdouble);
\foreach \p in {Rone,Ri,Rip,Ripp,Rn,Rdouble}
  \draw[switch edge] (Rtop)--(\p);
\draw[switch edge] (Rdouble)--(Rip);
\draw[switch after] (Ri)--(Ripp);
\draw[switch after] (Rip)--(Rn);
\foreach \p in {Rtop,Rone,Ri,Ripp,Rn,Rip,Rdouble}
  \node[switch vertex] at (\p) {};
\foreach \x in {93,95,97,113,115,117}
  \node[switch dot] at (\x,0) {};
\node[switch label] at (116,25) {$u^{*}$};
\node[switch label] at (90,-5) {$u_{1}$};
\node[switch label] at (100,-5) {$u_{i}$};
\node[switch label] at (110,-5) {$u_{i+2}$};
\node[switch label] at (121,-5) {$u_{n-2}$};
\node[switch label,font=\fontsize{10}{12}\selectfont]
  at (128,2) {$u_{i+1}$};
\node[switch label] at (140,-4) {$u^{**}$};
\end{tikzpicture}
\caption{The local edge-switching from $H^*$ to $H$. }
\label{fig.7}
\end{figure}

\begin{proof}
Let $u_{i+1}$ be an internal 2-vertex in $H^*$.
Then, by Claim \ref{cl5.3},
both $u_i$ and $u_{i+2}$ are forks.
Furthermore, by Claim \ref{cl5.2}, we conclude that $i\geq 2$ and $i+2\leq n-3$,

Suppose to the contrary that $u_iu_{i+2}\notin E(H^*)$.
Define $G$ as follows (see Figure \ref{fig.7}):
$$G=G^*\!-\!\{u_iu_{i+1},u_{i+1}u_{i+2}\}\!+\!\{u_iu_{i+2},u_{i+1}u_{n-2}\}.$$
Let $H$ be the subgraph induced by $V(G)\setminus\{u^*,u^{**}\}$.
Then, $w^{(1)}\!(H)=w^{(1)}\!(H^*)$.
Observe that Figure \ref{fig.7} is special case of Figure \ref{fig.6} where $j\!=\!i\!+\!2$.
Using a similar argument as in the proof of Claim \ref{cl5.2},
we know that $w^{(1)}\!(H)=w^{(1)}\!(H^*)$ and $G\in \mathbb{G}(n,\gamma)\cap\mathbb{H}(n,\gamma)$.
Furthermore, by setting $j=i+2$ in (\ref{eq37}), we obtain
$w^{(2)}\!(H)=w^{(2)}\!(H^*)$.
Moreover, a direct calculation using (\ref{eq22}) gives
\[
\begin{aligned}
w^{(3)}\!(H)-w^{(3)}\!(H^*)
=2\big(d_{H^*}(u_i)-2\big)\big(d_{H^*}(u_{i+2})-2\big)+2\big(d_{H^*}(u_{n-3})-2\big).
\end{aligned}
\]
Since $u_i$ and $u_{i+2}$ are forks and
$d_{H^*}(u_{n-3})\geq2$, it follows that
$w^{(3)}\!(H)>w^{(3)}\!(H^*)$.
By Lemma \ref{lem5.2}, $\rho(G)>\rho(G^*)$, a contradiction.
Therefore, the claim holds.
\end{proof}

By Claims \ref{cl5.3} and \ref{cl5.4},
every internal 2-vertex of $H^*$ always has two adjacent fork neighbors.
In particular, such internal 2-vertices form an independent set.

\begin{claim}\label{cl5.5}
There exists a base graph $H_0$ with $|H_0|\leq 9\gamma$ such that $H^*\cong H_0(a,b)$,
where $|a\!-\!b|\leq1$.
\end{claim}

\begin{proof}
By Claim \ref{cl5.2}, we have $d_{H^*}(u_1)=d_{H^*}(u_{n-2})=1$.
Recall that there are exactly $3\gamma$ edges outside the spanning path $u_1u_2\ldots u_{n-2}$.
Therefore, there are at most $6\gamma$ forks.

By Claim \ref{cl5.4}, each internal $2$-vertex $u_i$ gives rise to
the external edge $u_{i-1}u_{i+1}$.
Since distinct internal $2$-vertices give distinct external edges
and $H^*$ has exactly $3\gamma$ external edges,
there are at most $3\gamma$ internal $2$-vertices.
Let $u_p$ and $u_q$ be the first and the last forks, respectively,
on the spanning path, and let
$H_0=H^*[\{u_j~|~p\leq j\leq q\}]$.
Since $H^*$ contains at most $6\gamma$ forks and at most
$3\gamma$ internal $2$-vertices,
we have $|H_0|\leq9\gamma$.
The vertices preceding $u_p$ and following $u_q$ form two pendant
paths. Hence,
$H^*\cong H_0(a,b)$,
where $a+b+|H_0|=n-2$.

Every vertex of $H_0$ other than $u_p$ and $u_q$ has its two
neighbors on the spanning path in $H_0$.
Moreover, each of $u_p$ and $u_q$ has a path neighbor and an
external neighbor in $H_0$, since otherwise there would be a fork
outside $\{u_p,\ldots,u_q\}$.
Thus,
$\delta(H_0)\geq2$ and $|H_0|\geq3$.
Since $|H_0|$ is bounded by a constant $9\gamma$ with respect to $n$ and
$H_0(a,b)\cong H^*$ admits a spanning path, the hypotheses of
Lemma \ref{lem5.3} are satisfied whenever $a\geq b+1$.

Suppose, to the contrary, that $|a-b|\geq2$.
By reversing the spanning path if necessary, we may assume that
$a\geq b+2$, with the pendant path of length $a$ ending at $u_1$.
By Lemma \ref{lem5.3}, we have
$$\rho\big(K_2\nabla H_0(a-1,b+1)\big)>\rho\big(K_2\nabla H_0(a,b)\big)=\rho(G^*).$$

It remains to verify that $K_2\nabla H_0(a-1,b+1)$ is embeddable
on the same surface.
Define $G'=G^*-u_1u_2+u_1u_{n-2}$.
By Claim \ref{cl5.2}, $u_1$ has degree $3$ in $G^*$, with neighbors
$u^*,u^{**},u_2$.
Removing $u_1$ from the triangulation $\widetilde{G}^*$ creates the
$3$-face $u^*u^{**}u_2$.
Also, $u^*u^{**}u_{n-2}$ is a $3$-face since
$d_{H^*}(u_{n-2})=1$.
Hence, we may insert $u_1$ into this face and join it to
$u^*,u^{**}$ and $u_{n-2}$, obtaining a cellular embedding of
$G'$ on the same surface.
This operation transfers the end vertex $u_1$ from the longer
pendant path to the shorter one, and hence
$G'\cong K_2\nabla H_0(a-1,b+1)$.
Therefore, $G'\in\mathbb{G}(n,\gamma)$ and $\rho(G')>\rho(G^*)$,
contradicting $G^*\in SPEX(n,\gamma)$.
Consequently, $|a-b|\leq1$.
\end{proof}

\begin{claim}\label{cl5.6}
For any $uv\in E(\widetilde{G}^*)$, there are exactly two 3-faces incident to $uv$.
In particular, the two 3-faces incident to $u^*u^{**}$ are precisely $u^*u^{**}u_1$ and $u^*u^{**}u_{n-2}$.
\end{claim}

\begin{proof}
By Claim \ref{cl5.1},
$\widetilde{G}^*$ is a triangulation of $\Sigma^*$.
Thus the first statement holds immediately.
Furthermore, from Claim \ref{cl5.2}, we know that $d_{G^*}(u_1)=d_{G^*}(u_{n-2})=3$.
Since $N_{G^*}(u_1)=\{u^*,u_2,u^{**}\}$,
the two 3-faces incident to $u_1u^*$ must be $u_1u^*u_2$ and $u_1u^*u^{**}$.
Similarly, $u_{n-2}u^*u_{n-3}$ and $u_{n-2}u^*u^{**}$ are the two 3-faces incident to $u_{n-2}u^*$.
Therefore, the second statement follows.
\end{proof}

\begin{claim}\label{cl5.7}
Let $\mathbb{F}$ denote the set of 3-faces incident to neither $u^*$ nor $u^{**}$ in $\widetilde{G}^*$.
Then $|\mathbb{F}|=2\gamma$.
\end{claim}

\begin{proof}
By Claim \ref{cl5.6},
every edge is incident to exactly two 3-faces.
Thus, the total number of $3$-faces is $\frac{2}{3}e(G^*)=2(n-2+\gamma)$,
since $G^*\in\mathbb{H}(n,\gamma)$ and $e(G^*)=3(n-2+\gamma)$.
By Claims \ref{cl5.1} and \ref{cl5.6}, both $u^*$
and $u^{**}$ are incident to $n\!-\!1$ faces,
where exactly two faces, $u^*u^{**}u_1$ and $u^*u^{**}u_{n-2}$, are their common faces.
Hence, the number of 3-faces incident to neither $u^*$ nor $u^{**}$ is
$$|\mathbb{F}|=2(n-2+\gamma)-2(n-1)+2=2\gamma.$$
Therefore, the claim holds.
\end{proof}

For each $u\in V(\widetilde{G}^*)\setminus\{u^*,u^{**}\}$,
let $\mathbb{F}(u)$ be the subset of $\mathbb{F}$ consisting of all faces incident to $u$.
Recall that $\widetilde{G}^*$ contains a wheel $W(u^*)=K_1\nabla C_{u^*}$,
where $u^*$ is its center and $C_{u^*}=u^{**}u_1u_2\ldots u_{n-2}u^{**}$ is defined in (\ref{eq36}).

\begin{claim}\label{cl5.8}
Let $V^{*}$ denote the set of forks in $H^*$.
Then $|\mathbb{F}(u)|=d_{H^*}(u)-2$ for each $u\in V^*$ and $|\mathbb{F}(u)|=0$ for all $u\in V(H^*)\setminus V^*$.
\end{claim}

\begin{proof}
We first compute $|\mathbb{F}(u)|$ for each $u\in V(H^*)$.
It suffices to show $|\mathbb{F}(u)|=d_{H^*}(u_i)-2$ for $2\leq i\leq n-3$
and $|\mathbb{F}(u_i)|=0$ for $i\in\{1,n-2\}$.
By Claim \ref{cl5.6}, every edge is incident to two 3-faces.
By the structure of $W(u^*)$,
the two 3-faces incident to $u^*u_i$ are $u^*u_iu_{i-1}$ and $u^*u_iu_{i+1}$ whenever $i\in\{2,\ldots,n-3\}$.
Hence, the two 3-faces incident to $u^{**}u_i$ must be distinct from those incident to $u^*u_i$.
Moreover, by Claim \ref{cl5.1}, there are exactly $d_{\widetilde{G}^*}(u_i)$ faces incident to $u_i$.
Hence, $|\mathbb{F}(u_i)|=d_{\widetilde{G}^*}(u_i)-4=d_{H^*}(u_i)-2$ for $i\in\{2,\ldots,n-3\}$.
Assume now that $i\in\{1,n-2\}$.
Recall that $d_{H^*}(u_1)=d_{H^*}(u_{n-2})=1$, and thus $d_{G^*}(u_1)=d_{G^*}(u_{n-2})=3$.
Again by the structure of $W(u^*)$, the two 3-faces incident to $u^*u_1$ are $u^*u_1u_2$ and $u^*u_1u^{**}$.
Hence, $u^{**}u_1$ is incident to the 3-face $u^*u_1u^{**}$ and another 3-face.
Therefore, we conclude that $|\mathbb{F}(u_1)|=d_{\widetilde{G}^*}(u_1)-3=0$.
The same argument gives $|\mathbb{F}(u_{n-2})|=0$.
\end{proof}

\begin{claim}\label{cl5.9}
We have the following statements:\\
(i) $w^{(2)}(H^*)=4n+24\gamma-14+\sum_{u\in V(H^*)}{|\mathbb{F}(u)|}^2;$\\
(ii) $w^{(3)}(H^*)=8n+72\gamma-32+4\sum_{u\in V(H^*)}|\mathbb{F}(u)|^2+2\sum_{uv\in E(H^*)}|\mathbb{F}(u)||\mathbb{F}(v)|.$
\end{claim}

\begin{proof}
Set $\beta_u:=d_{H^*}(u)-2$ for each $u\in V(H^*)$.
By Claim \ref{cl5.8}, $\beta_{u_i}=|\mathbb{F}(u_i)|$ for $2\leq i\leq n-3$
and $\beta_{u_i}=-1=|\mathbb{F}(u_i)|-1$ for $i\in\{1,n-2\}$.
Hence,
$\sum_{u\in V(H^*)}\beta_u=-2+\sum_{u\in V(H^*)}|\mathbb{F}(u)|=-2+6\gamma$ and $\sum_{u\in V(H^*)}\beta_u^2=2+\sum_{u\in V(H^*)}{|\mathbb{F}(u)|}^2$.
Furthermore, Claim \ref{cl5.5} implies $d_{H^*}(u_2)=d_{H^*}(u_{n-3})=2$,
and Claim \ref{cl5.8} gives $|\mathbb{F}(u)|=0$ for each $u\in V(H^*)$ with $d_{H^*}(u)=2$.
Therefore, we conclude that $\sum_{uv\in E(H^*)}\beta_u\beta_v=\sum_{uv\in E(H^*)}|\mathbb{F}(u)||\mathbb{F}(v)|$.

In view of (\ref{eq22}), we have $w^{(2)}(H^*)=\sum_{u\in V(H^*)}d^2_{H^*}(u)$ and
$w^{(3)}(H^*)=\sum_{uv\in E(H^*)}2d_{H^*}(u)d_{H^*}(v).$
Now, we know that $d^2_{H^*}(u)=(\beta_u+2)^2$ and $d_{H^*}(u)d_{H^*}(v)=(\beta_u+2)(\beta_v+2)$ for any $u,v\in V(H^*)$.
It follows that $\sum_{u\in V(H^*)}d^2_{H^*}(u)=\sum_{u\in V(H^*)}(\beta_u^2+4\beta_u)+4|H^*|$
and $$\sum_{uv\in E(H^*)}\!\!\!d_{H^*}(u)d_{H^*}(v)=\!\!\!\sum_{uv\in E(H^*)}\!\!\!(\beta_u\beta_v+2\beta_u+2\beta_v)+4e(H^*),$$
where $$\sum_{uv\in E(H^*)}\!\!\!(\beta_u+\beta_v)=\!\!\!\sum_{u\in V(H^*)}\!\!\!\beta_ud_{H^*}(u)=\!\!\!\sum_{u\in V(H^*)}\!\!\!\beta_u(\beta_u+2).$$
Moreover, recall that $|H^*|=n-2$ and $e(H^*)=n+3\gamma-3.$
Straightforward calculations yield the desired two equalities for $w^{(2)}(H^*)$ and $w^{(3)}(H^*)$.
\end{proof}

Let $H^{*}[\mathbb{F}]$ be the subcomplex of $\widetilde{G}^*$ formed by all 3-faces in $\mathbb{F}$.
The underlying graph of $H^{*}[\mathbb{F}]$ is a subgraph of $H^*$.
By Claim \ref{cl5.8}, $|\mathbb{F}(u)|>0$ for each $u\in V^*$ and $|\mathbb{F}(u)|=0$ for all $u\in V(H^*)\setminus V^{*}$,
where $V^*$ is set of forks in $H^*$. Hence, $V(H^{*}[\mathbb{F}])=V^{*}$.
Let $H^*[V^{*}]$ be the subgraph of $H^*$ induced by $V^{*}$.
Then the underlying graph of $H^{*}[\mathbb{F}]$ is a subgraph of $H^*[V^{*}]$.

\begin{claim}\label{cl5.10}
For any $u\in V^{*}$, we have $|\mathbb{F}(u)|+1\leq d_{H^{*}[\mathbb{F}]}(u)\leq|\mathbb{F}(u)|+2$
and $d_{H^{*}}(u)\leq d_{H^{*}[\mathbb{F}]}(u)+1.$
Consequently, $E(H^*[V^{*}])\setminus E(H^{*}[\mathbb{F}])$ is a matching.
\end{claim}

\begin{proof}
Note that $u_1,u_{n-2}\notin V^{*}$.
For any $u\in V^{*}$,
Claim \ref{cl5.1} implies that the wheel $W(u)$ centered at $u$ has an outer $d_{G^*}(u)$-cycle $C_u$.
Since $u^*$ and $u^{**}$ are dominating vertices in $G^*$, we have $u^*, u^{**}\in V(C_u).$
Moreover, Claim \ref{cl5.6} implies that $u^*$ and $u^{**}$ are not consecutive on $C_u$.
Thus, removing these two dominating vertices from the cycle $C_u$ yields two paths, which may contain isolated vertices.
Consequently, $W(u)-\{u^*,u^{**}\}$ has one or two nontrivial components,
and these components precisely constitute all 3-faces in $\mathbb{F}(u)$ (see Figure \ref{fig.8}).
Hence, $d_{H^{*}[\mathbb{F}]}(u)|\in\{|\mathbb{F}(u)|+1,|\mathbb{F}(u)|+2\}$
and $d_{H^{*}}(u)-d_{H^{*}[\mathbb{F}]}(u)\in\{0,1\}$ for each $u\in V^{*}$.
This immediately implies that $E(H^*[V^{*}])\setminus E(H^{*}[\mathbb{F}])$ induces a matching.
\end{proof}

\begin{figure}[H]
\definecolor{WheelInk}{HTML}{293B4B}
\definecolor{WheelEdge}{HTML}{365E83}
\definecolor{WheelNode}{HTML}{5C8FB7}
\definecolor{WheelSpecial}{HTML}{B67D45}
\definecolor{WheelCenter}{HTML}{AE555F}
\definecolor{WheelDash}{HTML}{A0ABB5}
\definecolor{WheelFaceLight}{HTML}{F4F8FC}
\definecolor{WheelFaceMid}{HTML}{D5E3EF}
\definecolor{WheelFaceDark}{HTML}{A9C2D8}
\definecolor{WheelFaceTeal}{HTML}{BDD6D4}
\tikzset{
  wheel edge/.style={draw=WheelEdge,line width=0.60mm},
  wheel special edge/.style={draw=WheelSpecial,line width=0.40mm},
  wheel dashed edge/.style={draw=WheelDash,dashed,line width=0.30mm},
  wheel vertex/.style={circle,inner sep=0pt,minimum size=2mm,
    ball color=WheelNode,draw=WheelEdge,line width=0.12mm},
  wheel special vertex/.style={wheel vertex,
    ball color=WheelSpecial,draw=WheelSpecial!85!black},
  wheel center vertex/.style={wheel vertex,
    ball color=WheelCenter,draw=WheelCenter!85!black},
  wheel face one/.style={top color=WheelFaceLight,
    bottom color=WheelFaceMid,shading angle=-25},
  wheel face two/.style={left color=WheelFaceLight,
    right color=WheelFaceDark,shading angle=-12},
  wheel face three/.style={top color=WheelFaceTeal,
    bottom color=WheelFaceLight,shading angle=25},
  wheel panel label/.style={text=WheelInk,
    font=\fontsize{10.23}{15.07}\selectfont}}
\centering
\begin{tikzpicture}[scale=0.7,x=1mm,y=1mm,
  inner xsep=0pt,inner ysep=0pt,outer xsep=0pt,outer ysep=0pt,
  line cap=round,line join=round,
  every label/.append style={text=WheelInk}]
\coordinate (Lv) at (55.00,60.00);
\coordinate (Lu1) at (33.89,76.83);
\coordinate (La) at (55.00,87.00);
\coordinate (Lb) at (76.11,76.83);
\coordinate (Lc) at (81.32,53.99);
\coordinate (Ld) at (66.71,35.67);
\coordinate (Lu2) at (43.29,35.67);
\coordinate (Le) at (28.68,53.99);
\shade[wheel face one] (Lv)--(La)--(Lb)--cycle;
\shade[wheel face two] (Lv)--(Lb)--(Lc)--cycle;
\shade[wheel face three] (Lv)--(Lc)--(Ld)--cycle;
\draw[wheel dashed edge] (Lu1)--(La);
\draw[wheel edge] (La)--(Lb);
\draw[wheel edge] (Lb)--(Lc);
\draw[wheel edge] (Lc)--(Ld);
\draw[wheel dashed edge] (Ld)--(Lu2);
\draw[wheel dashed edge] (Lu2)--(Le);
\draw[wheel dashed edge] (Le)--(Lu1);
\draw[wheel special edge] (Lv)--(Lu1);
\draw[wheel edge] (Lv)--(La);
\draw[wheel edge] (Lv)--(Lb);
\draw[wheel edge] (Lv)--(Lc);
\draw[wheel edge] (Lv)--(Ld);
\draw[wheel special edge] (Lv)--(Lu2);
\draw[wheel dashed edge] (Lv)--(Le);
\node[wheel center vertex,
  label={[label distance=1.2mm]270:$u$}] at (Lv) {};
\node[wheel special vertex,
  label={[label distance=1.2mm]142:$u^*$}] at (Lu1) {};
\node[wheel vertex,
  label={[label distance=1.2mm]90:$a$}] at (La) {};
\node[wheel vertex,
  label={[label distance=1.2mm]39:$b$}] at (Lb) {};
\node[wheel vertex,
  label={[label distance=1.2mm]347:$c$}] at (Lc) {};
\node[wheel vertex,
  label={[label distance=1.2mm]296:$d$}] at (Ld) {};
\node[wheel special vertex,
  label={[label distance=1.2mm]244:$u^{**}$}] at (Lu2) {};
\node[wheel vertex,
  label={[label distance=1.2mm]193:$e$}] at (Le) {};
\node[wheel panel label] at (55.00,22.00)
  {$(a)\quad d_{H^*}(u)-d_{H^{*}[\mathbb{F}]}(u)=1$};
\coordinate (Rv) at (145.00,60.00);
\coordinate (Ru1) at (123.89,76.83);
\coordinate (Ra) at (145.00,87.00);
\coordinate (Rb) at (166.11,76.83);
\coordinate (Rc) at (171.32,53.99);
\coordinate (Ru2) at (156.71,35.67);
\coordinate (Rd) at (133.29,35.67);
\coordinate (Re) at (118.68,53.99);
\shade[wheel face one] (Rv)--(Ra)--(Rb)--cycle;
\shade[wheel face two] (Rv)--(Rb)--(Rc)--cycle;
\shade[wheel face three] (Rv)--(Rd)--(Re)--cycle;
\draw[wheel dashed edge] (Ru1)--(Ra);
\draw[wheel edge] (Ra)--(Rb);
\draw[wheel edge] (Rb)--(Rc);
\draw[wheel dashed edge] (Rc)--(Ru2);
\draw[wheel dashed edge] (Ru2)--(Rd);
\draw[wheel edge] (Rd)--(Re);
\draw[wheel dashed edge] (Re)--(Ru1);
\draw[wheel special edge] (Rv)--(Ru1);
\draw[wheel edge] (Rv)--(Ra);
\draw[wheel edge] (Rv)--(Rb);
\draw[wheel edge] (Rv)--(Rc);
\draw[wheel special edge] (Rv)--(Ru2);
\draw[wheel edge] (Rv)--(Rd);
\draw[wheel edge] (Rv)--(Re);
\node[wheel center vertex,
  label={[label distance=1.2mm]270:$u$}] at (Rv) {};
\node[wheel special vertex,
  label={[label distance=1.2mm]142:$u^*$}] at (Ru1) {};
\node[wheel vertex,
  label={[label distance=1.2mm]90:$a$}] at (Ra) {};
\node[wheel vertex,
  label={[label distance=1.2mm]39:$b$}] at (Rb) {};
\node[wheel vertex,
  label={[label distance=1.2mm]347:$c$}] at (Rc) {};
\node[wheel special vertex,
  label={[label distance=1.2mm]296:$u^{**}$}] at (Ru2) {};
\node[wheel vertex,
  label={[label distance=1.2mm]244:$d$}] at (Rd) {};
\node[wheel vertex,
  label={[label distance=1.2mm]193:$e$}] at (Re) {};
\node[wheel panel label] at (145.00,22.00)
  {$(b)\quad d_{H^*}(u)-d_{H^{*}[\mathbb{F}]}(u)=0$};
\end{tikzpicture}
\caption{Two possible structures of the wheel $W(u)$.}
\label{fig.8}
\end{figure}

Recall that $\mathbb{F}$ is defined to be the set of 3-faces incident to neither $u^*$ nor $u^{**}$ in $\widetilde{G}^*$.
For any $u\in V(H^*)$ and any subset $\mathbb{F}'\subseteq \mathbb{F}$, let
$\mathbb{F}'(u)$ denote the subset of $\mathbb{F}'$ consisting of all 3-faces incident to $u$.
Define $\pi(\mathbb{F}')$ to be the nonincreasing sequence of positive
incidences $|\mathbb{F}'(u)|$ over all vertices $u\in V(H^*)$, with zero entries omitted.

\begin{claim}\label{cl5.11}
If $\mathbb{F}'\neq\varnothing$, then we have $\sum_{u\in V(H^*)}|\mathbb{F}'(u)|^2\leq {|\mathbb{F}'|}^2+5|\mathbb{F}'|$.
Moreover, sharp bounds hold for the following values of $|\mathbb{F}'|$:
\begin{center}
\begin{tabular}{ccl}
\toprule
$|\mathbb{F}'|$&sharp bound for $\sum_{u\in V(H^*)}|\mathbb{F}'(u)|^2$& $\pi(\mathbb{F}')$ in equality case\\
\midrule
2&10&$(2^2,1^2)$\\
4&30&$(4,2^3,1^2),\ (3^3,1^3),\ (3^2,2^3)$\\
6&60&$(4^3,2^3),\ (4^2,3^3,1)$\\
8&96&$(4^6),\ (5^2,4^2,3,2,1)$\\
\bottomrule
\end{tabular}
\end{center}
Suppose $\max_{u\in V(H^*)}|\mathbb{F}'(u)|=|\mathbb{F}'(u_{i^*})|=d$ for some $u_{i^*}\in V(H^*)$,
and let $\mathbb{F}''=\mathbb{F}'\setminus\mathbb{F}'(u_{i^*})$.
Then
\begin{equation}\label{eq39}
\sum_{u\in V(H^*)}\!\!\!\big(\big|\mathbb{F}'(u)\big|^2\!-\big|\mathbb{F}''(u)\big|^2\big)\leq d^2+4d-2+12\big(|\mathbb{F}'|-d\big).
\end{equation}
\end{claim}

\begin{proof}
By the definition of $\mathbb{F}'(u)$, it is easy to see that
$$\sum_{u\in V(H^*)}\!\!\!|\mathbb{F}'(u)|^2=\!\!\sum_{u\in V(H^*)}\!\!\big(\!\!\sum_{\substack{F\in\mathbb{F}'(u)}}\!\!1\big)^2
=\!\!\sum_{u\in V(H^*)}\,\sum_{\substack{F,\,F'\in\mathbb{F}'(u)}}\!\!\!1,$$
where $F, F'\in \mathbb{F}'(u)$ are allowed to be the same 3-face.
Let $|F|$ be the number of vertices incident to $F$.
Then $|F|=3$ for each $F\in\mathbb{F}'$. By double counting, we further obtain
$$\sum_{u\in V(H^*)}\!\!\!|\mathbb{F}'(u)|^2=\!\!\sum_{u\in V(H^*)}\,\sum_{\substack{F,\,F'\in\mathbb{F}'(u)}}\!\!\!1
=\sum_{F\in\mathbb{F}'}\!|F|+\!\!\!\!\sum_{\{F,\,F'\}\subseteq\mathbb{F}'}\!\!2|F\cap F'|
=3k+\!\!\!\!\sum_{\{F,\,F'\}\subseteq\mathbb{F}'}\!\!2|F\cap F'|.$$

For any two distinct 3-faces $F,F'\in \mathbb{F}'$,
we have $|F\cap F'|\leq2$.
By Claim \ref{cl5.6},
for every edge in $E(\widetilde{G}^*)$,
there are exactly two 3-faces incident to it.
Thus, every 3-face in $\mathbb{F}'$ has at most
three members in $\mathbb{F}'$ that share an edge with it.
Let $k:=|\mathbb{F}'|$.
Then $\sum_{\{F,\,F'\}\subseteq\mathbb{F}'}|F\cap F'|\leq{k\choose 2}+\frac32k.$
Therefore, we obtain $\sum_{u\in V(H^*)}|\mathbb{F}'(u)|^2\leq 2\binom{k}{2}+6k=k^2+5k$,
as claimed.

Given $u\in V(H^*)$ with $|\mathbb{F}'(u)|>0$, Claim \ref{cl5.1} asserts that all faces incident to $u$
form a wheel $W(u)$ centered at $u$, with outer $d_{\widetilde{G}^*}(u)$-cycle $C_u$.
Let $C_u[\mathbb{F}']$ denote the subgraph of $C_u$ induced by edges contained in some face of $\mathbb{F}'(u)$.
Equivalently, it can be obtained by deleting $u$ from the subcomplex formed by all faces in \(\mathbb{F}'(u)\).
Clearly, $C_u[\mathbb{F}']$ is a linear forest with $|\mathbb{F}'(u)|$ edges,
since $u^*,u^{**}\in V(C_u)$ but neither is incident to any face in $\mathbb{F}'(u)$
(see Figure \ref{fig.8}).
Hence,
\begin{equation}\label{eq40}
\big|\mathbb{F}'(u)\big|+1\leq \big|C_u[\mathbb{F}']\big|
\leq2\big|\mathbb{F}'(u)\big|.
\end{equation}
Also, for any subset $U\subseteq V(H^*)$, we have
\begin{equation}\label{eq41}
\sum_{u\in U}\big|\mathbb{F}'(u)\big|\leq k+2{|U|\choose 2}.
\end{equation}
Indeed, since every $F\in\mathbb{F}'$ is a 3-face, we have $0\leq|V(F)\cap U|\leq 3$.
This implies that $|V(F)\cap U|\leq1+{|V(F)\cap U|\choose 2}$.
Since each pair of vertices in $U$ is contained in at most two 3-faces of $\mathbb{F}'$,
it follows that $\sum_{F\in\mathbb{F}'}\!{|V(F)\cap U|\choose 2}\leq2{|U|\choose 2}.$
By double counting, we deduce that
$$\sum_{u\in U}\big|\mathbb{F}'(u)\big|=\sum_{F\in\mathbb{F}'}\big|V(F)\cap U\big|\leq
\sum_{F\in\mathbb{F}'}1+2{|U|\choose 2}=k+2{|U|\choose 2}.$$

Recall that
$\mathbb{F}''=\mathbb{F}'\setminus\mathbb{F}'(u_{i^*})$,
where $u_{i^*}\in V(H^*)$ satisfies $|\mathbb{F}'(u_{i^*})|=\max_{u\in V(H^*)}|\mathbb{F}'(u)|=d$.
Denote $a_u:=|\mathbb{F}'(u)\cap\mathbb{F}'(u_{i^*})|$. Then $a_{u_{i^*}}=d$,
and $C_{u_{i^*}}[\mathbb{F}']$ is a linear forest with $d$ edges.
Assume $C_{u_{i^*}}[\mathbb{F}']$ has $c$ connected components. Then $\big|C_{u_{i^*}}[\mathbb{F}']\big|=d+c$
and $\sum_{u\neq u_{i^*}}a_u=2e(C_{u_{i^*}}[\mathbb{F}'])=2d$,
where exactly $2c$ of the $a_u$ equal $1$, $d-c$ of the $a_u$ equal $2$, and $a_u=0$ for all remaining $u$.
It follows that $\sum_{u\in V(H^*)}a_u^2=d^2+2c+4(d-c)=d^2+4d-2c$.
Note that $\sum_{u\in V(H^*)}|\mathbb{F}''(u)|=3|\mathbb{F}''|=3(k-d)$,
and $|\mathbb{F}'(u)|=a_u+|\mathbb{F}''(u)|$ for any $u\in V(H^*)$.
Consequently,
\begin{equation}\label{eq42}
\sum_{u\in V(H^*)}\!\!\!\big(\big|\mathbb{F}'(u)\big|^2\!-\big|\mathbb{F}''(u)\big|^2\big)
=\!\!\sum_{u\in V(H^*)}\!\!\!\big(a_u^2+2a_u|\mathbb F''(u)|\big)
\leq d^2+4d-2c+12(k-d),
\end{equation}
with equality if and only if all vertices $u$ with $|\mathbb F''(u)|>0$ satisfy $a_u=2$.
Thus (\ref{eq39}) follows.

Let $M:=\sum_{u\in V(H^*)}|\mathbb{F}'(u)|^2$.
Then $M\leq d\sum_{u\in V(H^*)}|\mathbb{F}'(u)|=3d|\mathbb{F}'|=3dk$,
where $d\leq k$ and $|\mathbb{F}''|=k-d.$
For $k=2$, the stated bound $M\leq 10$ follows directly from the two inequalities $M\leq 3dk$ and
(\ref{eq42}), with sharpness attained
only when $d=k=2$ and $c=1$.
Hence, sharpness occurs uniquely at $\pi(\mathbb{F}')=(2,2,1,1)$, as claimed.

Now, consider the case $k=4$ and prove that $M\leq30$.
If $d=4$, inequality (\ref{eq42}) gives $M\leq30$,
with equality only when $c=1$ and $\pi(\mathbb{F}')=(4,2^3,1^2)$.
If $d\leq2$, then $M\leq3dk\leq24$.
It remains to consider $d=3$. Now $|\mathbb{F}''|=1$.
Let $F\in\mathbb{F}''$ and $s=\sum_{u\in V(F)}a_u$.
Since $\sum_{u\in V(H^*)}a_u^2=d^2+4d-2c$ and $F$ increases the incidence at each of its three vertices by one,
we obtain $M=\sum_{u\in V(H^*)}a_u^2+2s+3=24-2c+2s.$
If $c\geq2$, at most one vertex $u$ of the three-edge linear forest $C_{u_{i^*}}[\mathbb{F}']$
satisfies $a_u=2$, whence $s\leq4$ and $M\leq28$.
We may thus assume $c=1$ and write $C_{u_{i^*}}[\mathbb{F}']=u_{i_1}u_{i_2}u_{i_3}u_{i_4}$.
Then, $\mathbb{F}'(u_{i^*})=\{u_{i^*}u_{i_j}u_{i_{j+1}}~|~j=1,2,3\}$.
We now have $a_{u_{i_1}}=1$ for $j\in\{1,4\}$, $a_{u_{i_j}}=2$ for $j\in\{2,3\}$, and $a_u=0$ for all other $u$.
Hence, $s\leq5$. If $s=5$, then $F=u_{i_{j-1}}u_{i_{j}}u_{i_{j+1}}$ for $j\in\{2,3\}$,
and $C_{u_{i_j}}[\mathbb{F}']$ contains the 3-cycle $u_{i_{j-1}}u_{i^*}u_{i_{j+1}}u_{i_{j-1}}$,
which contradicts the fact that $C_{u}[\mathbb{F}']$ is a linear forest for any $u\in V(H^*)$ with $|\mathbb{F}'(u)|>0$.
Therefore, $s\leq4$, which gives $M\leq30$.
Equality holds precisely when $s=4$. There are two possible cases.
When the contributions of the three vertices of $F$ to $s$ are $2,2,0$, then
$F=u_{i_2}u_{i_3}u$ for some vertex $u\notin\{u_{i^*},u_{i_1},u_{i_2},u_{i_3},u_{i_4}\}$, and now $\pi(\mathbb{F}')=(3^3,1^3)$.
If the contributions are $2,1,1$, then $F\in\{u_{i_1}u_{i_2}u_{i_4}, u_{i_1}u_{i_3}u_{i_4}\}$, and now $\pi(\mathbb{F'})=(3^2,2^3)$.
The three extremal configurations, up to relabelling, are illustrated in Figure \ref{fig.9}.

\begin{figure}[H]
\centering
\definecolor{ink}{HTML}{263443}
\definecolor{edgeink}{HTML}{405B70}
\definecolor{vertexblue}{HTML}{42627B}
\definecolor{starface}{HTML}{BDD3E5}
\definecolor{extraface}{HTML}{EBC5A2}
\definecolor{facelight}{HTML}{F9FCFE}
\definecolor{warmlight}{HTML}{FFF9F2}
\begin{tikzpicture}[
  x=0.75cm,y=0.75cm,font=\small,text=ink,
  line cap=round,line join=round,
  edge/.style={draw=edgeink,line width=.8pt},
  point/.style={circle,ball color=vertexblue,draw=vertexblue!75!ink,
    line width=.12pt,inner sep=0pt,outer sep=0pt,minimum size=1.6mm},
  lab/.style={font=\small,inner sep=2pt},
  title/.style={font=\bfseries\small},
  faces/.style={font=\footnotesize,align=center},
  blue left/.style={left color=facelight,right color=starface,
    shading angle=15},
  blue right/.style={left color=starface!92!white,right color=facelight,
    shading angle=-15},
  blue middle/.style={top color=facelight,bottom color=starface,
    shading angle=0},
  warm face/.style={top color=warmlight,bottom color=extraface,
    shading angle=10}]
\begin{scope}
  \coordinate (u_{i^*}) at (2.2,3.65);
  \coordinate (u_{i_1}) at (0,2.55);
  \coordinate (u_{i_2}) at (0.15,.85);
  \coordinate (u_{i_3}) at (2.2,0);
  \coordinate (u_{i_4}) at (4.25,0.85);
  \coordinate (u_{i_5}) at (4.4,2.55);
  \shade[blue left]
    (u_{i^*})--(u_{i_1})--(u_{i_2})--cycle;
  \shade[blue middle]
    (u_{i^*})--(u_{i_2})--(u_{i_3})--cycle;
  \shade[top color=facelight,bottom color=starface!85!white,shading angle=-20]
    (u_{i^*})--(u_{i_3})--(u_{i_4})--cycle;
  \shade[blue right]
    (u_{i^*})--(u_{i_4})--(u_{i_5})--cycle;
  \draw[edge] (u_{i_1})--(u_{i_2})--(u_{i_3})--(u_{i_4})--(u_{i_5});
  \foreach \u in {u_{i_1},u_{i_2},u_{i_3},u_{i_4},u_{i_5}}
    \draw[edge] (u_{i^*})--(\u);
  \foreach \u in {u_{i^*},u_{i_1},u_{i_2},u_{i_3},u_{i_4},u_{i_5}} {
    \fill[ink,opacity=.10] ([xshift=.15mm,yshift=-.18mm]\u) circle (.84mm);
    \node[point] at (\u) {};}
  \node[lab,above] at (u_{i^*}) {$u_{i^*}$};
  \node[lab,left] at (u_{i_1}) {$u_{i_1}$};
  \node[lab,left] at (u_{i_2}) {$u_{i_2}$};
  \node[lab,below] at (u_{i_3}) {$u_{i_3}$};
  \node[lab,right] at (u_{i_4}) {$u_{i_4}$};
  \node[lab,right] at (u_{i_5}) {$u_{i_5}$};
  \node at (2.2,-1) {(a) $d=4$, $\pi(\mathbb{F}')=(4,2^3,1^2)$};
\end{scope}
\begin{scope}[xshift=4.8cm]
  \coordinate (u_{i^*}) at (2.2,3.65);
  \coordinate (u_{i_1}) at (-0.15,2.75);
  \coordinate (u_{i_2}) at (.85,1.35);
  \coordinate (u_{i_3}) at (3.55,1.35);
  \coordinate (u_{i_4}) at (4.55,2.75);
  \coordinate (u) at (2.2,0);
  \shade[blue left]
    (u_{i^*})--(u_{i_1})--(u_{i_2})--cycle;
  \shade[blue middle]
    (u_{i^*})--(u_{i_2})--(u_{i_3})--cycle;
  \shade[blue right]
    (u_{i^*})--(u_{i_3})--(u_{i_4})--cycle;
  \shade[warm face]
    (u_{i_2})--(u_{i_3})--(u)--cycle;
  \draw[edge] (u_{i^*})--(u_{i_2})--(u_{i_3})--cycle;
  \draw[edge] (u_{i^*})--(u_{i_1})--(u_{i_2});
  \draw[edge] (u_{i^*})--(u_{i_4})--(u_{i_3});
  \draw[edge] (u_{i_2})--(u)--(u_{i_3});
  \foreach \u in {u_{i^*},u_{i_1},u_{i_2},u_{i_3},u_{i_4},u} {
    \fill[ink,opacity=.10] ([xshift=.15mm,yshift=-.18mm]\u) circle (.84mm);
    \node[point] at (\u) {};}
  \node[lab,above] at (u_{i^*}) {$u_{i^*}$};
  \node[lab,left] at (u_{i_1}) {$u_{i_1}$};
  \node[lab,left] at (u_{i_2}) {$u_{i_2}$};
  \node[lab,right] at (u_{i_3}) {$u_{i_3}$};
  \node[lab,right] at (u_{i_4}) {$u_{i_4}$};
  \node[lab,below] at (u) {$u$};
  \node at (2.2,-1) {(b) $d=3$, $\pi(\mathbb{F}')=(3^3,1^3)$};
\end{scope}
\begin{scope}[xshift=9.6cm]
  \coordinate (u_{i_2}) at (0,0);
  \coordinate (u_{i_3}) at (2.2,3.65);
  \coordinate (u_{i_4}) at (4.4,0);
  \coordinate (u_{i^*}) at (2.2,1.9);
  \coordinate (u_{i_1}) at (2.2,0.58);
  \shade[blue middle]
    (u_{i^*})--(u_{i_1})--(u_{i_2})--cycle;
  \shade[blue left]
    (u_{i^*})--(u_{i_2})--(u_{i_3})--cycle;
  \shade[blue right]
    (u_{i^*})--(u_{i_3})--(u_{i_4})--cycle;
  \shade[warm face]
    (u_{i_1})--(u_{i_2})--(u_{i_4})--cycle;
  \draw[edge] (u_{i_2})--(u_{i_3})--(u_{i_4})--cycle;
  \draw[edge] (u_{i^*})--(u_{i_2}) (u_{i^*})--(u_{i_3}) (u_{i^*})--(u_{i_4});
  \draw[edge] (u_{i_1})--(u_{i^*}) (u_{i_1})--(u_{i_2}) (u_{i_1})--(u_{i_4});
  \foreach \u in {u_{i^*},u_{i_1},u_{i_2},u_{i_3},u_{i_4}} {
    \fill[ink,opacity=.10] ([xshift=.15mm,yshift=-.18mm]\u) circle (.84mm);
    \node[point] at (\u) {};}
  \node[lab,above] at (u_{i_3}) {$u_{i_3}$};
  \node[lab,below left] at (u_{i_2}) {$u_{i_2}$};
  \node[lab,below right] at (u_{i_4}) {$u_{i_4}$};
  \node[lab,above left] at (u_{i^*}) {$u_{i^*}$};
  \node[lab,below] at (u_{i_1}) {$u_{i_1}$};
  \node[font=\tiny,align=center,text=ink!65] at (2.80,0.92)
    {$\notin\mathbb{F}'$};
  \node at (2.2,-1) {(c) $d=3$, $\pi(\mathbb{F}')=(3^2,2^3)$};
\end{scope}
\end{tikzpicture}
\vspace{-2.5mm}
\caption{The three extremal configurations for $k=4$.}
\label{fig.9}
\end{figure}

We now consider the case $k=6$ and prove that $M\leq60$.
Let $q$ denote the number of vertices $u\in V(H^*)$ for which $|\mathbb{F}'(u)|>0$.
Since the linear forest $C_{u_{i^*}}[\mathbb{F}']$ contains $d$ edges and $c$ components,
it has $d+c$ vertices, each satisfying $|\mathbb{F}'(u)|>0$.
Recall that $|\mathbb{F}'(u_{i^*})|=\max_{u\in V(H^*)}|\mathbb{F}'(u)|=d>0$.
Therefore, $q\geq d+c+1$.
When $d\in\{5,6\}$,
inequality \eqref{eq42} gives $M\leq58$. When $d\leq3$, we have $M\leq3dk\leq54$.
It remains the case $d=4$.
Suppose $M\geq60$.
As $\sum_{u\in V(H^*)}|\mathbb{F}'(u)|=3|\mathbb{F}'|=3k=18$,
the positive values $|\mathbb{F}'(u)|$ form a partition of $18$ whose largest part is $d=4$,
with at least $q\geq d+c+1\geq6$ parts, and whose sum of squares is at least $60$.
To enumerate these partitions, let $n_4$ be the number of parts equal to $4$.
If $n_4=4$, the remaining parts must be $(1,1)$, since there are at least six parts in total.
If $n_4=3$, the remaining parts sum to $6$, comprise at least three parts each at most 3,
and their sum of squares is at least 12. The possibilities are $(3,2,1)$,
$(3,1,1,1)$, or $(2,2,2)$.
If $n_4=2$, the remaining parts sum to $10$ and each is at most $3$, so their sum of squares is at most $28$,
with equality attained only for $(3,3,3,1)$.
Finally, if $n_4=1$, then
$M\leq4^2+4\cdot3^2+2^2=56<60$.
The possible $\pi(\mathbb{F}')$ and their associated $q$ and $M$ are given below:
\[
\begin{array}{c|ccccc}
\pi(\mathbb{F}') & (4^4,1^2) & (4^3,3,2,1) & (4^3,3,1^3) & (4^3,2^3) & (4^2,3^3,1) \\ \hline
q & 6 & 6 & 7 & 6 & 6 \\
M & 66 & 62 & 60 & 60 & 60
\end{array}
\]

Let $H'$ be the subcomplex of $\widetilde{G}^*$ formed by all faces in $\mathbb{F}'$.
Then for each $u\in V(H')$, we have $d_{H'}(u)=\big|C_u[\mathbb{F}']\big|$.
In view of (\ref{eq40}), $d_{H'}(u)\geq |\mathbb{F}'(u)|+1$,
and $d_{H'}(u)=2$ when $|\mathbb{F}'(u)|=1$.
This implies that the first two $\pi(\mathbb{F}')$ listed above are impossible.
Indeed, since $|H'|=q=6$ in either case,
every vertex $u$ with $|\mathbb{F}'(u)|=4$ is adjacent to the remaining five vertices of $H'$.
However, any vertex $v$ satisfying $|\mathbb{F}'(v)|=1$ has exactly two neighbors in $H'$,
which contradicts $n_4\geq3$.
Next, suppose that $\pi(\mathbb{F}')=(4^3,3,1^3)$.
Let $V_1=\{u\in V(H')~|~ |\mathbb{F}'(u)|\geq3\}$ and $V_2=V(H')\setminus V_1$.
Then $\sum_{u\in V_1}d_{H'}(u)\geq3\cdot5+4=19$.
On the other hand, $|V_1|=4$, and each of the three
vertices in $V_2$ has degree two in $H'$. Hence,
$$\sum_{u\in V_1}d_{H'}(u)=2e_{H'}(V_1)+e_{H'}(V_1,V_2)
 \leq2\binom{4}{2}+\sum_{u\in V_2}d_{H'}(u)=12+6=18,$$
a contradiction.
Hence, $M\leq60$, and equality holds only when $\pi(\mathbb{F}')=(4^3,2^3)$ or $(4^2,3^3,1)$.

Finally, we consider the case $k=8$ and prove that $M\leq96$.
If $d\in\{6,7,8\}$, inequality \eqref{eq42} gives $M\leq94$.
Recall that $M\leq d\sum_{u\in V(H^*)}|\mathbb{F}'(u)|=3d|\mathbb{F}'|=3dk$.
If $d\leq4$, then $M\leq96$, with equality only when $d=4$ and $\pi(\mathbb{F}')=(4^6)$.
It remains the case $d=5$.
Suppose now that $M\geq96$, and let $n_5$ be the number of vertices $u$ for which $|\mathbb{F}'(u)|=5$.
Then $n_5\leq2$. Indeed, suppose $n_5\geq3$ and $|\mathbb{F}'(u_{i_j})|=5$ for $j\in\{1,2,3\}$,
then $\sum_{j=1}^3|\mathbb{F}'(u_{i_j})|=15>14$, contradicting (\ref{eq41}).
On the other hand, since $\sum_{u\in V(H^*)}|\mathbb{F}'(u)|=3|\mathbb{F}'|=3k=24$,
the positive values $|\mathbb{F}'(u)|$ form a partition of $24$ whose largest part is $d=5$,
with at least $q\geq d+c+1\geq7$ parts, and whose sum of squares is at least $96$.
This forces $n_5\geq2$, as otherwise $M\leq94$.
We now conclude that $n_5=2$, and the only possible partitions satisfying $M\geq96$ are
$(5^2,4^3,1^2)$ and $(5^2,4^2,3,2,1)$.
Let $H'$, $V_1$ and $V_2$ be as defined above.
For $\pi(\mathbb{F}')=(5^2,4^3,1^2)$,
we have $\sum_{u\in V_1}d_{H'}(u)\geq2\cdot6+3\cdot5=27$, while
$$\sum_{u\in V_1}d_{H'}(u)=2e_{H'}(V_1)+e_{H'}(V_1,V_2)
 \leq2\binom{5}{2}+\sum_{u\in V_2}d_{H'}(u)=20+4=24,$$
a contradiction.
Therefore, $M\leq96$, and equality holds only when $\pi(\mathbb{F}')=(4^6)$ or $(5^2,4^2,3,2,1)$.
This completes the proof.
\end{proof}

\section{Proof of Theorem \ref{thm1.2}: The case $1\leq\gamma\leq4$}

The Heffter--Edmonds rotation principle (see \cite{MT}) states that
if $G$ is a connected graph, for which each vertex $u$ admits a cyclic permutation $\pi_u$ of the edges incident with $u$,
then there exists a unique cellular embedding $\widetilde{G}$ of $G$ on a closed connected surface such that
the cyclic order of edges around every $u$ in $\widetilde{G}$ is precisely $\pi_u$.
In particular, we have the following consequence.

\begin{lem}\label{lem6.1}
Let $G$ be a connected graph, and let $\mathbb{T}$ be a
collection of triangles of $G$. Suppose that\vspace{0.5mm}\\
(i) every edge of $G$ belongs to exactly two triangles in $\mathbb{T}$; \vspace{0.5mm}\\
(ii) for every $u\in V(G)$, the neighbors of $u$ admit a cyclic ordering $u_1,\ldots,u_d$ such that the triangles in
$\mathbb{T}$ containing $u$ are precisely $\{u,u_i,u_{i+1}\}$, where $i=1,\ldots,d$ and $u_{d+1}=u_1$.\vspace{0.5mm}\\
Then there exists a closed connected surface with a triangulation
whose vertices and edges are those of $G$ and whose triangular faces
are the triangles in $\mathbb{T}$.
\end{lem}


We emphasize that  $\mathbb{T}$ is merely a selected family of triangles of $G$, and need not contain every triangle of $G$.
The conditions in Lemma \ref{lem6.1} do not prescribe the Euler genus of the resulting surface.
Once $G$ and $\mathbb{T}$ are fixed, Euler's formula gives
$\gamma=e(G)-|G|-|\mathbb{T}|+2$,
which is the Euler genus of the resulting closed surface.

\begin{figure}[H]
\definecolor{figureink}{HTML}{263443}
\definecolor{boundarycolour}{HTML}{AA6249}
\definecolor{corecolour}{HTML}{294D70}
\definecolor{secondarycolour}{HTML}{4B8B85}
\definecolor{corefill}{HTML}{C4D7E7}
\definecolor{secondaryfill}{HTML}{C9DFD8}
\definecolor{figurelight}{HTML}{F9FCFE}
\centering
\begin{tikzpicture}
[scale=0.9, line cap=round, line join=round, every node/.style={text=figureink},
x=1.00mm, y=1.00mm, inner xsep=0pt, inner ysep=0pt, outer xsep=0pt, outer ysep=0pt]
\coordinate (v6) at (-10,75);
\coordinate (v7) at (0,75);
\coordinate (v2) at (-10,90);
\coordinate (v3) at (0,60);
\coordinate (w1) at (80,60);
\coordinate (w2) at (90,60);
\coordinate (w3) at (85,70);
\coordinate (w4) at (85,80);
\coordinate (wa) at (85,90);
\coordinate (w1a) at (110,90);
\coordinate (w2a) at (120,90);
\coordinate (w3a) at (115,80);
\coordinate (w4a) at (115,70);
\coordinate (waa) at (115,60);
\shade[top color=figurelight,bottom color=secondaryfill!60!white]
  (-20,60) rectangle (10,90);
\shade[top color=figurelight,bottom color=corefill!82!white,shading angle=-25]
  (-20,90)--(-10,90)--(-10,75)--cycle;
\shade[top color=figurelight,bottom color=secondaryfill!64!white,shading angle=25]
  (-20,90)--(-20,80)--(-10,75)--cycle;
\shade[top color=figurelight,bottom color=corefill!64!white,shading angle=0]
  (-20,80)--(-20,70)--(-10,75)--cycle;
\shade[top color=figurelight,bottom color=secondaryfill!82!white,shading angle=-25]
  (-20,70)--(-10,75)--(0,60)--cycle;
\shade[top color=figurelight,bottom color=corefill!64!white,shading angle=25]
  (-20,70)--(0,60)--(-10,60)--cycle;
\shade[top color=figurelight,bottom color=secondaryfill!64!white,shading angle=0]
  (-20,70)--(-10,60)--(-20,60)--cycle;
\shade[top color=figurelight,bottom color=corefill!82!white,shading angle=-25]
  (-10,90)--(0,90)--(10,80)--cycle;
\shade[top color=figurelight,bottom color=secondaryfill!64!white,shading angle=25]
  (0,90)--(10,90)--(10,80)--cycle;
\shade[top color=figurelight,bottom color=corefill!64!white,shading angle=0]
  (-10,90)--(10,80)--(0,75)--cycle;
\shade[top color=figurelight,bottom color=secondaryfill!82!white,shading angle=-25]
  (10,80)--(10,70)--(0,75)--cycle;
\shade[top color=figurelight,bottom color=corefill!64!white,shading angle=25]
  (10,70)--(10,60)--(0,75)--cycle;
\shade[top color=figurelight,bottom color=secondaryfill!64!white,shading angle=0]
  (0,75)--(10,60)--(0,60)--cycle;

\shade[top color=figurelight,bottom color=corefill,shading angle=-15]
  (-10,75)--(-10,90)--(0,75)--cycle;
\shade[top color=figurelight,bottom color=secondaryfill,shading angle=20]
  (-10,75)--(0,75)--(0,60)--cycle;
\shade[inner color=figurelight,outer color=secondaryfill!65!white]
  (43.95,75) circle (15mm);
\shade[top color=figurelight,bottom color=corefill!70!white]
  (43.95,81)--(43.95,90)
  arc[start angle=90,end angle=140,radius=15mm]--cycle;
\shade[top color=figurelight,bottom color=secondaryfill!78!white]
  (43.95,81)--(43.95,90)
  arc[start angle=90,end angle=40,radius=15mm]--cycle;
\shade[left color=figurelight,right color=corefill!75!white]
  (38.75,72)--(32.4593,84.6418)
  arc[start angle=140,end angle=220,radius=15mm]--cycle;
\shade[top color=figurelight,bottom color=secondaryfill!85!white]
  (38.75,72)--(32.4593,65.3582)
  arc[start angle=220,end angle=270,radius=15mm]--cycle;
\shade[left color=figurelight,right color=corefill!80!white]
  (49.20,72)--(43.95,60)
  arc[start angle=270,end angle=320,radius=15mm]--cycle;
\shade[left color=secondaryfill!82!white,right color=figurelight]
  (49.20,72)--(55.4407,65.3582)
  arc[start angle=320,end angle=400,radius=15mm]--cycle;
\shade[top color=figurelight,bottom color=secondaryfill!85!white]
  (32.4593,84.6418)--(43.95,81)--(38.75,72)--cycle;
\shade[top color=figurelight,bottom color=corefill!85!white]
  (43.95,81)--(55.4407,84.6418)--(49.20,72)--cycle;
\shade[top color=figurelight,bottom color=corefill,shading angle=-15]
  (38.75,72)--(43.95,81)--(49.20,72)--cycle;
\shade[top color=figurelight,bottom color=secondaryfill,shading angle=20]
  (38.75,72)--(49.20,72)--(43.95,60)--cycle;
\shade[top color=figurelight,bottom color=corefill,shading angle=15]
  (80,60)--(90,60)--(85,90)--cycle;
\shade[left color=figurelight,right color=corefill!85!white]
  (80,60)--(85,80)--(85,90)--cycle;
\shade[left color=corefill,right color=figurelight]
  (90,60)--(85,90)--(85,80)--cycle;
\shade[top color=figurelight,bottom color=corefill!90!white]
  (80,60)--(90,60)--(85,70)--cycle;
\shade[top color=secondaryfill,bottom color=figurelight,shading angle=-15]
  (110,90)--(120,90)--(115,60)--cycle;
\shade[left color=figurelight,right color=secondaryfill!85!white]
  (110,90)--(115,60)--(115,70)--cycle;
\shade[left color=secondaryfill,right color=figurelight]
  (120,90)--(115,70)--(115,60)--cycle;
\shade[top color=figurelight,bottom color=secondaryfill]
  (110,90)--(120,90)--(115,80)--cycle;
\path[line width=0.82pt, draw=secondarycolour] (-20.00,90.00) -- (-10.00,75.00);
\path[line width=0.82pt, draw=secondarycolour] (-10.00,90.00) -- (-10.00,75.00);
\path[line width=0.82pt, draw=secondarycolour] (-20.00,80.00) -- (-10.00,75.00);
\path[line width=1.05pt, draw=boundarycolour] (-20.00,90.00) -- (-20.00,60.00);
\path[line width=1.05pt, draw=boundarycolour] (-20.00,90.00) -- (10.00,90.00);
\path[line width=1.05pt, draw=boundarycolour] (10.00,60.00) -- (10.00,90.00);
\path[line width=1.05pt, draw=boundarycolour] (10.00,60.00) -- (-20.00,60.00);
\path[line width=0.82pt, draw=secondarycolour] (-20.00,70.00) -- (-10.00,75.00);
\path[line width=0.82pt, draw=secondarycolour] (0.00,75.00) -- (-10.00,75.00);
\path[line width=0.82pt, draw=secondarycolour] (-10.00,90.00) -- (0.00,75.00);
\path[line width=0.82pt, draw=secondarycolour] (-10.00,90.00) -- (10.00,80.00);
\path[line width=0.82pt, draw=secondarycolour] (0.00,75.00) -- (10.00,80.00);
\path[line width=0.82pt, draw=secondarycolour] (0.00,90.00) -- (10.00,80.00);
\path[line width=0.82pt, draw=secondarycolour] (-20.00,70.00) -- (-10.00,60.00);
\path[line width=0.82pt, draw=secondarycolour] (-20.00,70.00) -- (0.00,60.00);
\path[line width=0.82pt, draw=secondarycolour] (0.00,60.00) -- (-10.00,75.00);
\path[line width=0.82pt, draw=secondarycolour] (0.00,75.00) -- (0.00,60.00);
\path[line width=0.82pt, draw=secondarycolour] (0.00,75.00) -- (10.00,60.00);
\path[line width=0.82pt, draw=secondarycolour] (0.00,75.00) -- (10.00,70.00);
\draw[line width=1.15pt,color=corecolour]  (v6) -- (v7);
\draw[line width=1.15pt,color=corecolour]  (v6) -- (v3);
\draw[line width=1.15pt,color=corecolour]  (v6) -- (v2);
\draw[line width=1.15pt,color=corecolour]  (v7) -- (v3);
\draw[line width=1.15pt,color=corecolour]  (v7) -- (v2);
\path[line width=1.05pt, draw=boundarycolour] (43.95,75.00) circle (15.00mm);
\path[line width=1.15pt, draw=corecolour] (38.75,72.00) -- (43.95,60.00);
\path[line width=1.15pt, draw=corecolour] (43.95,60.00) -- (49.20,72.00);
\path[line width=1.15pt, draw=corecolour] (38.75,72.00) -- (49.20,72.00);
\path[line width=1.15pt, draw=corecolour] (38.75,72.00) -- (43.95,81.00);
\path[line width=1.15pt, draw=corecolour] (49.20,72.00) -- (43.95,81.00);
\path[line width=0.82pt, draw=secondarycolour] (43.95,90.00) -- (43.95,81.00);
\path[line width=0.82pt, draw=secondarycolour] (32.4593,84.6418) -- (43.95,81.00);
\path[line width=0.82pt, draw=secondarycolour] (55.4407,84.6418)-- (43.95,81.00);
\path[line width=0.82pt, draw=secondarycolour] (38.75,72.00) -- (32.4593,84.6418);
\path[line width=0.82pt, draw=secondarycolour] (38.75,72.00) -- (32.4593,65.3582);
\path[line width=0.82pt, draw=secondarycolour] (55.4407,65.3582) -- (49.20,72.00);
\path[line width=0.82pt, draw=secondarycolour] (55.4407,84.6418) -- (49.20,72.00);
\fill[secondarycolour!85!figureink] (85,75) circle (0.21mm);
\fill[secondarycolour!85!figureink] (85,73) circle (0.21mm);
\fill[secondarycolour!85!figureink] (85,77) circle (0.21mm);
\draw[line width=1.15pt,color=corecolour]  (w1) -- (w2);
\draw[line width=1.15pt,color=corecolour]  (w2) -- (wa);
\draw[line width=1.15pt,color=corecolour]  (w1) -- (wa);
\draw[line width=0.82pt,color=secondarycolour]  (w1) -- (w3);
\draw[line width=0.82pt,color=secondarycolour]  (w2) -- (w3);
\draw[line width=0.82pt,color=secondarycolour]  (wa) -- (w4);
\draw[line width=0.82pt,color=secondarycolour]  (w1) -- (w4);
\draw[line width=0.82pt,color=secondarycolour]  (w2) -- (w4);
\fill[secondarycolour!85!figureink] (115,75) circle (0.21mm);
\fill[secondarycolour!85!figureink] (115,73) circle (0.21mm);
\fill[secondarycolour!85!figureink] (115,77) circle (0.21mm);
\draw[line width=1.15pt,color=corecolour]  (w1a) -- (w2a);
\draw[line width=1.15pt,color=corecolour]  (w2a) -- (waa);
\draw[line width=1.15pt,color=corecolour]  (w1a) -- (waa);
\draw[line width=0.82pt,color=secondarycolour]  (w1a) -- (w3a);
\draw[line width=0.82pt,color=secondarycolour]  (w2a) -- (w3a);
\draw[line width=0.82pt,color=secondarycolour]  (waa) -- (w4a);
\draw[line width=0.82pt,color=secondarycolour]  (w1a) -- (w4a);
\draw[line width=0.82pt,color=secondarycolour]  (w2a) -- (w4a);
\foreach \x/\y in {-20/90,-20/80,-20/70,-20/60,-10/90,0/90,10/90,10/80,10/70,10/60,-10/60,0/60,43.95/60,43.95/90,32.4593/84.6418,55.4407/84.6418,32.4593/65.3582,55.4407/65.3582}{
  \fill[figureink,opacity=.10] ([xshift=.17mm,yshift=-.20mm]\x,\y) circle (.76mm);
  \shade[ball color=boundarycolour,draw=boundarycolour!75!figureink,line width=.12pt] (\x,\y) circle (.70mm);}
\foreach \x/\y in {-10/75,0/75,38.75/72,49.2/72,43.95/81}{
  \fill[figureink,opacity=.10] ([xshift=.17mm,yshift=-.20mm]\x,\y) circle (.76mm);
  \shade[ball color=corecolour,draw=corecolour!75!figureink,line width=.12pt] (\x,\y) circle (.70mm);}
\foreach \x/\y in {80/60,90/60,85/70,85/80,85/90,110/90,120/90,115/80,115/70,115/60}{
  \fill[figureink,opacity=.10] ([xshift=.17mm,yshift=-.20mm]\x,\y) circle (.76mm);
  \shade[ball color=secondarycolour,draw=secondarycolour!75!figureink,line width=.12pt] (\x,\y) circle (.70mm);}
\draw(-24.5,92.3) node[anchor=base west]{\fontsize{10}{12}\selectfont $v_1$};
\draw(-24.5,55.98) node[anchor=base west]{\fontsize{10}{12}\selectfont $v_1$};
\draw(-12.2,92.3) node[anchor=base west]{\fontsize{10}{12}\selectfont $v_2$};
\draw(-12.2,55.98) node[anchor=base west]{\fontsize{10}{12}\selectfont $v_2$};
\draw(-2.2,92.3) node[anchor=base west]{\fontsize{10}{12}\selectfont $v_3$};
\draw(-2.2,55.98) node[anchor=base west]{\fontsize{10}{12}\selectfont $v_3$};
\draw(10.47,55.98) node[anchor=base west]{\fontsize{10}{12}\selectfont $v_1$};
\draw(10.47,92.3) node[anchor=base west]{\fontsize{10}{12}\selectfont $v_1$};
\draw(-25.5,79.5) node[anchor=base west]{\fontsize{10}{12}\selectfont $v_4$};
\draw(-25.5,69) node[anchor=base west]{\fontsize{10}{12}\selectfont $v_5$};
\draw(11,79.5) node[anchor=base west]{\fontsize{10}{12}\selectfont $v_4$};
\draw(11,69) node[anchor=base west]{\fontsize{10}{12}\selectfont $v_5$};
\draw(-13.5,70.98) node[anchor=base west]{\fontsize{10}{12}\selectfont $u^*$};
\draw(0,77) node[anchor=base west]{\fontsize{10}{12}\selectfont $u^{**}$};
\draw(45,83.00) node[anchor=base west]{\fontsize{10}{12}\selectfont $v_1$};
\draw(34,71.00) node[anchor=base west]{\fontsize{10}{12}\selectfont $u^*$};
\draw(51,71.00) node[anchor=base west]{\fontsize{10}{12}\selectfont $u^{**}$};
\draw(42,92.00) node[anchor=base west]{\fontsize{10}{12}\selectfont $v_2$};
\draw(42,56.00) node[anchor=base west]{\fontsize{10}{12}\selectfont $v_2$};
\draw(27.53,85.57) node[anchor=base west]{\fontsize{10}{12}\selectfont $v_3$};
\draw(27.53,63.25) node[anchor=base west]{\fontsize{10}{12}\selectfont $v_4$};
\draw(56.42,85.57) node[anchor=base west]{\fontsize{10}{12}\selectfont $v_4$};
\draw(56.72,63.25) node[anchor=base west]{\fontsize{10}{12}\selectfont $v_3$};
\draw(80,57) node[anchor=center]{\fontsize{10}{12}\selectfont $u^*$};
\draw(90,57) node[anchor=center]{\fontsize{10}{12}\selectfont $u^{**}$};
\draw(85,65) node[anchor=center]{\fontsize{10}{12}\selectfont $w_1$};
\draw(85,92.3) node[anchor=center]{\fontsize{10}{12}\selectfont $w_{a\!+\!1}$};
\draw(110,92.5) node[anchor=center]{\fontsize{10}{12}\selectfont $u^*$};
\draw(120,92.5) node[anchor=center]{\fontsize{10}{12}\selectfont $u^{**}$};
\draw(115,85) node[anchor=center]{\fontsize{10}{12}\selectfont $u_1$};
\draw(115,57) node[anchor=center]{\fontsize{10}{12}\selectfont $u_{b\!+\!1}$};
\draw(-8,77) node[anchor=base west]{\fontsize{10}{12}\selectfont $H_1$};
\draw(-5.2,69) node[anchor=base west]{\fontsize{10}{12}\selectfont $H_2$};
\draw(41.5,74) node[anchor=base west]{\fontsize{10}{12}\selectfont $H_1$};
\draw(41.5,67) node[anchor=base west]{\fontsize{10}{12}\selectfont $H_2$};
\node[anchor=base,font=\normalsize] at (-5,48) {$G_1$};
\node[anchor=base,font=\normalsize] at (43.95,48) {$G_2$};
\node[anchor=base,font=\normalsize] at (85,48) {$H_1$};
\node[anchor=base,font=\normalsize] at (115,48) {$H_2$};
\end{tikzpicture}
\caption{A triangular embedding of $K_2\nabla H_{0,\gamma}(a,b)$, where $\gamma\in\{1,2\}$.}
\label{fig.10}
\end{figure}

\begin{lem}\label{lem6.2}
Let $H_{0,\gamma}$ be given in Definition \ref{def1.1}. For $1\leq\gamma\leq4$ and $a,b\geq1$,
the graph $K_2\nabla H_{0,\gamma}(a,b)$ admits a triangular embedding on a surface of Euler genus $\gamma$.
\end{lem}

\begin{proof}
Write $K_2=\{u^*,u^{**}\}$. We first treat the cases $\gamma\in\{1,2\}$,
embedding $K_2\nabla H_{0,\gamma}(a,b)$ on the projective plane and the torus, respectively.
By Definition \ref{def1.1}, the base graph $H_{0,\gamma}$ is isomorphic to $K_{\gamma+3}$ for $\gamma\in\{1,2\}$.
Let $H_1=K_2\nabla P_{a+1}$ and $H_2=K_2\nabla P_{b+1}$ be two plane graphs.
For $\gamma=1$,
the embedding of $K_2\nabla K_{\gamma+3}$ on the projective plane depicted in Figure \ref{fig.2} ($b$) is cellular and triangular.
Let $G_1$ be the graph obtained from this embedding by inserting $H_1$ into the 3-face $u^*u^{**}v_1$
and $H_2$ into the 3-face $u^*u^{**}v_2$,
with $v_1=w_{a+1}$ and $v_2=u_{b+1}$ (see Figure \ref{fig.10} ($b$)).
Thus we obtain a triangular embedding of $G_1$ on the projective plane.
Furthermore, one checks that $G_1\cong K_2\nabla H_{0,1}(a,b)$.

For $\gamma=2$,
the embedding of $K_2\nabla K_{\gamma+3}$ on the torus shown in Figure \ref{fig.2} ($a$) is also cellular and triangular.
Let $G_2$ be the graph obtained from this embedding by inserting
$H_1$ into the 3-face $u^*u^{**}v_2$ and $H_2$ into the 3-face $u^*u^{**}v_3$,
with $v_2=w_{a+1}$ and $v_3=u_{b+1}$ (see Figure \ref{fig.10} ($a$)).
Similarly, we obtain a triangular embedding of $G_2$ on the torus,
and $G_2\cong K_2\nabla H_{0,2}(a,b)$.

Now let $\gamma\in\{3,4\}$.
As described in Definition \ref{def1.1}, let $v_0,v_1$ be the attachment vertices of $H_{0,\gamma}$.
Denote the two paths attached to $H_{0,\gamma}$ by $v_0x_1x_2\cdots x_a$ and $v_1y_1y_2\cdots y_b$,
where $x_0:=v_0$ and $y_0:=v_1$.
For $\gamma=3$, choose a family of triangles
\[\mathbb{T}_3=\{
v_0v_1v_3,\,
v_0v_1v_4,\,
v_0v_2v_3,\,
v_1v_2v_4,\,
v_2v_3v_4,\,
v_3v_4v_5\},\]
and let $P_3$ and $Q_3$ be the paths in $H_{0,3}$ with vertex sequences $v_1v_2v_5v_3v_6v_4v_0$ and
$v_0v_2v_5v_4v_6v_3v_1$, respectively.
For $\gamma=4$, choose a family of triangles
\[\mathbb{T}_4=\{
v_0v_1v_2,\,
v_0v_1v_3,\,
v_0v_2v_4,\,
v_0v_3v_5,\,
v_1v_2v_5,\,
v_1v_4v_5,\,
v_2v_3v_4,\,
v_3v_4v_5\},\]
and let $P_4$ and $Q_4$ be the paths in $H_{0,4}$ with vertex sequences
$v_1v_3v_7v_5v_2v_6v_4v_0$ and
$v_0v_5v_7v_3v_2v_6v_4v_1$, respectively.
Furthermore, let $\mathbb{T}_\gamma(a,b)$ consist of the following triangles:

(i) all triangles in $\mathbb{T}_\gamma$;

(ii) $u^*v_iv_j$ for each edge $v_iv_j\in E(P_\gamma)$,
and $u^{**}v_iv_j$ for each edge $v_iv_j\in E(Q_\gamma)$;

(iii) $u^*x_{i-1}x_i$ and $u^{**}x_{i-1}x_i$ for $1\le i\le a$, together with $u^*u^{**}x_a$;

(iv) $u^*y_{j-1}y_j$ and $u^{**}y_{j-1}y_j$ for $1\le j\le b$, together with $u^*u^{**}y_b$.

\vspace{1mm}
We now verify the two conditions in Lemma~\ref{lem6.1}.
A direct check of the listed 3-cycles shows that
every edge of $K_2\nabla H_{0,\gamma}(a,b)$ belongs to exactly two triangles in $\mathbb{T}_\gamma(a,b)$.
The cyclic ordering of the neighbors of $u^*$ is obtained by traversing
$P_\gamma$ from $v_1$ to $v_0$, then proceeding through
$x_1,\ldots,x_a,u^{**},y_b,\ldots,y_1$, and back to $v_1$.
Similarly, the cyclic ordering at $u^{**}$ follows $Q_\gamma$ from
$v_0$ to $v_1$, then continues through
$y_1,\ldots,y_b,u^*,x_a,\ldots,x_1$, and returns to $v_0$.
For $1\le i\le a-1$, the cyclic ordering of the neighbors of $x_i$ is $x_{i-1}u^*x_{i+1}u^{**}x_{i-1}$,
and for $i=a$, it is $x_{a-1}u^*u^{**}x_{a-1}$.
The cyclic ordering of the neighbours of each $y_i$ with $1\le i\le b$ is analogous.
It remains to verify the cyclic-order condition at each vertex of the base graph $H_{0,\gamma}$.
For $\gamma=3$, the
following cyclic orderings of the neighbors of $v_i$ work:
\[
\begin{aligned}
 v_0&: v_1v_3v_2u^{**}x_1u^*v_4v_1, ~~~~~~v_1: v_0v_3u^{**}y_1u^*v_2v_4v_0,\\
 v_2&: v_0v_3v_4v_1u^*v_5u^{**}v_0, ~~~~~~v_3: v_0v_1u^{**}v_6u^*v_5v_4v_2v_0,\\
 v_4&: v_0v_1v_2v_3v_5u^{**}v_6u^*v_0, ~~~v_5: v_3v_4u^{**}v_2u^*v_3, ~~~v_6: u^*v_3u^{**}v_4u^*.
\end{aligned}
\]
For $\gamma=4$, one may take
\[
\begin{aligned}
 v_0&: v_1v_2v_4u^*x_1u^{**}v_5v_3v_1, ~~~~~~v_1: v_0v_2v_5v_4u^{**}y_1u^*v_3v_0,\\
 v_2&: v_0v_1v_5u^*v_6u^{**}v_3v_4v_0, ~~~~~~v_3: v_0v_1u^*v_7u^{**}v_2v_4v_5v_0\\
 v_4&: v_0v_2v_3v_5v_1u^{**}v_6u^*v_0, ~~~~~~v_5: v_0v_3v_4v_1v_2u^*v_7u^{**}v_0,\\
 v_6&: u^*v_2u^{**}v_4u^*, ~~~~~~~~~~~~~~~~~~~ v_7: u^*v_3u^{**}v_5u^*.
\end{aligned}
\]

In each displayed row, two neighbors of $v_i$ are consecutive
exactly when they appear together with $v_i$ in a triangle of
$\mathbb{T}_\gamma(a,b)$. These triangles come from
$\mathbb{T}_\gamma$, from consecutive pairs in $P_\gamma$ and
$Q_\gamma$, and, for $v_0$ and $v_1$, from the corresponding
path triangles. Hence, the triangles in $\mathbb{T}_\gamma(a,b)$ containing
$v_i$ are precisely those determined by consecutive neighbors in the
listed cyclic ordering.
Therefore, every vertex of $K_2\nabla H_{0,\gamma}(a,b)$ satisfies the cyclic-order
condition in Lemma~\ref{lem6.1}.

By Lemma~\ref{lem6.1}, $K_2\nabla H_{0,\gamma}(a,b)$
has a triangular embedding on a closed connected surface $\Sigma$ with Euler genus $\gamma^*$.
For $\gamma=3$, the numbers of vertices, edges, and faces are
$|V|=9+a+b, |E|=30+3(a+b)$ and $|\mathbb{T}_\gamma|=20+2(a+b).$
Thus, Euler’s formula gives $\gamma^*=|E|-|V|-|\mathbb{T}_\gamma|+2=3$.
For $\gamma=4$, these numbers are
$|V|=10+a+b,$ $|E|=36+3(a+b)$ and $|\mathbb{T}_\gamma|=24+2(a+b)$, respectively.
Hence, Euler’s formula gives $\gamma^*=4$.
This completes the proof.
\end{proof}

Let $G^*\in SPEX(n,\gamma)$.
By Theorem \ref{thm4.1}, we have $G^*\cong K_2\nabla H^*$,
where $H^*$ is obtained from $K_2\nabla P_{n-2}$ by adding $3\gamma$ edges.
By Claim \ref{cl5.5}, we further know that $H^*\cong H_0(a,b)$,
where $|a-b|\leq1$ and the base graph $H_0$ is of order at most $9\gamma$.

\begin{lem}\label{lem6.3}
For $1\leq\gamma\leq4$, the base graph $H_0$, together with its two attachment vertices, is isomorphic
to $H_{0,\gamma}$ from Definition~\ref{def1.1}.
\end{lem}

\begin{proof}
Let $G:=K_2\nabla H_{0,\gamma}(a,b)$,
where $V(K_2)=\{u^*,u^{**}\}$, $|a-b|\leq1$ and $a+b+|H_{0,\gamma}|=n-2$, as specified in Theorem~\ref{thm1.2}.
Lemma~\ref{lem6.2} provides a triangular embedding $\widetilde{G}$ of $G$ on a surface of Euler genus $\gamma$.
Thus, $K_2\nabla H_{0,\gamma}(a,b)\in \mathbb{G}(n,\gamma)$.
Let $\mathcal{F}$ denote the set of 3-faces of $\widetilde{G}$ incident to neither $u^*$ nor $u^{**}$,
and let $\mathcal{F}(u)$ be the set of 3-faces in $\mathcal{F}$ incident to a vertex $u$.
Observe that every 3-face of $\widetilde{G}$ incident to an edge on the two pendent paths of lengths $a$ and $b$ in $H_{0,\gamma}(a,b)$
is incident to both $u^*$ and $u^{**}$.
Thus, $|\mathcal{F}(u)|=0$ for each $u\in V(H_{0,\gamma}(a,b))\setminus V(H_{0,\gamma}).$
Hence, one computes
$\sum_{u\in V(H_{0,\gamma}(a,b))}{|\mathcal{F}(u)|}^2
=\sum_{u\in V(H_{0,\gamma})}{|\mathcal{F}(u)|}^2=10,30,60,96$ for $\gamma=1,2,3,4$, respectively.

On the other hand, consider the extremal graph $G^*$.
Recall that $\mathcal{F}$ denotes the set of 3-faces of $\widetilde{G}^*$ incident to neither $u^*$ nor $u^{**}$,
and $\mathbb{F}(u)$ is the set of 3-faces in $\mathbb{F}$ incident to a vertex $u$.
Now, take $\mathbb{F}'=\mathbb{F}$ in Claim \ref{cl5.11}.
By Claim \ref{cl5.7}, $|\mathbb{F}|=2\gamma$.
Thus, for $\gamma=1,2,3,4$, the corresponding values of $|\mathbb{F}|$ are $2,4,6,8$, respectively.
Moreover, by the table in Claim \ref{cl5.11}, the corresponding upper bounds
for $\sum_{u\in V(H^*)}{|\mathbb{F}(u)|}^2$ are $10,30,60,96$, respectively.

By Remark \ref{rem1.2},
we know that $H_{0,\gamma}(a,b)$ is constructed from a spanning path $P_{n-2}$ by adding $3\gamma$ edges.
Thus, $e(H^*)=e(H_{0,\gamma}(a,b))$, i.e., $w^{(1)}(H^*)=w^{(1)}(H_{0,\gamma}(a,b))$.
By Lemma~\ref{lem5.2}, $w^{(2)}(H^*)\geq w^{(2)}(H_{0,\gamma}(a,b))$.
This, together with Claim \ref{cl5.9}, yields
$\sum_{u\in V(H^*)}{|\mathbb{F}(u)|}^2\geq\sum_{u\in V(H_{0,\gamma}(a,b))}{|\mathcal{F}(u)|}^2.$
Hence, for $\gamma=1,2,3,4$, the sum $\sum_{u\in V(H^*)}{|\mathbb{F}(u)|}^2$ equals $10,30,60,96$, respectively.
Therefore, $w^{(2)}(H^*)=w^{(2)}(H_{0,\gamma}(a,b))$ for $1\leq \gamma\leq4$.
We next employ two claims to characterize $H^*$ for the equality cases in the table of Claim \ref{cl5.11}.

\begin{table}[htbp]
\centering
\label{tab:small}
\begin{tabular}{cccc}
\toprule
$\gamma$&$\pi(\mathbb{F})$&$\sum_{u\in V(H^*)}|\mathbb{F}(u)|^2$
&upper bound for $\sum_{uv\in E(H^*)}|\mathbb{F}(u)||\mathbb{F}(v)|$\\
\midrule
1&$(2^2,1^2)$&10&$13$\\
2&$(4,2^3,1^2)$&30&$51$\\
2&$(3^3,1^3)$&30&$54$\\
2&$(3^2,2^3)$&30&$57$\\
3&$(4^3,2^3)$&60&$124$\\
3&$(4^2,3^3,1)$&60&$126$\\
4&$(5^2,4^2,3,2,1)$&96&$225$\\
4&$(4^6)$&96&$240$\\
\bottomrule
\end{tabular}
\end{table}

\begin{claim}\label{cl6.1}
For $1\leq \gamma\leq4$,
the maximum of $w^{(3)}(H^*)$ can be attained by the base graph $H_{0,\gamma}$.
\end{claim}

\begin{proof}
We first establish the following table.
By Claim \ref{cl5.11} and the preceding discussion, we have obtained its first three columns.
By Part (ii) of Claim \ref{cl5.9}, for all $\gamma\geq1$,
$$w^{(3)}(H^*)=8n+72\gamma-32+4\sum_{u\in V(H^*)}|\mathbb{F}(u)|^2+2\sum_{uv\in E(H^*)}|\mathbb{F}(u)||\mathbb{F}(v)|.$$
We now derive the upper bounds for $\sum_{uv\in E(H^*)}|\mathbb{F}(u)||\mathbb{F}(v)|$ given in the fourth column of the table.
Once this is done, the upper bounds of $W^{(3)}(H^*)$ follow immediately.

By Claim \ref{cl5.8},
$|\mathbb{F}(u)|=d_{H^*}(u)-2\geq1$ for each $u\in V^*$,
while $|\mathbb{F}(u)|=0$ for all $u\in V(H^*)\setminus V^*$,
where $V^*$ denotes the set of forks in $H^*$.
For $\gamma\in\{1,2\}$ and the rows $(2^2,1^2)$ and $(3^2,2^3)$, summing over all unordered pairs of forks yields
$\sum_{uv\in E(H^*)}|\mathbb{F}(u)||\mathbb{F}(v)|\leq13$ and $\leq57$, respectively.
Equality holds if and only if $V^*$ is a clique.

For $\gamma=2$ and the row $(4,2^3,1^2)$,
we have $\sum_{uv\in E(H^*)}|\mathbb{F}(u)||\mathbb{F}(v)|\leq57$ by summing over all unordered pairs of forks.
However, every fork $u$ with $|\mathbb{F}(u)|=2$ satisfies $d_{H^*}(u)=|\mathbb{F}(u)|+2=4$,
and thus misses at least one of the other five forks.
Therefore, each non-edge among forks decreases the sum, yielding
$\sum_{uv\in E(H^*)}|\mathbb{F}(u)||\mathbb{F}(v)|\leq57-6=51.$

For $\gamma=2$ and the row $(3^3,1^3)$,
we have $d_{H^*}(u)=|\mathbb{F}(u)|+2=5$ for each fork $u$ with $|\mathbb{F}(u)|=3$,
and $d_{H^*}(u)=3$ for each fork $u$ with $|\mathbb{F}(u)|=1$.
Hence, the maximum of $\sum_{uv\in E(H^*)}|\mathbb{F}(u)||\mathbb{F}(v)|$ is attained when
$V^*$ induces a copy of $K_3\nabla 3K_1$,
implying $\sum_{uv\in E(H^*)}|\mathbb{F}(u)||\mathbb{F}(v)|\leq 54$.

For $\gamma=3$ and the row $(4^3,2^3)$, we have $d_{H^*}(u)=|\mathbb{F}(u)|+2=6$ for each fork $u$ with $|\mathbb{F}(u)|=4$,
and $d_{H^*}(u)=4$ for each fork $u$ with $|\mathbb{F}(u)|=2$.
Hence, the maximum of $\sum_{uv\in E(H^*)}|\mathbb{F}(u)||\mathbb{F}(v)|$ is attained when
$V^*$ induces a copy of $K_3\nabla (K_2\cup K_1)$,
which implies that $\sum_{uv\in E(H^*)}|\mathbb{F}(u)||\mathbb{F}(v)|\leq3\cdot16+9\cdot8+4=124.$

For $\gamma=3$ and the row $(4^2,3^3,1)$, there are two forks of degree $6$, three forks of degree $5$ and one fork of degree $3$.
Hence, the maximum of $\sum_{uv\in E(H^*)}|\mathbb{F}(u)||\mathbb{F}(v)|$ is attained when
the five forks of degree at least $5$ form a clique
and the remaining fork is adjacent to two forks of degree $6$ and one fork of degree $5$.
Thus,
$\sum_{uv\in E(H^*)}|\mathbb{F}(u)||\mathbb{F}(v)|\leq16+6\cdot12+3\cdot9+2\cdot4+3=126.$

For $\gamma=4$ and the row $(5^2,4^2,3,2,1)$,
the maximum of $\sum_{uv\in E(H^*)}|\mathbb{F}(u)||\mathbb{F}(v)|$ is attained when
the five forks of degree $7$, $6$ or $5$ form a clique,
the fork of degree $4$ is adjacent to two forks of degree $7$ and two forks of degree $6$,
and the remaining fork of degree $3$ is adjacent to two forks of degree $7$ and one fork of degree $6$.
A straightforward calculation gives
$\sum_{uv\in E(H^*)}|\mathbb{F}(u)||\mathbb{F}(v)|\leq225.$

Finally, consider $\gamma=4$ and the row $(4^6)$. Now, $|V^*|=6$ and each fork has degree $6$ in $H^*$.
In view of (\ref{eq40}), $d_{V^*}(u)=|C_u[\mathbb{F}]]|\geq|\mathbb{F}(u)|+1=5$ for each $u\in V^*$.
This forces $V^*$ to be a 6-clique, and thus $\sum_{uv\in E(H^*)}|\mathbb{F}(u)||\mathbb{F}(v)|=15\cdot 16=240.$

On the other hand, from the proof of Claim \ref{cl6.1},
we know that $|\mathcal{F}(u)|=0$ for every vertex $u\in V(H_{0,\gamma}(a,b))\setminus V(H_{0,\gamma}).$
A direct calculation gives
$$\sum_{uv\in E(H_{0,\gamma}(a,b))}\!\!|\mathcal{F}(u)||\mathcal{F}(v)|
=\sum_{uv\in E(H_{0,\gamma})}\!\!|\mathcal{F}(u)||\mathcal{F}(v)|=13,\,57,\,126,\,240$$ for $\gamma=1,2,3,4$, respectively.
Therefore, for each $\gamma\in\{1,2,3,4\}$,
the maximum value of $w^{(3)}(H^*)$ can be attained by $H_{0,\gamma}(a,b)$.
\end{proof}

By Lemma~\ref{lem5.2}, the extremal graph $K_2\nabla H^*$ must achieve the maximum value of $W^{(3)}(H^*)$.
This, together with Lemma \ref{lem6.2} and Claim \ref{cl6.1}, implies that for each $\gamma\in\{1,2,3,4\}$,
the positive incidence consequence $\pi(\mathbb{F})$ must maximize $\sum_{uv\in E(H^*)}|\mathbb{F}(u)||\mathbb{F}(v)|$.

\begin{claim}\label{cl6.2}
For $\gamma\in\{1,2,4\}$,
the maximum value of $w^{(3)}(H^*)$ is uniquely attained by the base graph $H_{0,\gamma}$;
and for \(\gamma=3\), the pair $(w^{(3)}(H^*),w^{(4)}(H^*))$ attains its joint maximum uniquely on $H_{0,\gamma}$,
while the maximum of $w^{(3)}(H^*)$ alone is not uniquely attained by $H_{0,\gamma}$.
Moreover, the attachment vertices are as specified in Definition \ref{def1.1}.
\end{claim}

\begin{proof}
By Claim \ref{cl6.1}, $\pi(\mathbb{F})=(2^2,1^2)$ for $\gamma=1$ and $\pi(\mathbb{F})=(3^2,2^3)$ for $\gamma=2$.
Moreover, the proof of Claim \ref{cl6.1} implies that $V^*$ is a $(\gamma+3)$-clique.
Since $d_{H^*}(u)=|\mathbb{F}(u)|+2$ for each fork $u$,
exactly two forks with $|\mathbb{F}(u)|=\gamma+1$ each have one remaining degree available for an edge outside the clique.
These two edges form the starting edges of the two long pendant paths, so no internal $2$-vertex remains.
Hence, the base graph $H_0$ of $H^*$ is isomorphic to $K_{3+\gamma}$,
where the attachment vertices of $H_0$ are the two forks with $|\mathbb{F}(u)|=\gamma+1$.
Definition \ref{def1.1} now gives $H_0\cong H_{0,\gamma}$.

By Claim \ref{cl6.1}, $\pi(\mathbb{F})=(4^6)$ for $\gamma=4$.
The proof of Claim \ref{cl6.1} indicates that $V^*$ is a $6$-clique.
Since $d_{H^*}(u)=|\mathbb{F}(u)|+2=6$ for each fork $u$,
each has one unused degree. Two of these start the pendant paths. Since there
are no missing edges among the six forks, the remaining four are paired
to serve as the neighbors of two internal 2-vertices. Every such pairing
yields the same base graph $H_0$ together with its attachment vertices up to an
automorphism of $K_6$, namely $H_{0,4}$.

By Claim \ref{cl6.1}, $\pi(\mathbb{F})=(4^2,3^3,1)$ for $\gamma=3$.
Since $d_{H^*}(u)=|\mathbb{F}(u)|+2$ for each fork $u$, we may label the two $6$-vertices $v_3,v_4$, the
three $5$-vertices $v_0,v_1,v_2$, and the unique $3$-vertex $v_5$.
The equality $\sum_{uv\in E(H^*)}|\mathbb{F}(u)||\mathbb{F}(v)|=126$ forces $\{v_i~|~i=0,1,2,3,4\}$ to be a clique and
$v_5$ to be adjacent to $\{v_3,v_4\}$ and one vertex from $\{v_0,v_1,v_2\}$.
Label this neighbor of $v_5$, other than $v_3$ and $v_4$, as $v_2$.
Then each of the four forks $v_0,v_1,v_3,v_4$ has one unused degree.
Two of these start the long pendant paths, while the other
two are joined via an internal 2-vertex.

In view of (\ref{eq23}), we have
$w^{(4)}(H^*)=\sum_{v\in V(H^*)}\big(\sum_{u\in N_{H^*}(v)}d_{H^*}(u)\big)^2.$
All terms corresponding to the forks are independent of the pairing: for $v_0,v_1,v_3,v_4$, every unused
degree is always assigned to a 2-vertex. The only varying terms
arise at the unique internal 2-vertex and at the external 2-vertex adjacent to the attachment vertex of
each pendant path. If $v_{i_1},v_{i_2}$ are the neighbors of the former and $v_{i_3},v_{i_4}$
are the two attachment vertices, their combined contribution is
$$(d_{H^*}(v_{i_1})+d_{H^*}(v_{i_2}))^2+(d_{H^*}(v_{i_3})+2)^2+(d_{H^*}(v_{i_4})+2)^2.$$
This expression evaluates to $242,234,228$ when $(d_{H^*}(v_{i_1}),d_{H^*}(v_{i_2}))$ equals $(6,6)$, $(6,5)$,
and $(5,5)$, respectively. The remaining portions of the long paths yield identical
contributions, depending only on their total length. Thus the unique
optimal choice joins the two 6-vertices $v_3,v_4$ via the unique internal 2-vertex,
which we label $v_6$, and attaches the pendant paths at the two 5-vertices $v_0,v_1$. Hence,
$N_{H_0}(v_5)=\{v_2,v_3,v_4\}$ and
$N_{H_0}(v_6)=\{v_3,v_4\}$, giving precisely $H_{0,3}$ with the
attachment vertices as specified in Definition \ref{def1.1}.
\end{proof}

Combining Claims \ref{cl6.1} and \ref{cl6.2} with Lemma \ref{lem5.2}, we conclude that
for $1\leq\gamma\leq4$, the base graph $H_0$, together with its two attachment vertices, is isomorphic
to $H_{0,\gamma}$ from Definition~\ref{def1.1}.
This completes the proof of the lemma.
\end{proof}

\begin{proof}[{\bf Proof of Theorem~\ref{thm1.2} for $1\leq\gamma\leq4$}]
By Lemma \ref{lem6.2}, $K_2\nabla H_{0,\gamma}(a,b)\in \mathbb{G}(n,\gamma)$.
Moreover, Lemma \ref{lem6.3} identifies the base graph $H_0$ of $H^*$ as $H_{0,\gamma}$ and specifies its two attachment vertices.
Therefore, the extremal graph $G^*$ is isomorphic to $K_2\nabla H_{0,\gamma}(a,b)$,
where the integers $a,b$ are determined by Lemma \ref{lem5.3} to satisfy $a+b=n-|H_{0,\gamma}|-2$ and $|a-b|\leq1$.
Finally, Claim \ref{cl5.1} gives $\gamma(G^*)=\gamma$.
This proves the theorem for $1\leq\gamma\leq4$.
\end{proof}

\section{Proof of Theorem \ref{thm1.2}: The case $\gamma\geq5$}

Throughout this section, $n$ is taken to be sufficiently large, and $\gamma\geq5$ is fixed.
We begin by constructing an embedding of the target extremal graph $K_2\nabla H_{0,\gamma}(a,b)$
on a surface of Euler genus $\gamma$,
where $a,b\geq1$. This construction is established via the two lemmas below.

Let $Q$ be a simple cubic graph with a proper 3-edge-coloring such that
\begin{equation}\label{eq43}
 \text{every two color classes induce a spanning cycle of $Q$.}
\end{equation}
Choose a vertex $v\in V(Q)$, and let $R=Q-\{v\}$ and $N_Q(v)=\{v_1,v_2,v_3\}$.
In the join graph $K_1\nabla R$, let $v_0$ be the added vertex,
and take $v_0,v_1$ as the ordered attachment vertices.
Note that the three neighbors of $v$ are of degree two in $R$, while all remaining vertices have degree three.
We now introduce a key triangular expansion–contraction lemma on a simple cubic graph $Q$ satisfying \eqref{eq43}.

\begin{lem}\label{lem7.1}
The following structural facts hold for $Q$ with \(|Q|\ge 5\) and subject to \eqref{eq43}:\vspace{0.5mm}\\
\noindent
(i) There exists at most one edge inside $N_Q(v)$.\\
\noindent
(ii) Suppose $v_iv_j$ is such an edge, with $\{i,j,k\}=\{1,2,3\}$.
Let $y_i,y_j$ be the respective other neighbors of $v_i,v_j$ in $R$.
Then $y_i,y_j,v_k$ are distinct.
Contracting the triangle $vv_iv_j$ to a vertex $v'$ produces a simple cubic graph $Q'$ satisfying \eqref{eq43},
with $N_{Q'}(v')=\{y_i,y_j,v_k\}$ and $Q'-\{v'\}=R-\{v_i,v_j\}$.\\
\noindent
(iii) Conversely, expanding any vertex into a triangle preserves \eqref{eq43}.
This inverse operation is uniquely determined once the
correspondence between the three triangle vertices and the three labeled external neighbors is fixed.
\end{lem}

\begin{proof}
(i) Suppose there are at least two edges among the neighbors of $v$.
After suitably relabeling the vertices, we may assume these edges are $v_1v_2$ and $v_2v_3$.
Let the colors of $vv_1,vv_2,vv_3$ be $1,2,3$, respectively. By the proper 3-edge-coloring,
$v_1v_2$ and $v_2v_3$ must receive colors 3 and 1, respectively.
It follows that $v\,v_1\,v_2\,v_3\,v$ is a 4-cycle alternating between colors 1 and 3,
contradicting \eqref{eq43}.

(ii) We first eliminate the possibilities $y_i=v_k$ and $y_j=v_k$.
If $y_i=v_k$, then $N_Q(v)$ contains the two edges $v_iv_j$ and $v_iv_k$, which is impossible by (i).
Likewise, $y_j\ne v_k$.
It remains to rule out $y_i=y_j$. Assume $y_i=y_j=y$. Then
$N_Q(v_i)=\{v,v_j,y\}.$
As $vv_j,v_jy\in E(Q)$, $N_Q(v_i)$ now has two edges, also a contradiction.
Therefore $y_i,y_j,v_k$ are pairwise distinct.

Let $T$ be the triangle induced by $\{v,v_i,v_j\}$. The contraction deletes the three interior edges of $T$,
identifies all vertices of $T$ into a single vertex $v'$, and retains all remaining edges.
Since $Q'$ is a simple cubic graph, the edges joining $T$ to the rest of the graph are precisely $v_iy_i,v_jy_j,vv_k.$
We have already shown that $y_i,y_j,v_k$ are pairwise distinct.
Hence, contraction yields three distinct edges $v'y_i,v'y_j,v'v_k$.
Thus $Q'$ is simple and $N_{Q'}(v')=\{y_i,y_j,v_k\}$. All other vertices preserve their degrees, so $Q'$ is also cubic.
Moreover, we have $Q'-\{v'\} =Q-\{v,v_i,v_j\} =R-\{v_i,v_j\}.$

We now show that the original 3-edge-coloring $c$ of $Q$ induces a proper 3-edge-coloring $c'$ of $Q'$.
After a suitable permutation of colors, set $c(vv_i)=i,c(vv_j)=j,$ and $c(vv_k)=k.$
Properness of the edge-coloring forces $c(v_iv_j)=k,c(v_iy_i)=j,$ and $c(v_jy_j)=i.$
The three edges leaving the triangle $T$ keep their colors after contraction:
$c'(v'y_i)=j,c'(v'y_j)=i,$ and $c'(v'v_k)=k.$
Therefore, the three edges incident to $v'$ have distinct colors.
Since the edge colors at all other vertices are unchanged, $c'$ is a proper 3-edge-coloring of $Q'$.

Finally, the edges of any two color classes of $c'$ still induce a spanning cycle of $Q'$.
Since $Q$ admits a proper 3-edge-coloring $c$ such that \eqref{eq43} holds,
the original $(i,j)$-bichromatic spanning cycle of $Q$ must contain the consecutive segment $y_i\,v_i\,v\,v_j\,y_j.$
After contraction, this segment becomes $y_i\,v'\,y_j$, while the rest of the cycle is unchanged.
This results a bichromatic spanning cycle of $Q'$.
The other two pairs of colors follow by symmetry.
This proves that $Q'$ satisfies \eqref{eq43}.

(iii) We now show the expansion together with the edge coloring.
Let $u$ be the vertex to be expanded, with three neighbors $x_1,x_2,x_3$.
Replace $u$ by a triangle $T=t_1t_2t_3$, with the correspondence $t_i\leftrightarrow x_i$ fixed.
Denote by $Q''$ the resulting graph.
Label the colors of the original edges so that $c(ux_i)=i$. Keep the color of every edge in $Q-\{u\}$, and define
$c''(t_ix_i)=i$ and $c''(t_it_j)=k$ whenever $\{i,j,k\}=\{1,2,3\}.$
The three edges incident to each new vertex receive distinct colors,
so $c''$ is a proper 3-edge-coloring of $Q''$.

Take any two colors $i,j$, and let $k$ be the remaining color.
The original $(i,j)$-bichromatic spanning cycle passing through $u$ contains the segment $x_iux_j$.
After expansion, this segment is replaced by $x_i\,t_i\,t_k\,t_j\,x_j.$
The edges along this path are colored alternately $i,j,i,j$,
and the path traverses all three new vertices. Together with the rest of the original cycle,
it forms an $(i,j)$-bichromatic spanning cycle of $Q''$.
Hence, the expansion preserves \eqref{eq43}.

Finally, we prove the uniqueness of this inverse operation.
Once the correspondence is fixed, the edge set of $Q''$ must be
$E(Q-\{u\})\cup\{t_1t_2,t_2t_3,t_3t_1\}\cup\{t_1x_1,t_2x_2,t_3x_3\}.$
That is, the three interior edges of the triangle $T$ and the three edges leaving $T$ are all determined,
while all other edges remain unchanged;
hence this inverse operation is uniquely determined. If the original edge coloring is also fixed,
the colors on the interior edges of $T$ are uniquely determined as well,
since $t_it_j$ must take the third color different from $i$ and $j$.
\end{proof}

For $k=0,\ldots,\gamma-1$, let $R_k^+$ be the graph obtained from $R_k$,
where $R_k$ and its current ordered triple $(w_{2k},w_{2k+2},w_{2k+3})$ are defined iteratively in Definition \ref{def1.1},
by adding a vertex adjacent to the three vertices of this ordered triple.
In particular, $R_{\gamma-1}^+$ is obtained from $R_{\gamma-1}$ by adding a vertex adjacent
to the vertices of the final ordered triple $(v_1,v_2,v_3)=(w_{2\gamma-2},w_{2\gamma},w_{2\gamma+1})$.

\begin{lem}\label{lem7.2}
For $k\in\{0,\ldots,\gamma-1\}$,
$R_{k}^+$ is a simple cubic graph of order $2k+4$ satisfying \eqref{eq43}.
\end{lem}

\begin{proof}
By Remark \ref{rem1.1}, $R_k$ is a simple graph on $2k+3$ vertices,
in which every vertex of its current ordered triple $(w_{2k},w_{2k+2},w_{2k+3})$ has degree two,
and all remaining vertices have degree three.
Thus $R_k^+$ is a simple cubic graph of order $2k+4$.
We next show that every $R_k^+$ satisfies \eqref{eq43}.

Indeed, since $R_0^+\cong K_4$, its three perfect matchings give
a proper $3$-edge-coloring in which any two color classes form a 4-cycle.
For each $k\in\{1,\ldots,\gamma-1\}$, the ordered triple of $R_{k-1}$ is $(w_{2k-2},w_{2k},w_{2k+1})$.
Denote by $v,v'$ the added vertices of $R_{k-1}^+$ and $R_k^+$, respectively.
By construction, $R_k$ is obtained from $R_{k-1}$ by adding
vertices $w_{2k+2},w_{2k+3}$ and a path $w_{2k+1}w_{2k+2}w_{2k+3}w_{2k-2}$.
Hence, $R_k^+$ can be obtained from $R_{k-1}^+$ by expanding the vertex $v$
into a triangle $v'w_{2k+2}w_{2k+3}$, with external edges $v'w_{2k}$, $w_{2k+2}w_{2k+1}$, $w_{2k+3}w_{2k-2}$.
By Lemma \ref{lem7.1}, this expansion preserves \eqref{eq43}.
Induction therefore shows that every $R_k^+$ satisfies \eqref{eq43}.
\end{proof}

\begin{lem}\label{lem7.3}
Let $Q$ be a simple cubic graph of order $2\gamma+2$ satisfying \eqref{eq43}.
Choose a vertex $v\in V(Q)$ and any neighbor $v_1\in N_Q(v)$, and define $R=Q-\{v\}$.
Add a new vertex $v_0$ to form the join $K_1\nabla R$,
and let $H$ be the graph obtained from $K_1\nabla R$ by attaching paths of lengths $a,b\geq1$
at $v_0,v_1$, respectively. Then $K_1\nabla R$ has a spanning path from $v_0$ to $v_1$.
Moreover, $\gamma(K_2\nabla H)=\gamma$.
\end{lem}

\begin{proof}
Since $Q$ is a simple cubic graph satisfying \eqref{eq43},
it admits a bichromatic spanning cycle $C$ such that $vv_1\in E(C)$.
Deleting $v$ from $C$ yields a spanning path of $R$ from
$v_1$ to another neighbor of $v$, say $v_2$.
Appending the edge $v_2v_0$ gives a spanning path of $K_1\nabla R$ with endpoints $v_0$ and $v_1$.

To show $\gamma(K_2\nabla H)=\gamma$,
we begin by proving that $K_2\nabla H$ embeds into a surface of Euler genus $\gamma$.
Let $T=x_1x_2x_3$ be a triangle on $\{x_1,x_2,x_3\}$, equipped with a fixed cyclic vertex ordering $(x_1,x_2,x_3)$.
We fix a proper 3-edge-coloring of $Q$, whose existence is guaranteed by Lemma \ref{lem7.1},
and assign the labels $x_1x_2,x_2x_3,_3x_1$ to its three color classes such that the edge $vv_1\in E(Q)$ receives the color $x_1x_2$.

Using this labeled proper 3-edge-coloring of $Q$,
we take a bipartite graph $B\cong K_{3,2\gamma+2}$ with bipartition $(V(T),V(Q))$,
and construct a quadrangular cellular embedding $\widetilde{B}$ of $B$ into some closed surface $\Sigma$ as follows:
For each edge $st\in E(Q)$ with color label $x_ix_j$, we specify a 4-face of $\widetilde{B}$ bounded by the cycle $(x_i,s,x_j,t,x_i)$.

Firstly, this prescription ensures that every edge of $\widetilde{B}$ is contained in exactly two 4-faces.
To verify this property, fix a vertex $s\in V(Q)$ and a vertex, say $x_1$, in $V(T)$.
The three edges of $Q$ incident to $s$ receive distinct color labels $x_1x_2, x_2x_3, x_3x_1$.
Exactly two of them, say $st_1$ and $st_2$, carry labels $x_1x_2$ and $x_3x_1$, respectively.
The 4-faces bounded by $(x_1,s,x_2,t_1,x_1)$ and $(x_3,s,x_1,t_2,x_3)$ are then the two 4-faces of $\widetilde{B}$ containing the edge $sx_1$.

Secondly, for each $u\in V(\widetilde{B})$,
the neighbors of $u$ in $\widetilde{B}$ have a cyclic ordering
such that the 4-faces constructed above and incident with $u$ occur consecutively around $u$ along this ordering.
Indeed, we check the two families of vertices separately.
For $s\in V(Q)$, its neighbors are $x_1,x_2,x_3$.
We adopt the cyclic ordering $(x_1,x_2,x_3)$ inherited from our fixed vertex ordering on $T$.
The incident 4-faces are exactly those associated with consecutive pairs $(x_1,x_2),(x_2,x_3), (x_3,x_1)$ in this sequence.

First consider a vertex $s\in V(Q)$.
The neighbors of $s$ in $\widetilde{B}$ are exactly $x_1,x_2,x_3$.
Following the fixed cyclic ordering $(x_1,x_2,x_3)$ of $V(T)$ and the cyclic sequence of color labels
on the edges of $Q$ incident to $s$, we take the cyclic ordering $(x_1,x_2,x_3)$ of these neighbors.
By our face construction, the 4-faces incident to $s$ are in turn bounded by
$(x_1,s,x_2,t_1,x_1)$, $(x_2,s,x_3,t_2,x_2)$, and $(x_3,s,x_1,t_3,x_3)$.
These 4-faces are precisely the faces spanned by consecutive pairs of neighbors in the cyclic ordering of $s$, as required.

Next consider a vertex, say $x_1$, in $V(T)$.
The neighbors of $x_1$ are precisely the vertices of $V(Q)$.
By \eqref{eq43}, $Q$ admits an $(x_1x_2,x_3x_1)$-bichromatic spanning cycle,
and the cyclic ordering of this cycle induces a cyclic permutation of these neighbors.
For each consecutive pair $s_i,s_{i+1}$ in this ordering,
the edge $s_is_{i+1}$ has color label either $x_1x_2$ or $x_3x_1$.
Accordingly, the corresponding $4$-face has boundary
$(x_1,s_i,x_2,s_{i+1},x_1)$ or
$(x_1,s_i,x_3,s_{i+1},x_1)$, respectively.
Thus the 4-faces around $x_1$ are ordered consistently with this cyclic permutation of neighbors.

Since $\widetilde{B}$ is connected, the same argument from Lemma \ref{lem6.1},
with quadrilaterals instead of triangles, yields the desired closed surface $\Sigma$.
Let $\gamma^*$ be the Euler genus of $\Sigma$ and let $f$ be the number of 4-faces in $\widetilde{B}$.
By the face-edge handshaking lemma, $4f=2e(B)=6(2\gamma+2)$,
and Euler’s formula gives $\gamma^*=e(B)-|B|-f+2=\frac12e(B)-|B|+2=\gamma$.

We now triangulate each 4-face of $\widetilde{B}$.
By the construction of this quadrangular cellular embedding,
every edge $st\in E(Q)$, with color label $x_ix_j$,
uniquely determines a 4-face of $\widetilde{B}$ with boundary cycle $(x_i,s,x_j,t,x_i)$.
If $vv_k\in E(Q)$ bears color label $x_ix_j$,
we insert the diagonal $x_ix_j$ inside the quadrilateral bounded by $(x_i,v,x_j,v_k,x_i)$,
splitting it into triangles $x_ix_jv$ and $x_ix_jv_k$.
If $st\in E(Q)$ bears color label $x_ix_j$ and is not incident to $v$, that is, $st\in E(R)$,
we insert the diagonal $st$ inside the quadrilateral bounded by $(x_i,s,x_j,t,x_i)$,
which splits the quadrilateral into triangles $stx_i$ and $stx_j$.
All edges from $V(T)$ to $V(Q)$ are preserved,
and the three $x_ix_j$ diagonals form the triangle $T=x_1x_2x_3$.
Since each color label appears exactly once among the edges incident to $v$,
we obtain three triangles of the form $x_ix_jv$, one associated with each color label $x_ix_j\in E(T)$.
The union of these 3-faces forms a closed topological disk on $\Sigma$, bounded by the triangle $T=x_1x_2x_3$.
We remove this disk containing $v$ and its three incident triangles, then attach a new 3-face
$u^*u^{**}v_0$ along the boundary cycle $x_1x_2x_3$, where $u^*,u^{**},v_0$ replace $x_1,x_2,x_3$, respectively.
This is a disk-for-disk operation that leaves the surface $\Sigma$ unchanged.
The resulting underlying graph is $K_3\nabla(Q-\{v\})=K_2\nabla(K_1\nabla R)$.

Edge $u^*u^{**}$ belongs to the new 3-face $u^*u^{**}v_0$ and, since $vv_1$ carries color label $x_1x_2$,
also to the 3-face $u^*u^{**}v_1$.
To construct a pendant path of length $a$ rooted at $v_0$,
insert a new vertex $y_1$ inside the face $u^*u^{**}v_0$ and join it to all three boundary vertices.
The newly created face $u^*u^{**}y_1$ allows this operation to be iterated.
After $a$ iterations, the vertices $v_0,y_1,\dots,y_a$ form the desired pendant path rooted at $v_0$.
We repeat the same construction within the face $u^*u^{**}v_1$ to produce the second pendant path of length $b$.
The resulting embedding is therefore a triangulation of $K_2\nabla H$ on the surface $\Sigma$ of Euler genus $\gamma$.
Denote $n=|K_2\nabla H|$.
By Euler’s formula and the handshaking lemma, the above triangulation yields $e(K_2\nabla H)=3(n+\gamma-2)$.
If $K_2\nabla H$ were embedded on a surface of Euler genus $g<\gamma$, Lemma \ref{lem3.3} would give
$e(K_2\nabla H)\leq3(n+g-2)$, a contradiction.
Thus $\gamma(K_2\nabla H)=\gamma$.
\end{proof}

Recall that $G^*\in SPEX(n,\gamma)$, where $G^*\cong K_2\nabla H^*$ with $V(K_2)=\{u^*,u^{**}\}$.
Moreover, $H^*$ is obtained from a base graph $H_0$
by appending pendant paths of lengths $a$ and $b$ to two distinct attachment vertices of $H_0$, respectively,
where $a+b+|H_0|=n-2$ and $|a-b|\leq 1$.
By Claim \ref{cl5.1}, for any $u\in V(H^*)$,
there are exactly $d_{\widetilde{G}^*}(u)$ faces incident to $u$;
these faces form a wheel $W(u)$ centered at $u$ and bounded by a $d_{\widetilde{G}^*}(u)$-cycle $C_u$.
Recall also that $\widetilde{G}^*$ is a triangulation of a closed surface of Euler genus $\gamma$,
and $\mathbb{F}$ denotes the set of 3-faces of $\widetilde{G}^*$ incident to neither $u^*$ nor $u^{**}$.

\begin{lem}\label{lem7.4}
$H^*$ has degree sequence
\begin{equation}\label{eq44}
 \pi(H^*)=(2\gamma+2,\ 4^{2\gamma-1},\ 3^2,\ 2^{n-2\gamma-6},\ 1^2).
\end{equation}
Each face in $\mathbb{F}$ is incident to $v_0$,
$C_{v_0}\!-\!\{u^*\!,u^{**}\}$ is a path $P$ of length $2\gamma$ plus an isolated vertex $u_0$,
and $N_{H^*}(v_0)=V(P)\cup\{u_0\}$,
where $V(P)$ comprises all other forks and $u_0$ is a 2-vertex in $H^*$.
\end{lem}

\begin{proof}
Recall that $R_{\gamma-1}^+$ is obtained from $R_{\gamma-1}$ by adding a vertex adjacent
to the vertices of the final ordered triple $(v_1,v_2,v_3)$.
By Lemma \ref{lem7.2}, $R_{\gamma-1}^+$ is a simple cubic graph of order $2\gamma+2$ satisfying \eqref{eq43},
and by Definition \ref{def1.1},
$H_{0,\gamma}$ is isomorphic to $K_1\nabla R_{\gamma-1},$ with $V(K_1)=\{v_0\}$ and attachment vertices $v_0$ and $v_1$.
Lemma \ref{lem7.3} therefore gives $\gamma(K_2\nabla H_{0,\gamma}(a,b))=\gamma$,
where $a=\lceil\frac{n-2\gamma-4}2\rceil$ and $b=\lfloor\frac{n-2\gamma-4}2\rfloor$ as specified in Theorem \ref{thm1.2}.
Hence, $K_2\nabla H_{0,\gamma}(a,b)\in \mathbb{G}(n,\gamma)$.
The proof of Lemma \ref{lem7.3} provides its triangulation on a closed surface of Euler genus $\gamma$,
which we now denote $\widetilde{G}$ for convenience in the subsequent discussion.

We now determine the degree sequence of $H_{0,\gamma}(a,b)$.
Since $R_{\gamma-1}$ is obtained from the cubic graph $R_{\gamma-1}^+$ on $2\gamma+2$ vertices
by deleting a vertex $v$ and its incident edges $vv_1,vv_2,vv_3$,
we have $\pi(R_{\gamma-1})=(3^{2\gamma-2},2^3)$
and $\pi(H_{0,\gamma})=\pi(K_1\nabla R_{\gamma-1})=(2\gamma+1,4^{2\gamma-2},3^3)$,
where $v_0$ is the $(2\gamma+1)$-vertex and $v_1$ is a 3-vertex.
As $v_0$ and $v_1$ serve as the attachment vertices for $H_{0,\gamma}(a,b)$,
the degree sequence of $H_{0,\gamma}(a,b)$ satisfies (\ref{eq44}),
with $v_0$ as its unique $(2\gamma+2)$-vertex.

We next compare the triangulations $\widetilde{G}^*$ and $\widetilde{G}$,
where $\widetilde{G}^*$ is the triangulation of the extremal graph $K_2\nabla H^*$
and $\widetilde{G}$ is the triangulation of our candidate graph $K_2\nabla H_{0,\gamma}(a,b)$.
Moreover, let $\mathcal{F}$ denote the set of faces of $\widetilde{G}$ incident to neither $u^*$ nor $u^{**}$,
and let $\mathcal{F}(u)$ be the subset of $\mathcal{F}$ consisting of all faces incident to a vertex $u.$

By the same argument as for Claim \ref{cl5.8},
we know that $|\mathcal{F}(u)|=d_{H_{0,\gamma}(a,b)}(u)-2$ for each fork $u$ in $H_{0,\gamma}(a,b)$
and $|\mathcal{F}(u)|=0$ for all other $u$ in $H_{0,\gamma}(a,b)$.
Since the degree sequence of $H_{0,\gamma}(a,b)$ satisfies (\ref{eq44}),
it follows that
\begin{equation}\label{eq45}
M_2:=\sum_{u\in V(H_{0,\gamma}(a,b))}|\mathcal{F}(u)|^2=4\gamma^2+8\gamma-2.
\end{equation}

Both $H^*$ and the candidate $H_{0,\gamma}(a,b)$ have $n-2$ vertices and $n-3+3\gamma$ edges,
so $w^{(1)}(H^*)=w^{(1)}(H_{0,\gamma}(a,b))$.
Lemma~\ref{lem7.3} also guarantees that $H_{0,\gamma}$ has a spanning path between its two attachment vertices $v_0$ and $v_1$,
and thus $H_{0,\gamma}(a,b)$ correspondingly admits a spanning path.
The edge count further shows that $H_{0,\gamma}(a,b)$ is obtained from this spanning path by adding exactly $3\gamma$ edges,
i.e., $K_2\nabla H_{0,\gamma}(a,b)\in \mathbb{H}(n,\gamma)$.
Theorem \ref{thm4.1} gives the same structural properties for $H^*$.
Thus both graphs satisfy the spanning-path, edge-count, and the condition of Lemma \ref{lem5.2}.
By Claim~\ref{cl5.9},
$w^{(2)}(H^*)-w^{(2)}(H_{0,\gamma}(a,b))=M_1-M_2,$
where $M_1=\sum_{u\in V(H^*)}|\mathbb{F}(u)|^2$.
If $M_1<M_2$, Lemma \ref{lem5.2} would force $\rho\big(K_2\nabla H_{0,\gamma}(a,b)\big)>\rho(G^*)$, a contradiction.
Thus $M_1\geq M_2$.

Let $d=\max_{u\in V(H^*)}|\mathbb F(u)|$, attained at $u_{i^*}=v_0$.
Clearly, $d=|\mathbb F(v_0)|\leq|\mathbb F|.$
We aim to prove $d=|\mathbb F|.$
Setting $\mathbb F''=\mathbb F\setminus\mathbb F(v_0)$,
the general bound in Claim~\ref{cl5.11} gives
$\sum_{u\in V(H^*)}|\mathbb{F}''(u)|^2\leq |\mathbb F''|^2+5|\mathbb F''|=(|\mathbb F|-d)^2+5(|\mathbb F|-d).$
Substituting $\mathbb F'=\mathbb F$ into \eqref{eq39}, we further obtain
\begin{equation*}
M_1\leq d^2+4d-2+12(|\mathbb F|-d)+(|\mathbb F|-d)^2+5(|\mathbb F|-d) \\
=M_2+(|\mathbb F|-d)(13-2d),
\end{equation*}
where $|\mathbb F|=2\gamma$ by Claim \ref{cl5.7}.
This immediately eliminates the case $7\leq d<|\mathbb F|$.
For $\gamma\geq8$ and $d\leq6$, we have $M_1\leq 3d|\mathbb{F}|\leq 36\gamma$.
Comparing with \eqref{eq45}, we then have
$M_1<M_2$. It remains to treat the cases $\gamma=5,6,7$ under the restriction $d\leq 6$.

Let $d=6$. Using sharp bounds tabulated in Claim \ref{cl5.11} for $|\mathbb{F}''|$,
we find that $\sum_{u\in V(H^*)}|\mathbb{F}''(u)|^2$ is bounded above by 30, 60, 96
for $|\mathbb{F}''|=|\mathbb{F}|-d=4, 6, 8$, respectively.
Substituting these sharp upper bounds for $\sum_{u\in V(H^*)}|\mathbb{F}''(u)|^2$ into \eqref{eq39}
yields
\[
\begin{array}{l|ccc}
|\mathbb F| & 10 & 12 & 14 \\
\hline
M_1\leq & 136 & 190 & 250 \\
M_2= & 138 & 190 & 250
\end{array}
\]
Equality in either of the last two columns is impossible.
Indeed, equality in \eqref{eq39} implies equality in \eqref{eq42} and $c=1$.
Furthermore, equality holds in \eqref{eq42} if and only if all vertices $u$ with $|\mathbb F''(u)|>0$ satisfy $a_u=2$.
There are at most $d-c\leq 5$ vertices $u$ with $a_u=2$, whereas for both equality cases with $|\mathbb F''|=6,8$,
 $\pi(\mathbb F'')$ contains at least six vertices $u$ satisfying $|\mathbb F''(u)|>0$, a contradiction.

If $d\leq 5$ and $|\mathbb F|=2\gamma\in\{12,14\}$,
then $M_1\leq 3d|\mathbb F|\leq 30\gamma$, and \eqref{eq45} gives $M_1<M_2$.
When $|\mathbb F|=2\gamma=10$, we have $M_2=138$.
If $d\leq 4$, then $M_1\leq 3d|\mathbb F|\leq 120<M_2$.
If $d=5$, then \eqref{eq40} yields $|C_{v_0}[\mathbb F]|\geq |\mathbb F(v_0)|+1=d+1=6$.
Combined with the structure of $C_{v_0}[\mathbb F]$, this forces at least seven vertices $u$ with $|\mathbb F(u)|>0$.
Since $\sum_{u\in V(H^*)}|\mathbb F(u)|=3|\mathbb F|=30$,
the only partitions of $30$ into at least seven positive parts satisfying $M_1\ge 138$
are $\pi(\mathbb F)=(5^5,4,1)$ and $(5^5,3,2)$.
By \eqref{eq40}, each of the five vertices with incidence $5$ satisfies
$|C_u[\mathbb F]|\ge6$. Since there are exactly $7$ vertices $u$ with
$|\mathbb F(u)|>0$, for each such vertex $v$, the set $C_v[\mathbb F]$
contains all the other six vertices with positive incidence. Let $u$ be a
vertex with $|\mathbb F(u)|\le2$. Then each of the five vertices of incidence
$5$ belongs to $C_u[\mathbb F]$, and hence $|C_u[\mathbb F]|\ge5$.
On the other hand, \eqref{eq40} gives
$|C_u[\mathbb F]|\le2|\mathbb F(u)|\le4$, a contradiction.
This eliminates all possibilities with $d<|\mathbb F|$.

We now conclude $d=|\mathbb F|=2\gamma$,
i.e., all faces of $\mathbb{F}$ share a common vertex $v_0$.
By Claim~\ref{cl5.6}, the only two 3-faces containing the edge $u^*u^{**}$ are $u^*u^{**}u_1$ and $u^*u^{**}u_{n-2}$.
Thus $u^*$ and $u^{**}$ are nonconsecutive on the cycle $C_{v_0}$;
otherwise, $u^*u^{**}v_0$ would form an additional 3-face incident to $u^*u^{**}$.
Deleting $u^*$ and $u^{**}$ from $C_{v_0}$ leaves two nonempty paths, total with $d$ edges and  $d+2$ vertices.
If both paths are nontrivial, then the four endpoints have incidence one and
the remaining $d-2$ vertices on the two paths have incidence two to $\mathbb F$.
Consequently, $M_1=d^2+4+4(d-2)=4\gamma^2+8\gamma-4<M_2$, a contradiction.
Hence, $C_{v_0}\!-\!\{u^*\!,u^{**}\}$ consists of a path $P$ of length $2\gamma$ and an isolated vertex $u_{0}$,
which gives $N_{H^*}(v_0)=V(P)\cup\{u_0\}$.

By Claim \ref{cl5.8}, $d_{H^*}(u)=|\mathbb F(u)|+2$ for each fork $u$,
which implies that every internal vertex of $P$ is a 4-vertex,
while both endpoints of $P$, denoted $v_2$ and $v_3$, are 3-vertices.
In particular, $u_0$ is not a fork, since $|\mathbb F(u_0)|=0$.
Note that the two 1-vertices $u_1$ and $u_{n-2}$ each have a neighbor of degree two.
Since $v_0$, a neighbor of $u_0$, is of degree $2\gamma+2$, we conclude that $u_0$ is a 2-vertex in $H^*$.
\end{proof}

\begin{lem}\label{lem7.5}
We have the following statement for the structure of $H^*$: \vspace{0.5mm}\\
\noindent
(i) $H^*$ contains no internal 2-vertices,
and hence its base graph $H_0$ is the join of $v_0$ and a graph $R$ of order $2\gamma+1$,
with $v_0$ and a fork $v_1$ as its attachment vertices.\\
\noindent
(ii) $R$ has three $2$-vertices $v_1,v_2,v_3$ and $2\gamma-2$ vertices of degree three,
and $v_2v_3\in E(R)$.\\
\noindent
(iii) Adding a vertex $v$ and edges $vv_1,vv_2,vv_3$ to $R$ yields a cubic graph $Q$ satisfying \eqref{eq43}.
\end{lem}

\begin{proof}
(i) By Lemma \ref{lem7.4},
$\pi(H^*)=(2\gamma+2,4^{2\gamma-1},3^2,2^{n-2\gamma-6}, 1^2).$
Denote by $U$ the set of 4-vertices, $V$ the set of 2-vertices, and $\{v_2,v_3\}$ the set of 3-vertices.
Let $t:=e(U,V)$, $r:=e(\{v_2,v_3\},V)$, and let $S$ be the set of internal 2-vertices of $H^*$.
Moreover, we define $\varepsilon=\mathbf1_{v_2v_3\in E(H^*)}$,
where $\varepsilon$ is the indicator variable which takes the value 1 if $v_2v_3\in E(H^*)$, and 0 otherwise.

From the final proof of Lemma \ref{lem7.4},
we know that $v_2,v_3$ are the two endpoints of $P$, and $U=V(P)\setminus\{v_2,v_3\}$.
Thus, each of $v_2$ and $v_3$ already has two fork neighbors, $v_0$ and its neighbor on $P$,
and hence admits at most one neighbor of degree two.
Therefore, $r\leq 2$.
By Claims \ref{cl5.3} and \ref{cl5.4}, all internal 2-vertices form an independent set.
The two pendant paths together with $|S|$ internal 2-vertices in $H^*$ contribute exactly $2|S|+2$ edges from 2-vertices to forks.
Consequently,
\begin{equation}\label{eq46}
t=\big(2|S|+2\big)-r-\big|\{u_0v_0\}\big|\geq 2|S|-1.
\end{equation}

We apply the formula for $w^{(3)}(H^*)$ in Claim \ref{cl5.9}.
By Lemma \ref{lem7.4}, $|U|=2\gamma-1$ and $|\mathbb{F}(v_0)|=|\mathbb{F}|=2\gamma.$
Recall that $v_2,v_3$ are endpoints of $P$ and $U=V(P)\setminus\{v_2,v_3\}$.
Thus $|\mathbb{F}(v_2)|=|\mathbb{F}(v_3)|=1$ and $|\mathbb{F}(u)|=2$ for all $u\in U$,
while $|\mathbb{F}(u)|=0$ for all remaining vertices in $H^*$.
Hence,
\begin{equation}\label{eq47}
\sum_{u\in V(H^*)}|\mathbb{F}(u)|^2
=(2\gamma)^2+4(2\gamma-1)+2
=4\gamma^2+8\gamma-2.
\end{equation}

We now compute $\sum_{uv\in E(H^*)}|\mathbb{F}(u)||\mathbb{F}(v)|.$
By Claim \ref{cl5.8}, $|\mathbb F(u)|=0$ for any vertex $u$ of degree one or two.
Since $V(P)=U\cup\{v_2,v_3\}$ and $N_{H^*}(v_0)=V(P)\cup \{u_0\}$, where $u_0$ is a 2-vertex by Lemma \ref{lem7.4},
the edges incident to $v_0$ contribute $4\gamma(2\gamma-1)+4\gamma=8\gamma^2$ to the sum.
Every edge within $U$ contributes $4$, every edge
between $U$ and $\{v_2,v_3\}$ contributes $2$,
and the edge $v_2v_3$, if present, contributes $1$.
All other edges contribute zero. Hence
\[
\sum_{uv\in E(H^*)}|\mathbb{F}(u)||\mathbb{F}(v)|
=8\gamma^2+4e(U)+2e(U,\{v_2,v_3\})+\varepsilon.
\]

To evaluate the middle two terms, sum the degrees over $U$.
Each of $2\gamma-1$ vertices in $U$ has degree four and is adjacent to $v_0$.
Thus
$4(2\gamma-1)=(2\gamma-1)+t+2e(U)+e(U,\{v_2,v_3\}),$
and therefore $4e(U)+2e(U,\{v_2,v_3\})=12\gamma-6-2t.$
Consequently,
\begin{equation}\label{eq48}
\sum_{uv\in E(H^*)}|\mathbb{F}(u)||\mathbb{F}(v)|
=8\gamma^2+12\gamma-6-2t+\varepsilon.
\end{equation}

Substituting \eqref{eq47} and \eqref{eq48} into the formula for $w^{(3)}(H^*)$ in Claim \ref{cl5.9},
we obtain
\begin{align}
w^{(3)}(H^*)
&=8n+72\gamma-32
  +4\big(4\gamma^2+8\gamma-2\big) \notag
  +2\big(8\gamma^2+12\gamma-6-2t+\varepsilon\big) \notag\\
&=8n+32\gamma^2+128\gamma-52-4t+2\varepsilon.
\label{eq49}
\end{align}

We now show $S=\varnothing$.
Let $H:=H_{0,\gamma}(a,b)$, where $a,b$ are as in Theorem~\ref{thm1.2}.
By Definition~\ref{def1.1} and Lemma~\ref{lem7.4},
$H$ and $H^*$ have the same order and degree sequence.
Moreover, in both graphs, $v_0$ is adjacent to every other
vertex of degree at least three and to exactly one $2$-vertex.
Thus the calculation leading to \eqref{eq49} also applies to $H$.
Since the corresponding parameters for $H$ are $t=1$ and $\varepsilon=1$, we have
$w^{(3)}(H)=8n+32\gamma^2+128\gamma-54,$
and $w^{(3)}(H)-w^{(3)}(H^*)=4t-2\varepsilon-2.$
Also, $w^{(j)}(H)=w^{(j)}(H^*)$ for $j\in\{1,2\}$,
as the walk count is determined by the degree sequence.

Suppose for a contradiction that $|S|\geq1$.
By \eqref{eq46} and $r\leq2$, we have
$t=1+2|S|-r\geq1$. If $t\geq2$, then $\varepsilon\leq1$ gives
$w^{(3)}(H)-w^{(3)}(H^*)\geq4.$
Now suppose $t=1$.
Then \eqref{eq46} forces $|S|=1$ and $r=2$.
Recall that each of $v_2,v_3$ has at most one degree-two neighbor.
Thus $r=2$ forces both vertices to have exactly one such neighbor.
If $v_2v_3\in E(H^*)$, then the three neighbors of each of
$v_2,v_3$ are $v_0$, the other vertex in $\{v_2,v_3\}$,
and its degree-two neighbor.
Recall that $|\mathbb{F}(v_2)|=|\mathbb{F}(v_3)|=1$ and $|\mathbb{F}(u)|=0$ for each 2-vertex $u$.
Thus the unique face in $\mathbb F(v_2)\cup \mathbb F(v_3)$ must be $v_0v_2v_3$.
By the final part of the proof of Lemma~\ref{lem7.4},
$v_2,v_3$ are the two endpoints of the path $P$ in $C_{v_0}\!-\!\{u^*,u^{**}\}$.
The face $v_0v_2v_3\in\mathbb F$ would give $v_2v_3\in E(P)$,
forcing $e(P)=1$.
This contradicts $e(P)=|\mathbb F|=2\gamma\geq10$.
Hence, $v_2v_3\notin E(H^*)$ and thus $\varepsilon=0$,
which yields $w^{(3)}(H)-w^{(3)}(H^*)=2.$

Whether $t\geq2$ or $t=1$, we derive $w^{(3)}(H)>w^{(3)}(H^*)$, while
$w^{(j)}(H)=w^{(j)}(H^*)$ for $j=1,2$.
The feasibility hypotheses of Lemma~\ref{lem5.2}
for $H_{0,\gamma}(a,b)$ were verified in Remark \ref{rem1.2} and Lemmas \ref{lem7.2} and \ref{lem7.3}.
Applying Lemma \ref{lem5.2} yields a contradiction to the maximality of $\rho(G^*)$.
Hence, $S=\varnothing$. This, together with Lemma~\ref{lem7.4}, implies that
the base graph $H_0$ is the join of $v_0$ and a graph $R$ with $V(R)=V(P)=U\cup \{v_2,v_3\}$.
Thus $|R|=|U|+2=2\gamma+1$.
Also by Lemma~\ref{lem7.4}, $v_0$ has a unique $2$-neighbor $u_0$.
Since $S=\varnothing$, $u_0$ lies on a pendant path attached at $v_0$.
Thus the other pendant path has an attachment vertex $v_1$, which is clearly a fork.
This proves statement (i).

(ii) We next determine the degree sequence of $R$ and show $v_2v_3\in E(R)$.
Both points rely on the attachment vertex $v_1\in U$.
Since $S=\varnothing$ while $v_1$ is a fork,
we have $v_1\in U\cup\{v_2,v_3\}$.
We shall exclude $v_1\in\{v_2,v_3\}$ by using the cyclic
order of the faces around the attachment vertices.

To use this cyclic order conveniently,
we shall remove the two pendant paths and verify that their deletion preserves the triangulation.
Delete the vertices of each pendant path successively from
its leaf towards its attachment vertex, retaining the
attachment vertex itself.
At each step, the deleted vertex has degree three in the full join graph containing $u^*$ and $u^{**}$.
By Claim \ref{cl5.1}, its three incident 3-faces form a wheel, and replacing
this wheel by the 3-face bounded on its three neighbors preserves the underlying surface.
Thus the resulting graph $\widehat G=K_2\nabla H_0$ inherits a triangular embedding,
with faces $u^*u^{**}v_0$ and $u^*u^{**}v_1$.
For each vertex $x$ of $\widehat G$, let $C_x$ be the
cycle formed by its neighbors in their cyclic order around $x$.
In particular, a face $xyz$ gives $yz\in E(C_x)$.

We now show that $v_1\notin\{v_2,v_3\}$.
Suppose otherwise.
Then $d_{H^*}(v_1)=3$, so deleting its pendant path leaves
$v_1$ with two neighbors in $H_0$, namely $v_0$ and some vertex $w$.
Thus $C_{v_1}$ is a 4-cycle on $\{u^*,u^{**},v_0,w\}$,
and it contains the edge $u^*u^{**}$, corresponding to the face $u^*u^{**}v_1$.
Hence, $v_0$ is adjacent to either $u^*$ or $u^{**}$ in this cycle.
Suppose first that $v_0u^*\in E(C_{v_1})$.
Then $u^*v_0v_1$ is a face.
Together with the faces $u^*u^{**}v_0$ and
$u^*u^{**}v_1$, it gives the three edges
$u^{**}v_0$, $v_0v_1$, and $v_1u^{**}$ in $C_{u^*}$.
Thus $C_{u^*}$ contains a $3$-cycle.
This is impossible,
since Claim \ref{cl5.1} implies that
$C_{u^*}$ is a cycle with vertex set
$V(H_0)\cup\{u^{**}\}$, which has more than three vertices.
The case $v_0u^{**}\in E(C_{v_1})$ is symmetric.
Therefore, $v_1\in U$.

We can now determine the degree sequence of $R$ and prove $v_2v_3\in E(R)$.
Since $H^*$ has no internal $2$-vertices and $v_0$ is adjacent
to every other vertex of degree at least three,
$H_0=K_1\nabla R$, where $R=H^*[V(P)]$ has order $2\gamma+1$.
The vertex $v_1$ loses one incident edge when its pendant
path is removed and another when $v_0$ is deleted.
Thus $d_R(v_1)=4-2=2$.
Neither $v_2$ nor $v_3$ is an attachment vertex, so
$d_R(v_2)=d_R(v_3)=3-1=2$.
Every remaining vertex of $R$ has degree $4-1=3$.
Thus $\pi(R)$ is determined as desired.
Since $S=\varnothing$ and neither $v_2$ nor $v_3$ is an
attachment vertex, neither has a $2$-neighbor.
Thus $r=0$, and \eqref{eq46} gives $t=1$.
Hence,
$w^{(3)}(H)-w^{(3)}(H^*)=4t-2\varepsilon-2=2(1-\varepsilon)$.
If $v_2v_3\notin E(R)$, then $\varepsilon=0$, so
$w^{(3)}(H)>w^{(3)}(H^*)$, a contradiction.
Therefore, $v_2v_3\in E(R)$, and (ii) is completely proved.

(iii) We now prove (iii).
Let $Q$ be obtained from $R$ by adding a new vertex $v$
and the edges $vv_1,vv_2,vv_3$.
By (ii), the vertices $v_1,v_2,v_3$ have degree two in $R$,
and every other vertex of $R$ has degree three.
Thus every vertex of $Q$, including $v$, has degree three,
so $Q$ is a simple cubic graph.
It remains to construct the 3-edge-colouring required
in \eqref{eq43}.

We use the triangular embedding of
$\widehat G=K_2\nabla H_0$ inherited after deleting
the two pendant paths.
Recall that this deletion produces the faces
$u^*u^{**}v_0$ and $u^*u^{**}v_1$.
By Lemma~\ref{lem7.4}, $\mathbb{F}(v_0)=\mathbb{F}$, i.e.,
every face of the original embedding avoiding $u^*,u^{**}$ contains $v_0$.
Hence, every face of the embedding of $\widehat G$
is incident to at least one of $u^*,u^{**},v_0$.
Moreover, deleting $u^*,u^{**}$ from the $C_{v_0}$ in this embedding
leaves the path $P$ with endpoints $v_2$ and $v_3$.
Thus, after interchanging $v_2,v_3$ if necessary,
the other faces incident with $u^*v_0$ and $u^{**}v_0$
are $u^*v_0v_2$ and $u^{**}v_0v_3$, respectively.
Hence, the faces incident to at least two vertices
from $\{u^*,u^{**},v_0\}$ are exactly
\begin{equation}\label{eq50}
u^*u^{**}v_0,u^*u^{**}v_1,u^*v_0v_2,u^{**}v_0v_3.
\end{equation}

We next modify this embedding to obtain a triangular
embedding of $\overline K_3\nabla Q$ on the same surface.
Place the new vertex $v$ inside the face $u^*u^{**}v_0$
and join it to the three boundary vertices.
This replaces that face by
$vu^*u^{**}$, $vu^*v_0$, and $vu^{**}v_0$.
Then replace the edges $u^*u^{**}$, $u^*v_0$, $u^{**}v_0$
by $vv_1$, $vv_2$, $vv_3$, respectively.
Each replacement deletes the common edge of two 3-faces from (\ref{eq50}).
and inserts the other diagonal of the disk formed by those faces (see Figure \ref{fig.11}).
These changes preserve the triangular embedding.
Indeed, the faces $vu^*u^{**}$ and $v_1u^*u^{**}$
form an embedded disk with four distinct boundary vertices.
Replacing their common diagonal $u^*u^{**}$ by $vv_1$
replaces them by the faces $vu^*v_1$ and $vu^{**}v_1$,
without changing the boundary of the disk.
The other two replacements act in the same way on the pairs
$vu^*v_0$, $v_2u^*v_0$ and
$vu^{**}v_0$, $v_3u^{**}v_0$.
These pairs have disjoint interiors,
and each new edge is absent when its corresponding replacement is carried out.
Thus the operations preserve both simplicity and a triangular embedding on the same surface.

After these operations, $u^*,u^{**},v_0$ form an independent set, each adjacent to every vertex of $Q$,
and the remaining vertices induce $Q$.
The resulting graph is therefore
$\widetilde G=3K_1\nabla Q$.
The four faces in (\ref{eq50}) have been replaced by
$\{vu^*v_1,vu^{**}v_1,vu^*v_2,vv_0v_2,vu^{**}v_3,vv_0v_3\}.$
Each of these six faces contains exactly one vertex of
$\{u^*,u^{**},v_0\}$ and two vertices of $Q$.
All other faces are unchanged and have the same property.
We fix this triangular embedding of $\widetilde G$. In the remainder of the proof,
all faces and vertex rotations are taken with respect to this embedding.

\begin{figure}[H]
\definecolor{oldedge}{RGB}{174,45,51}
\definecolor{newedge}{RGB}{28,91,155}
\definecolor{patchink}{HTML}{263443}
\definecolor{patchboundary}{HTML}{405868}
\definecolor{patchvertex}{HTML}{42627B}
\definecolor{patchpaper}{HTML}{F8FAFC}
\definecolor{patchgray}{HTML}{DBE4EB}
\definecolor{patchblue}{HTML}{C5D9E8}
\definecolor{patchmint}{HTML}{D5E5DF}
\usetikzlibrary{arrows.meta}
\tikzset{
  boundary/.style={draw=patchboundary,line width=.8pt,line join=round},
  retained/.style={draw=patchboundary,line width=.75pt},
  original/.style={draw=oldedge,line width=1.1pt},
  removed/.style={draw=oldedge,line width=1.1pt,dash pattern=on 3pt off 2pt},
  inserted/.style={draw=newedge,line width=1.25pt},
  operation/.style={-{Latex[length=2.5mm,width=1.7mm]},
    draw=patchboundary,line width=.7pt},
  vertexlabel/.style={text=patchink,font=\normalsize,inner sep=1.8pt},
  patch vertex/.style={circle,ball color=patchvertex,draw=none,
    inner sep=0pt,minimum size=1.55mm},
  central vertex/.style={patch vertex,minimum size=1.7mm}}
\newcommand{\PatchCoordinates}{%
  \coordinate (A) at (-1.7,.98);
  \coordinate (B) at (1.7,.98);
  \coordinate (C) at (0,-1.96);
  \coordinate (D) at (0,1.96);
  \coordinate (E) at (-1.7,-.98);
  \coordinate (F) at (1.7,-.98);
  \coordinate (V) at (0,0);
  \coordinate (V1) at (0.125,0.25);}
\newcommand{\PatchBoundary}{%
  \draw[boundary] (A)--(D)--(B)--(F)--(C)--(E)--cycle;}
\newcommand{\PatchLabels}{%
  \foreach \p in {A,B,C,D,E,F}
    \node[patch vertex] at (\p) {};
  \node[vertexlabel,above left] at (A) {$u^*$};
  \node[vertexlabel,above right] at (B) {$u^{**}$};
  \node[vertexlabel,below] at (C) {$v_0$};
  \node[vertexlabel,above] at (D) {$v_1$};
  \node[vertexlabel,below left] at (E) {$v_2$};
  \node[vertexlabel,below right] at (F) {$v_3$};}
\newcommand{\PatchOuterFaces}{%
  \shade[top color=patchpaper,bottom color=patchblue,shading angle=0]
    (A)--(D)--(B)--cycle;
  \shade[left color=patchpaper,right color=patchblue,shading angle=18]
    (A)--(C)--(E)--cycle;
  \shade[left color=patchmint,right color=patchpaper,shading angle=-18]
    (B)--(F)--(C)--cycle;}
\centering
\begin{tikzpicture}[line cap=round,line join=round]
  \begin{scope}
    \PatchCoordinates
    \shade[top color=patchpaper,bottom color=patchgray,shading angle=15]
      (A)--(B)--(C)--cycle;
    \PatchOuterFaces
    \PatchBoundary
    \draw[original] (A)--(B)--(C)--cycle;
    \PatchLabels
  \end{scope}
  \begin{scope}[shift={(5,0)}]
    \PatchCoordinates
    \shade[top color=patchpaper,bottom color=patchgray,shading angle=0]
      (V)--(A)--(B)--cycle;
    \shade[left color=patchpaper,right color=patchgray,shading angle=18]
      (V)--(C)--(A)--cycle;
    \shade[left color=patchgray,right color=patchpaper,shading angle=-18]
      (V)--(B)--(C)--cycle;
    \PatchOuterFaces
    \PatchBoundary
    \draw[removed] (A)--(B)--(C)--cycle;
    \draw[retained] (V)--(A) (V)--(B) (V)--(C);
    \PatchLabels
    \node[central vertex] at (V) {};
    \node[vertexlabel] at (V1) {$v$};
  \end{scope}
  \begin{scope}[shift={(10,0)}]
    \PatchCoordinates
    \shade[left color=patchpaper,right color=patchblue,shading angle=18]
      (V)--(A)--(D)--cycle;
    \shade[left color=patchgray,right color=patchpaper,shading angle=-18]
      (V)--(D)--(B)--cycle;
    \shade[top color=patchpaper,bottom color=patchblue,shading angle=-18]
      (V)--(B)--(F)--cycle;
    \shade[left color=patchgray,right color=patchpaper,shading angle=-18]
      (V)--(F)--(C)--cycle;
    \shade[left color=patchpaper,right color=patchblue,shading angle=18]
      (V)--(C)--(E)--cycle;
    \shade[top color=patchpaper,bottom color=patchgray,shading angle=18]
      (V)--(E)--(A)--cycle;
    \PatchBoundary
    \draw[retained] (V)--(A) (V)--(B) (V)--(C);
    \draw[inserted] (V)--(D) (V)--(E) (V)--(F);
    \PatchLabels
    \node[central vertex] at (V) {};
    \node[vertexlabel] at (V1) {$v$};
  \end{scope}
  \draw[operation] (2.1,0)--(2.9,0);
\draw[operation] (7.1,0)--(7.9,0);
\end{tikzpicture}
\caption{An embedding transformation from $\widehat G=K_2\nabla H_0$ to $\widetilde G=3K_1\nabla Q$.}
\label{fig.11}
\end{figure}

We now use the faces of $\widetilde G$ to colour
the edges of its subgraph $Q$.
First, fix $x\in V(Q)$.
Its neighbors in $\widetilde G$ are
$u^*,u^{**},v_0$ and its three neighbors in $Q$.
Since every face is incident to exactly one in
$\{u^*,u^{**},v_0\}$, $C_x$ is a $6$-cycle
alternating between these three vertices and its three neighbors in $Q$.
In particular, the two faces incident with any edge
$xy\in E(Q)$ have distinct third vertices from $\{u^*,u^{**},v_0\}$.
For each edge $xy\in E(Q)$, assign colour $1$ if its
two incident faces in $\widetilde G$ are
$u^{**}xy$ and $v_0xy$,
colour $2$ if they are $u^*xy$ and $v_0xy$,
and colour $3$ if they are $u^*xy$ and $u^{**}xy$.
The alternating order in $C_x$ shows that
the three edges of $Q$ incident to $x$ correspond
to the three different pairs from $\{u^*,u^{**},v_0\}$.
Hence, they receive distinct colours.
Since this holds for every $x\in V(Q)$,
the colouring is a proper $3$-edge-colouring of $Q$.

Finally, each of $u^*,u^{**},v_0$ has neighborhood $V(Q)$ in $\widetilde G$,
so each $C_u$, $u\in \{u^*,u^{**},v_0\}$, is a cycle containing every vertex of $Q$ exactly once.
Its edges belong to $Q$, because consecutive vertices
in the $C_u$ form a 3-face with the corresponding vertex of $\{u^*,u^{**},v_0\}$.
Thus each of these three $C_u$ is a spanning cycle of $Q$.
An edge $xy\in E(Q)$ lies in $C_{u^*}$ if and only if
$u^*xy$ is a face of $\widetilde G$, which holds
if and only if its colour is $2$ or $3$.
Similarly, $C_{u^{**}}$ consists precisely of the edges
of colours $1$ and $3$, and $C_{v_0}$ consists precisely
of those of colours $1$ and $2$.
Therefore, any two colour classes form a spanning cycle of $Q$, proving \eqref{eq43}.
\end{proof}

\begin{lem}\label{lem7.6}
Let $G$ be a connected graph with Perron vector $\mathbf x$
and spectral radius $\rho(G)=\rho\ge6$.
Suppose that $G_B$ is constructed from $G$ by replacing
the induced subgraph $G[U]$, where
$\varnothing\ne U\subsetneq V(G)$ and $\Delta(G[U])\le3$,
with a graph on $U$ having the same degree at each vertex.
Let $A=A(G[U])$ and $B=A(G_B[U])$.
For a real number $c>0$, set
$\boldsymbol{\alpha}_1=\mathbf1-c^{-1}\mathbf x_U$
and define $\boldsymbol{\alpha}_2$ by
$(\rho \mathbf{I}_U-B)\boldsymbol{\alpha}_2
 =(\rho\mathbf{I}_U-A)\boldsymbol{\alpha}_1.$
If $\rho(G_B)\le\rho$, then
$\boldsymbol{\alpha}_1^{\top}(B-A)\boldsymbol{\alpha}_2\le0$.
\end{lem}

\begin{proof}
Since the replacement preserves every vertex degree,
$(A-B)\mathbf1=0$ and $\max\{\rho(A),\rho(B)\}\le\Delta(G[U])\leq3<\rho$.
Thus $\rho \mathbf{I}_U-B$ is invertible, so
$\boldsymbol{\alpha}_2$ is uniquely defined.
Its definition gives
\begin{equation}\label{eq51}
 (\rho \mathbf{I}_U-A)(\boldsymbol{\alpha}_2-\boldsymbol{\alpha}_1)
 =(B-A)\boldsymbol{\alpha}_2.
\end{equation}

Define $\widehat{\mathbf x}$ to equal $\mathbf x$ outside $U$
and $c(\mathbf1-\boldsymbol{\alpha}_2)$ on $U$.
By the Rayleigh principle and $\rho(G_B)\le\rho$, we obtain
$\widehat{\mathbf x}^{\top}(\rho \mathbf{I}-A(G_B))\widehat{\mathbf x}\ge0.$
Note that
$(\rho \mathbf{I}-A(G))\mathbf x=0$ and that
$G$ and $G_B$ differ only inside $U$.
Consequently,
\begin{equation}\label{eq52}
\widehat{\mathbf x}^{\top}
   \big(\rho \mathbf{I}-A(G_B) \big)\widehat{\mathbf x}=
   (\widehat{\mathbf x}-\mathbf x)^{\top}
    \big(\rho \mathbf{I}-A(G) \big)
   (\widehat{\mathbf x}-\mathbf x)
   +\widehat{\mathbf x}_U^{\top}(A-B)
    \widehat{\mathbf x}_U.
\end{equation}
Since $\mathbf x_U=c(\mathbf1-\boldsymbol{\alpha}_1)$,
the difference $\widehat{\mathbf x}-\mathbf x$ equals
$-c(\boldsymbol{\alpha}_2-\boldsymbol{\alpha}_1)$ on $U$
and vanishes elsewhere. Hence,
the first right-hand term of (\ref{eq52}) equals
$c^2(\boldsymbol{\alpha}_2-\boldsymbol{\alpha}_1)^{\top}
   (\rho \mathbf{I}_U-A)
   (\boldsymbol{\alpha}_2-\boldsymbol{\alpha}_1)$.
Since $(A-B)\mathbf1=0$, the second right-hand term of (\ref{eq52}) equals
$\widehat{\mathbf x}_U^{\top}(A-B)\widehat{\mathbf x}_U
 =c^2\boldsymbol{\alpha}_2^{\top}(A-B)\boldsymbol{\alpha}_2.$
Combining (\ref{eq51}) gives
\begin{align}
\widehat{\mathbf x}^{\top}\big(\rho \mathbf{I}-A(G_B)\big)\widehat{\mathbf x}
&= c^2\bigl(
   (\boldsymbol{\alpha}_2-\boldsymbol{\alpha}_1)^{\top}
   (\rho \mathbf{I}_U-A)
   (\boldsymbol{\alpha}_2-\boldsymbol{\alpha}_1)
   +\boldsymbol{\alpha}_2^{\top}(A-B)\boldsymbol{\alpha}_2
   \bigr)\notag\\
&=c^2\bigl(
   (\boldsymbol{\alpha}_2-\boldsymbol{\alpha}_1)^{\top}
   (B-A)\boldsymbol{\alpha}_2
   -\boldsymbol{\alpha}_2^{\top}(B-A)\boldsymbol{\alpha}_2
   \bigr)\notag\\
&=-c^2\boldsymbol{\alpha}_1^{\top}(B-A)\boldsymbol{\alpha}_2.\notag
\end{align}
Recall that $\widehat{\mathbf x}^{\top}(\rho \mathbf{I}-A(G_B))\widehat{\mathbf x}\ge0$.
Since $c>0$, it follows that
$\boldsymbol{\alpha}_1^{\top}(B-A)\boldsymbol{\alpha}_2\le0$.
\end{proof}

Before stating the next lemma, we note that all vector norms appearing below are Euclidean norms,
and all matrix norms are the associated operator norms.
The next lemma gives estimates for the two vectors in
Lemma \ref{lem7.6} and an expansion of
$\boldsymbol{\alpha}_1^{\top}(B-A)\boldsymbol{\alpha}_2$ for large $\rho$.

\begin{lem}\label{lem7.7}
Let $A,B$ be two real symmetric matrices of the same order with
$\max\{\|A\|_2,\|B\|_2\}\le3$.
For $\rho\ge6$ and a real vector $\boldsymbol{\alpha}_0$, define
$\boldsymbol{\alpha}_1,\boldsymbol{\alpha}_2$ by
$(\rho\mathbf I-A)\boldsymbol{\alpha}_1
 =(\rho\mathbf I-B)\boldsymbol{\alpha}_2=\boldsymbol{\alpha}_0$.
Then
\begin{equation}\label{eq53}
\big\|\boldsymbol{\alpha}_1-\rho^{-1}\boldsymbol{\alpha}_0
-\rho^{-2}A\boldsymbol{\alpha}_0\big\|_2
\le 9\rho^{-2}\big\|\boldsymbol{\alpha}_1\big\|_2
\le 9\rho^{-2}(\rho-3)^{-1}
\big\|\boldsymbol{\alpha}_0\big\|_2.
\end{equation}

The estimate remains valid if $(\boldsymbol{\alpha}_1,A)$ is replaced by $(\boldsymbol{\alpha}_2,B)$.
Moreover, as $\rho\to\infty$,
\begin{equation}\label{eq54}
\boldsymbol{\alpha}_1^{\top}(B-A)\boldsymbol{\alpha}_2
=\rho^{-2}\boldsymbol{\alpha}_0^{\top}(B-A)\boldsymbol{\alpha}_0
+\rho^{-3}\boldsymbol{\alpha}_0^{\top}(B^2-A^2)\boldsymbol{\alpha}_0
+R_\rho,
\end{equation}
where $|R_\rho|\le
108\rho^{-4}\big\|\boldsymbol{\alpha}_0\big\|_2^2$.
\end{lem}

\begin{proof}
Since $\max\{\|A\|_2,\|B\|_2\}\le3<\rho$, both $\rho\mathbf I-A$ and
$\rho\mathbf I-B$ are invertible.
Taking norms in
$\rho\boldsymbol{\alpha}_1=\boldsymbol{\alpha}_0+A\boldsymbol{\alpha}_1$
gives
$\rho\|\boldsymbol{\alpha}_1\|_2
 \le\|\boldsymbol{\alpha}_0\|_2+3\|\boldsymbol{\alpha}_1\|_2$, and thus
 $\|\boldsymbol{\alpha}_1\|_2\le(\rho-3)^{-1}\|\boldsymbol{\alpha}_0\|_2.$
Multiplying the defining identity
$(\rho\mathbf I-A)\boldsymbol{\alpha}_1=\boldsymbol{\alpha}_0$ by $\rho$ yields
\begin{equation}\label{eq55}
 \rho^2\boldsymbol{\alpha}_1=\rho\boldsymbol{\alpha}_0+A\rho\boldsymbol{\alpha}_1
 =\rho\boldsymbol{\alpha}_0+A\boldsymbol{\alpha}_0+A^2\boldsymbol{\alpha}_1.
\end{equation}
Dividing by $\rho^2$ and using $\|A^2\|_2\le9$ gives \eqref{eq53}.
The same argument applies to $(\boldsymbol{\alpha}_2,B)$.

For \eqref{eq54}, symmetry and the defining equations give
\begin{equation}\label{eq56}
\boldsymbol{\alpha}_1^{\top}(B-A)\boldsymbol{\alpha}_2
 =\boldsymbol{\alpha}_1^{\top}
   \bigl((\rho\mathbf I-A)-(\rho\mathbf I-B)\bigr)\boldsymbol{\alpha}_2
 =\boldsymbol{\alpha}_0^{\top}
   (\boldsymbol{\alpha}_2-\boldsymbol{\alpha}_1).
\end{equation}
Multiplying \eqref{eq55} by $\rho$
and substituting $\rho\boldsymbol{\alpha}_1=\boldsymbol{\alpha}_0+A\boldsymbol{\alpha}_1$ yields
$\rho^3\boldsymbol{\alpha}_1
 =(\rho^2\mathbf I+\rho A+A^2)\boldsymbol{\alpha}_0
  +A^3\boldsymbol{\alpha}_1.$
The same calculation for $(\boldsymbol{\alpha}_2,B)$ yields
$\rho^3\boldsymbol{\alpha}_2
=(\rho^2\mathbf I+\rho B+B^2)\boldsymbol{\alpha}_0
+B^3\boldsymbol{\alpha}_2$.
Subtracting the identity for $\boldsymbol{\alpha}_1$ from that for
$\boldsymbol{\alpha}_2$ cancels the common term
$\rho^2\boldsymbol{\alpha}_0$, giving
\[
 \rho^3(\boldsymbol{\alpha}_2-\boldsymbol{\alpha}_1)
 =\rho(B-A)\boldsymbol{\alpha}_0+(B^2-A^2)\boldsymbol{\alpha}_0
   +B^3\boldsymbol{\alpha}_2-A^3\boldsymbol{\alpha}_1.
\]
Multiplying on the left by $\rho^{-3}\boldsymbol{\alpha}_0^{\top}$
and combining \eqref{eq56} for
$\boldsymbol{\alpha}_0^{\top}(\boldsymbol{\alpha}_2-\boldsymbol{\alpha}_1)$,
we obtain
$$
 \boldsymbol{\alpha}_1^{\top}(B-A)\boldsymbol{\alpha}_2
 =\rho^{-2}\boldsymbol{\alpha}_0^{\top}(B-A)\boldsymbol{\alpha}_0
  +\rho^{-3}\boldsymbol{\alpha}_0^{\top}(B^2-A^2)\boldsymbol{\alpha}_0
+\rho^{-3}\boldsymbol{\alpha}_0^{\top}
       (B^3\boldsymbol{\alpha}_2-A^3\boldsymbol{\alpha}_1).
$$
To prove \eqref{eq54}, it remains to bound the last term.
Recall that $\max\{\|A^3\|_2,\|B^3\|_2\}\le27$
and $\|\boldsymbol{\alpha}_j\|_2\le
(\rho-3)^{-1}\|\boldsymbol{\alpha}_0\|_2$ for $j=1,2$.
By Cauchy--Schwarz inequality, we obtain
$$\rho^{-3}\bigl|\boldsymbol{\alpha}_0^{\top}
      (B^3\boldsymbol{\alpha}_2-A^3\boldsymbol{\alpha}_1)\bigr|
 \le 27\rho^{-3}\|\boldsymbol{\alpha}_0\|_2
      \bigl(\|\boldsymbol{\alpha}_2\|_2+\|\boldsymbol{\alpha}_1\|_2\bigr)
 \le54\rho^{-3}(\rho-3)^{-1}\|\boldsymbol{\alpha}_0\|_2^2.$$
Since $(\rho-3)\geq\frac12\rho$,
we have $\rho^{-3}\bigl|\boldsymbol{\alpha}_0^{\top}
(B^3\boldsymbol{\alpha}_2-A^3\boldsymbol{\alpha}_1)\bigr| \le108\rho^{-4}\|\boldsymbol{\alpha}_0\|_2^2.$
This proves \eqref{eq54}.
\end{proof}

Call a graph with three distinct marked vertices
\emph{admissible} if adding a new vertex adjacent to the
three marked vertices produces a simple cubic graph with a
proper $3$-edge-colouring in which any two colour classes
form a spanning cycle, as in \eqref{eq43}.
The marked vertices have degree two, while all other
vertices have degree three.

We shall use the following consequences of Lemma~\ref{lem7.1}.
Let $T$ be an admissible graph with marked vertices
$r_1,r_2,r_3$ and $|T|>3$.
The marked vertices span at most one edge.
If $r_ir_j$ is such an edge, where
$\{i,j,k\}=\{1,2,3\}$, write
$N_T(r_i)=\{r_j,p_i\}$ and $N_T(r_j)=\{r_i,p_j\}$.
Then $p_i,p_j,r_k$ are distinct, and deleting $r_i,r_j$
leaves an admissible graph whose marked vertices are
$p_i,p_j,r_k$, carrying labels $i,j,k$, respectively.
Conversely, given such an admissible graph, adding two new
vertices $r_i,r_j$ and three edges $p_ir_i,r_ir_j,r_jp_j$
produces an admissible graph with marked vertices $r_i,r_j,r_k$.
This inverse operation is unique once the correspondence
between the new vertices and their external neighbors is fixed.
Starting from a triangle, these expansions give admissible
graphs of every odd order at least three; at every larger
odd order, any prescribed pair of marked vertices can be made adjacent.

For the rest of the proof, fix an extremal graph
$G^*=K_2\nabla H^*$ with base graph $H_0=K_1\nabla R$
and attachment vertices $v_0,v_1$, as specified in Lemma \ref{lem7.5}.
The constants in $O_\gamma(\cdot)$ depend only on $\gamma$,
uniformly in the lengths of the two pendant paths.

\begin{definition}\label{def7.1}
Initially, the marked vertices $v_1,v_2,v_3$ of $R$
carry labels $1,2,3$, respectively.
Let $T$ be an admissible graph obtained during the deletion
process, and let $r_i$ denote its current vertex carrying
label $i$.
Suppose that $r_ir_j\in E(T)$, where
$\{i,j,k\}=\{1,2,3\}$, and write
$ N_T(r_i)=\{r_j,p_i\}$,
$N_T(r_j)=\{r_i,p_j\}.$
Delete $r_i,r_j$, assign labels $i,j$ to $p_i,p_j$,
respectively, and retain label $k$ at $r_k$.
We call this operation \emph{removing the marked pair $ij$}.
Any graph obtained from $R$ by zero or more such operations
is called a \emph{marked remainder} of $R$.
Thus initially $r_i=v_i$ for $i=1,2,3$.
\end{definition}

Lemma~\ref{lem7.1} ensures that the label-inheritance rule
in Definition~\ref{def7.1} is well defined.
An \emph{admissible replacement} of a marked remainder $T$
is an admissible graph $\widetilde T$ satisfying
$V(\widetilde T)=V(T),$
such that $T$ and $\widetilde T$ have the same three marked
vertices with the same labels.
Thus, $\widetilde T$ is obtained by changing only the edges
inside $V(T)$, and no vertex is added or deleted.

Set $\rho:=\rho(G^*)$, and let $\mathbf x=(x_u)_{u\in V(G^*)}$ be the Perron vector
of $G^*$. The eigenvalue equations at the two dominating vertices give
$(\rho+1)(x_{u^*}-x_{u^{**}})=0$.
We normalize $\mathbf x$ so that
$x_{u^*}=x_{u^{**}}=\frac12$.
Since $G^*$ contains a copy of $K_{2,n-2}$,
we have $\rho\ge\sqrt{2(n-2)}\ge6$ for sufficiently large $n$.

Set $\lambda:=(1+x_{v_0})/(\rho-3)$ and $\mathbf y:=(y_u)_{u\in V(R)}$ by
$y_u=1-\lambda^{-1}x_u$.
For a marked remainder $T$ of $R$ with adjacency matrix $A$,
define the \emph{standard boundary vector}
$\boldsymbol{\beta}_T=(\rho\mathbf I-A)\mathbf y|_{V(T)}$.
Any positive multiple of $\boldsymbol{\beta}_T$ is called a
\emph{boundary vector} of $T$.
For a vector $\boldsymbol{\alpha}_\ell$ indexed by $V(T)$,
where $\ell\in\{0,1,2\}$, write
$\alpha_{\ell,i}=(\boldsymbol{\alpha}_\ell)_{r_i}$ for its
coordinate at the current vertex carrying label $i$.

\begin{lem}\label{lem7.8}
Let $T$ be a marked remainder with adjacency matrix $A$,
and let $\widetilde T$ be an admissible replacement of $T$ with adjacency matrix $B$.
For any boundary vector $\boldsymbol{\alpha}_0$ of $T$,
define $\boldsymbol{\alpha}_1,\boldsymbol{\alpha}_2$ by
$(\rho\mathbf I-A)\boldsymbol{\alpha}_1
=(\rho\mathbf I-B)\boldsymbol{\alpha}_2=\boldsymbol{\alpha}_0.$
Then $\boldsymbol{\alpha}_0$ vanishes on all unmarked vertices of $V(T)$ and
\begin{equation}\label{eq57}
\boldsymbol{\alpha}_1^{\top}(B-A)\boldsymbol{\alpha}_2\le 0.
\end{equation}
When $T=R-\{v_2,v_3\}$, there exists a boundary vector $\boldsymbol{\alpha}_0$ of $T$ satisfying
\begin{equation}\label{eq58}
\alpha_{0,1}=2+O_\gamma(\rho^{-1}),~
\alpha_{0,2}=1+O_\gamma(\rho^{-2}),~
\alpha_{0,3}=1+O_\gamma(\rho^{-2}).
\end{equation}
\end{lem}

\begin{proof}
Let $a,b$ be the lengths of the pendant paths attached at $v_0,v_1$, respectively.
The balance from Claim~\ref{cl5.5} together with $a+b=n-2\gamma-4$ implies $a,b\ge 2$ for sufficiently large $n$.
These lengths remain fixed in all comparisons below.

The eigenvalue equation at an arbitrary vertex $u\in V(H^*)$ gives
$\rho x_{u}=1+\sum_{w\in N_{H^*}(u)}x_w.$
Let $x_{v^*}:=\max_{u\in V(H^*)}x_u$.
Then $\rho x_{v^*}\leq1+(2\gamma+2)x_{v^*},$
since $\Delta(H^*)=d_{H^*}(v_0)=2\gamma+2$ by Lemma \ref{lem7.4}.
Thus $x_{v^*}\le(\rho-2\gamma-2)^{-1}$.
Iterating the eigenvalue equation at $u\in V(H^*)$ once gives
$$\rho^2x_u=\rho+d_{H^*}(u)+\sum_{w\in N_{H^*}(u)}\sum_{v\in N_{H^*}(w)}x_v.$$
The double sum contains at most $(2\gamma+2)^2$ terms and $x_v\leq x_{v^*}\le(\rho-2\gamma-2)^{-1}$. Hence,
$x_u=\rho^{-1}+d_{H^*}(u)\rho^{-2}+O_\gamma(\rho^{-3})$ for any $u\in V(H^*)$.

Let $w_1$ be the first vertex after $v_1$ on its pendant path, and set $\eta:=1-\lambda^{-1}x_{w_1}$.
Since $d_{H^*}(w_1)=2$, we obtain $x_{w_1}=\rho^{-1}+2\rho^{-2}+O_\gamma(\rho^{-3}).$
Moreover, we have $x_{v_0}=\rho^{-1}+O_\gamma(\rho^{-2})$.
This, together with $\lambda=(1+x_{v_0})/(\rho-3)$, gives
$\lambda=\rho^{-1}+4\rho^{-2}+O_\gamma(\rho^{-3}).$
It follows that
\[
 \eta
 =1-\frac{1+2\rho^{-1}+O_\gamma(\rho^{-2})}
          {1+4\rho^{-1}+O_\gamma(\rho^{-2})}
 =2\rho^{-1}+O_\gamma(\rho^{-2}).
\]
For $u\in V(R)$, the eigenvalue equation at $u$ becomes
$\rho x_u
 =1+x_{v_0}+\sum_{w\in N_R(u)}x_w
  +x_{w_1}\cdot\mathbf1_{\{u=v_1\}}.$
Recall that $x_u=\lambda(1-y_u)$ and $(\rho-3)\lambda=1+x_{v_0}$. Substituting gives
\[
 \rho y_u
 =\sum_{w\in N_R(u)}y_w
  +3-d_R(u)-\lambda^{-1}x_{w_1}\cdot\mathbf1_{\{u=v_1\}}.
\]
By Lemma \ref{lem7.5} and Definition \ref{def7.1},
the three marked vertices $v_1,v_2,v_3$ are 2-vertices in the initial marked remainder $R$.
Thus for the initial vector $\boldsymbol{\beta}^0:=\boldsymbol{\beta}_R=\big(\rho\mathbf I-A(R)\big)\mathbf y$,
its entries at the marked vertices $v_1,v_2,v_3$ are $\eta,1,1$, respectively,
while its entry at every unmarked vertex $u$ satisfies
\begin{equation}\label{eq59}
\big(\boldsymbol{\beta}^0\big)_{u}=0.
\end{equation}

We first take $\boldsymbol{\alpha}_0$ to be the standard boundary vector $\boldsymbol{\beta}_T=(\rho\mathbf I-A)\mathbf y|_{V(T)}$.
Write $T=R-S$, where $S$ is the set of vertices already removed.
For every $u\in V(T)$, comparing the definitions of
$\boldsymbol{\beta}_T$ and
$\boldsymbol{\beta}^0=\boldsymbol{\beta}_R$ gives
\begin{equation}\label{eq60}
 (\boldsymbol{\alpha}_0)_u
 =\rho y_u-\!\!\sum_{w\in N_T(u)}\!\!y_w
 =\big(\boldsymbol{\beta}^0\big)_{u}+\!\!\sum_{w\in S\cap N_R(u)}\!\!y_w.
\end{equation}
Since
$\boldsymbol{\alpha}_0
=(\rho\mathbf I-A)\mathbf y|_{V(T)}$
and $\|A\|_2\le\Delta(R)=3<\rho$, it follows that $\rho\mathbf I-A$ is invertible, and the defining equation
$(\rho\mathbf I-A)\boldsymbol{\alpha}_1=\boldsymbol{\alpha}_0$ for $\boldsymbol{\alpha}_1$ has the unique solution
$\boldsymbol{\alpha}_1=\mathbf y|_{V(T)}$.

We prove by induction on the number of removals that
$\boldsymbol{\beta}_T$ vanishes at every unmarked vertex.
For $T=R$, this follows from \eqref{eq59}.
Suppose that the assertion holds for $T$ and that the
marked pair $ij$ is removed.
Let $p_i,p_j,r_k$ be as in Definition~\ref{def7.1}.
By Lemma~\ref{lem7.1}, these vertices are distinct.
Neither $p_i$ nor $p_j$ is marked in $T$, since otherwise
the marked vertices would span a second edge.
For the same reason, $r_k$ is adjacent to neither
$r_i$ nor $r_j$.

Set $T':=T-\{r_i,r_j\}$ and
$\boldsymbol{\alpha}_0'
:=\boldsymbol{\beta}_{T'}$.
By Definition~\ref{def7.1}, the marked vertices of $T'$
carrying labels $i,j,k$ are $p_i,p_j,r_k$, respectively.
By \eqref{eq60}, deleting $r_i,r_j$ changes the
coordinates at the remaining vertices only at $p_i,p_j$.
Their old coordinates are zero by the induction hypothesis,
whereas their new coordinates are
$y_{r_i}=\alpha_{1,i}$ and
$y_{r_j}=\alpha_{1,j}$.
The coordinate at $r_k$ remains $\alpha_{0,k}$.
Therefore,
$\alpha'_{0,i}=\alpha_{1,i}$,
$\alpha'_{0,j}=\alpha_{1,j}$ and
$\alpha'_{0,k}=\alpha_{0,k}$,
while every coordinate of $\boldsymbol{\alpha}_0'$ corresponding to other unmarked vertices is zero.
This completes the induction.

Since $\boldsymbol{\alpha}_0$ vanishes at every unmarked
vertex and
$N_T(r_i)=\{r_j,p_i\}$ and
$N_T(r_j)=\{r_i,p_j\}$, we have
$(A\boldsymbol{\alpha}_0)_{r_i}=\alpha_{0,j}$ and
$(A\boldsymbol{\alpha}_0)_{r_j}=\alpha_{0,i}$.
Applying \eqref{eq53} to
$(\rho\mathbf I-A)\boldsymbol{\alpha}_1
=\boldsymbol{\alpha}_0$
and using the coordinate identities above gives
\begin{equation}\label{eq61}
\begin{aligned}
 \alpha'_{0,i}
 &=\rho^{-1}\alpha_{0,i}
   +\rho^{-2}\alpha_{0,j}
   +O\bigl(\rho^{-3}
           \|\boldsymbol{\alpha}_0\|_2\bigr),\\
 \alpha'_{0,j}
 &=\rho^{-1}\alpha_{0,j}
   +\rho^{-2}\alpha_{0,i}
   +O\bigl(\rho^{-3}
           \|\boldsymbol{\alpha}_0\|_2\bigr),\\
 \alpha'_{0,k}&=\alpha_{0,k}.
\end{aligned}
\end{equation}
Here and below, the implicit constant in this update is
absolute.

We next prove \eqref{eq57} for the choice
$\boldsymbol{\alpha}_0=\boldsymbol{\beta}_T$.
Let $\widetilde T$ be an admissible replacement of $T$,
with adjacency matrix $B$.
Reverse the recorded removals, preserving at each step the
correspondence between the restored marked pair and its two
external neighbors.
By part (iii) of Lemma \ref{lem7.1}, this produces an admissible graph
$\widetilde R$ on $V(R)$ such that
$\widetilde R[V(T)]=\widetilde T$ and every edge of $R$ not
having both endpoints in $V(T)$ is restored unchanged.

Replace $R$ by $\widetilde R$ in $G^*$, keeping $v_0$,
the two pendant paths and the two dominating vertices
unchanged. Denote by $G_B$ the resulting graph.
Thus $G_B$ is obtained from $G^*$ by replacing the subgraph $T$ with $\widetilde T$.
By Lemma~\ref{lem7.3},
$G_B\in\mathbb G(n,\gamma)$.
The maximality of $\rho(G^*)$ gives $\rho(G_B)\le\rho$.

Let $U=V(T)$.
Since $T$ and $\widetilde T$ are admissible with the same
marked vertices, both have degree two at $r_1,r_2,r_3$
and degree three at every other vertex.
Hence, the replacement preserves every degree in $U$, and
both induced subgraphs have maximum degree at most three.
Moreover, $\varnothing\ne U\subsetneq V(G^*)$ and
$\boldsymbol{\alpha}_1
 =\mathbf y|_U
 =\mathbf1-\lambda^{-1}\mathbf x_U.$
Together with the connectedness of $G^*$,
$\rho\ge6$, and $\rho(G_B)\le\rho$,
these facts verify all the hypotheses of Lemma \ref{lem7.6}.
It follows that
$\boldsymbol{\alpha}_1^{\top}
 (B-A)\boldsymbol{\alpha}_2\le0,$
which proves \eqref{eq57} for
$\boldsymbol{\alpha}_0=\boldsymbol{\beta}_T$.

Now let
$\boldsymbol{\alpha}_0=c\boldsymbol{\beta}_T$, where
$c>0$, be an arbitrary boundary vector.
By linearity, the corresponding vectors
$\boldsymbol{\alpha}_1$ and $\boldsymbol{\alpha}_2$ are
$c$ times those corresponding to
$\boldsymbol{\beta}_T$.
Hence
$\boldsymbol{\alpha}_1^{\top}(B-A)\boldsymbol{\alpha}_2$
is multiplied by $c^2$, and so
\eqref{eq57} remains valid.
Furthermore, since $\boldsymbol{\beta}_T$ vanishes at every
unmarked vertex, so does
$\boldsymbol{\alpha}_0=c\boldsymbol{\beta}_T$.
If removing the marked pair $ij$ produces $T'$, take
$\boldsymbol{\alpha}_0'=c\boldsymbol{\beta}_{T'}$.
Then all coordinates in the exact update and its expansion
are multiplied by $c$, so \eqref{eq61} remains valid.

Finally, Lemma~\ref{lem7.5} gives $v_2v_3\in E(R)$.
Remove the marked pair $23$ from the initial boundary vector
whose marked coordinates are $(\eta,1,1)$.
Since label $1$ is retained, \eqref{eq61} gives
$\alpha'_{0,1}=\eta
=2\rho^{-1}+O_\gamma(\rho^{-2})$, while
$\alpha'_{0,\ell}=\rho^{-1}+\rho^{-2}+O_\gamma(\rho^{-3})$ for $\ell\in\{2,3\}$.
After multiplying by $\rho/(1+\rho^{-1})$, these
coordinates become $2+O_\gamma(\rho^{-1}),$ $1+O_\gamma(\rho^{-2})$ and $1+O_\gamma(\rho^{-2}),$
respectively, proving \eqref{eq58}.
The estimate for $\eta$ depends only on $\gamma$, whereas
the implicit constant in \eqref{eq61} is absolute.
Hence, the constants in \eqref{eq58} depend only on $\gamma$,
uniformly in the lengths of the two pendant paths.
\end{proof}

\begin{lem}\label{lem7.9}
Up to interchanging $v_2$ and $v_3$, or equivalently,
labels $2$ and $3$, the graph $R$ can be reduced to a
triangle by successively removing the marked pairs
\begin{equation}\label{eq62}
 23,\ 12,\ 13,\ 23,\ 12,\ 13,\ldots,
\end{equation}
with labels inherited according to Definition~\ref{def7.1}.
\end{lem}

\begin{proof}
By Lemma \ref{lem7.5}, $v_2v_3\in E(R)$, so the first
marked pair removed is $23$.
Lemma \ref{lem7.8} shows that, after this removal and multiplication by the positive scalar $\rho/(1+\rho^{-1})$,
the resulting marked remainder admits a boundary vector $\boldsymbol{\alpha}_0$ satisfying
\begin{equation}\label{eq63}
\begin{aligned}
\alpha_{0,1}&=\tau+O_\gamma(\rho^{-1}),\\
\alpha_{0,2}&=1+\delta\rho^{-1}+O_\gamma(\rho^{-2}),\\
\alpha_{0,3}&=1+O_\gamma(\rho^{-2}),
\end{aligned}
\end{equation}
where initially $(\tau,\delta)=(2,0)$.

We prove that whenever \eqref{eq63} holds, the next
three marked pairs are $12,13,23$. After these removals,
the resulting boundary vector can be positively rescaled
so that \eqref{eq63} holds again with
$(\tau,\delta)\longmapsto(\tau+1,\delta+\tau).$
Here $\tau\ge2$ and $\delta\ge0$ depend only on $\gamma$.
If $\delta=0$, labels $2$ and $3$ may be interchanged.
At any stage, the process stops if the remainder is a triangle.
Let $T$ be the current marked remainder with adjacency matrix $A$.
Set $\varepsilon_{ij}:=\mathbf1_{\{r_ir_j\in E(T)\}}.$
Since every coordinate of $\boldsymbol{\alpha}_0$ outside
the three marked vertices $r_1,r_2,r_3$ is zero, we have
\begin{equation}\label{eq64}
 \boldsymbol{\alpha}_0^{\T}A\boldsymbol{\alpha}_0
 =2\alpha_{0,1}\alpha_{0,2}\varepsilon_{12}
  +2\alpha_{0,1}\alpha_{0,3}\varepsilon_{13}
  +2\alpha_{0,2}\alpha_{0,3}\varepsilon_{23}.
\end{equation}
For an admissible replacement $\widetilde T$ of $T$ with adjacency
matrix $B$, we define $\boldsymbol{\alpha}_1,\boldsymbol{\alpha}_2$
by $(\rho\mathbf I-A)\boldsymbol{\alpha}_1
=(\rho\mathbf I-B)\boldsymbol{\alpha}_2
=\boldsymbol{\alpha}_0$.
Lemma \ref{lem7.8} gives
$\boldsymbol{\alpha}_1^{\T}(B-A)\boldsymbol{\alpha}_2\le0$,
while Lemma \ref{lem7.7} provides the expansion used below
to test this inequality.

{\bf We first determine the marked edge following the initial removal of $23$.}
By \eqref{eq63} and \eqref{eq64},
\[
 \boldsymbol{\alpha}_0^{\T}A\boldsymbol{\alpha}_0
 =2\bigl(
   \tau(\varepsilon_{12}+\varepsilon_{13})
   +\varepsilon_{23}
   \bigr)+O_\gamma(\rho^{-1}).
\]
If neither $12$ nor $13$ were present, choose an admissible
replacement containing $12$.
The expression in brackets would increase from at most $1$
to $\tau\ge2$.
Since $\|\boldsymbol{\alpha}_0\|_2=O_\gamma(1)$,
\eqref{eq54} would give
\[
 \boldsymbol{\alpha}_1^{\T}
 (B-A)\boldsymbol{\alpha}_2
 \ge 2(\tau-1)\rho^{-2}
      +O_\gamma(\rho^{-3})>0
\]
for sufficiently large $n$, contradicting
\eqref{eq57}.
Thus the marked edge is either $12$ or $13$.

Suppose that $\delta>0$ and the marked edge is $13$.
Let $\widetilde T$ be the admissible replacement obtained by
interchanging the roles of $r_2$ and $r_3$ in $T$, while
retaining their labels.
Thus the marked edge of $\widetilde T$ is $12$.
Let $P$ be the permutation matrix that interchanges the
standard basis vectors $\mathbf e_{r_2},\mathbf e_{r_3}$
and fixes $\mathbf e_u$ for every
$u\in V(T)\setminus\{r_2,r_3\}$.
Then $B:=A(\widetilde T)=PAP^{\T}$.
Let $\mathbf c$ be the vector whose coordinates at
$r_1,r_2,r_3$ are $\tau,1,1$, respectively, and zero elsewhere.
Since $P\mathbf c=\mathbf c$, we have
$B\mathbf c=PAP^{\T}\mathbf c=PA\mathbf c,$
and hence $\|B\mathbf c\|_2=\|A\mathbf c\|_2$.
In view of \eqref{eq63},
$\|\boldsymbol{\alpha}_0-\mathbf c\|_2
=O_\gamma(\rho^{-1})$.
Since $\max\{\|A\|_2,\|B\|_2\}\le3$, it follows that
$\|B\boldsymbol{\alpha}_0\|_2^2
 -\|A\boldsymbol{\alpha}_0\|_2^2
 =O_\gamma(\rho^{-1}).$
On the other hand, $A$ has marked edge $13$, whereas
$B$ has marked edge $12$. Therefore,
\[
 \boldsymbol{\alpha}_0^{\T}
 (B-A)\boldsymbol{\alpha}_0
 =2\alpha_{0,1}
   (\alpha_{0,2}-\alpha_{0,3})
 =2\tau\delta\rho^{-1}
  +O_\gamma(\rho^{-2}).
\]
Since $A$ and $B$ are symmetric,
$\boldsymbol{\alpha}_0^{\T}
 (B^2-A^2)\boldsymbol{\alpha}_0
 =\|B\boldsymbol{\alpha}_0\|_2^2
  -\|A\boldsymbol{\alpha}_0\|_2^2
 =O_\gamma(\rho^{-1}).$
Hence \eqref{eq54} gives
$\boldsymbol{\alpha}_1^{\T}
 (B-A)\boldsymbol{\alpha}_2
 =2\tau\delta\rho^{-3}
  +O_\gamma(\rho^{-4})>0,$
contradicting \eqref{eq57}.
Thus the marked edge must be $12$.

If $\delta=0$, then both $\alpha_{0,2}$ and
$\alpha_{0,3}$ equal $1+O_\gamma(\rho^{-2})$.
Hence, after interchanging labels $2$ and $3$ if necessary,
we may assume that the marked edge is $12$.
This interchange preserves the form of \eqref{eq63},
because it merely interchanges these two coordinates.

Now, whether $\delta>0$ or not, the marked pair is exactly 12.
We remove this marked pair.
By the update formula \eqref{eq61}, the new boundary
vector $\boldsymbol{\alpha}_0'$ has marked coordinates
\[
 \begin{aligned}
 \alpha'_{0,1}
 &=\tau\rho^{-1}+O_\gamma(\rho^{-2}),\\
 \alpha'_{0,2}
 &=\rho^{-1}+(\delta+\tau)\rho^{-2}
   +O_\gamma(\rho^{-3}),\\
 \alpha'_{0,3}
 &=1+O_\gamma(\rho^{-2}).
 \end{aligned}
\]

{\bf We next prove that, after removing $12$, the marked edge is $13$.}
Let $A$ now denote the adjacency matrix of the new marked
remainder.
Using the preceding coordinates in
\eqref{eq64} gives
$(\boldsymbol{\alpha}_0')^{\T}
 A\boldsymbol{\alpha}_0'
 =2\rho^{-1}
  (\tau\varepsilon_{13}+\varepsilon_{23})
  +O_\gamma(\rho^{-2}).$
Moreover,
since every coordinate of $\boldsymbol{\alpha}_0'$ outside
the three marked vertices $r_1,r_2,r_3$ is zero,
the preceding estimates imply
$\bigl\|\boldsymbol{\alpha}_0'
       -\mathbf e_{r_3}\bigr\|_2
 =O_\gamma(\rho^{-1}).$
For every admissible replacement $\widetilde T$ with
adjacency matrix $B$, the vertex $r_3$ has degree two in
both $T$ and $\widetilde T$. Hence
$\|A\mathbf e_{r_3}\|_2^2
 =\|B\mathbf e_{r_3}\|_2^2=2.$
Since $\max\{\|A\|_2,\|B\|_2\}\le3$, it follows that
$\|A\boldsymbol{\alpha}_0'\|_2^2
 =\|B\boldsymbol{\alpha}_0'\|_2^2
 =2+O_\gamma(\rho^{-1}).$
Thus the leading terms coming from
$B^2-A^2$ cancel.

If $13$ were absent, choose an admissible replacement
containing $13$.
The coefficient
$\tau\varepsilon_{13}+\varepsilon_{23}$ would increase from
at most $1$ to $\tau\ge2$.
Equation~\eqref{eq54} would then give
\[
 \boldsymbol{\alpha}_1^{\T}
 (B-A)\boldsymbol{\alpha}_2
 \ge2(\tau-1)\rho^{-3}
    +O_\gamma(\rho^{-4})>0,
\]
contradicting \eqref{eq57}.
Hence, $13$ is present.
Remove the marked pair $13$ and multiply the resulting
boundary vector by $\rho$.
The update formula gives
\[
 \begin{aligned}
 \alpha''_{0,1}
 &=(\tau+1)\rho^{-1}+O_\gamma(\rho^{-2}),\\
 \alpha''_{0,2}
 &=1+(\delta+\tau)\rho^{-1}
     +O_\gamma(\rho^{-2}),\\
 \alpha''_{0,3}
 &=1+O_\gamma(\rho^{-2}).
 \end{aligned}
\]

{\bf Finally, we prove that, after removing $13$, the
marked edge is $23$.}
Now $\boldsymbol{\alpha}_0''
 =(0,1,1)+O_\gamma(\rho^{-1}),$
and hence
$(\boldsymbol{\alpha}_0'')^{\T}
 A\boldsymbol{\alpha}_0''
 =2\varepsilon_{23}+O_\gamma(\rho^{-1}).$
If $23$ were absent, an admissible replacement containing
$23$ would give, by \eqref{eq54},
$\boldsymbol{\alpha}_1^{\T}
 (B-A)\boldsymbol{\alpha}_2
 =2\rho^{-2}+O_\gamma(\rho^{-3})>0,$
contradicting \eqref{eq57}.
Therefore, $23$ is present.
Remove the marked pair $23$ and multiply the new boundary
vector by the positive factor
$\rho/(1+\rho^{-1}).$
Using \eqref{eq61}, its marked coordinates become
\[
 \begin{aligned}
 \alpha'''_{0,1}
 &=\tau+1+O_\gamma(\rho^{-1}),\\
 \alpha'''_{0,2}
 &=1+(\delta+\tau)\rho^{-1}
     +O_\gamma(\rho^{-2}),\\
 \alpha'''_{0,3}
 &=1+O_\gamma(\rho^{-2}).
 \end{aligned}
\]
Thus \eqref{eq63} is recovered with
$(\tau',\delta')=(\tau+1,\delta+\tau).$

In the first application of \eqref{eq63}, we have
$\delta=0$, so the argument determines the next marked edge
only up to $12$ or $13$.
If it is $13$, we interchange the names of the original
vertices $v_2$ and $v_3$.
By the label-inheritance rule in Definition \ref{def7.1}, this interchanges labels
$2$ and $3$ in every subsequent marked remainder while
leaving label $1$, and hence the attachment vertex $v_1$,
unchanged.
We may therefore assume that the marked edge following the
initial removal of $23$ is $12$.

After $r$ complete blocks $12,13,23$, the recurrence
$(\tau,\delta)\mapsto(\tau+1,\delta+\tau)$ gives
$(\tau,\delta)=(r+2,\frac12r(r+3)).$
In particular, $\delta>0$ for every $r\ge1$.
Hence, after the first block, the argument uniquely forces
the next marked edge to be $12$, and no further interchange
of labels is needed.
Thus, beginning with the initial pair $23$, the successive
marked pairs form the sequence
\[
 23,\ 12,\ 13,\ 23,\ 12,\ 13,\ldots,
\]
until the remainder is a triangle.
This proves \eqref{eq62}.

Each removal deletes two vertices and leaves an admissible
marked remainder.
Since $|R|=2\gamma+1$, after $\gamma-1$ removals the
remainder has three vertices.
All three are marked vertices and hence have degree two,
so the remainder is a triangle.

It remains to verify that one threshold for $n$ suffices
for all the preceding comparisons.
There are exactly $\gamma-1$ removals.
In each application of \eqref{eq54}, the remainder has
absolute value at most
$108\rho^{-4}\|\boldsymbol{\alpha}_0\|_2^2$,
independently of the current remainder and its admissible
replacement.
The initial error constants depend only on $\gamma$, and
the update formulas and positive rescalings are applied at
most $\gamma-1$ times.
Thus all accumulated error constants remain bounded in
terms of $\gamma$.

The recurrence also shows that $\tau$ and $\delta$ take
only finitely many values depending on $\gamma$.
Every positive leading coefficient used above is of the form
$2(\tau-1)$, $2\tau\delta$, or $2$, and is at least $2$
whenever it occurs.
Therefore, all comparisons hold whenever
$\rho\ge\rho_0(\gamma)$ for some constant $\rho_0(\gamma)$.
Since $\rho\ge\sqrt{2(n-2)}$, this is true for all
sufficiently large $n$.
\end{proof}

By Lemma \ref{lem7.5}, the base graph $H_0$ of $H^*$ is isomorphic to
$K_1\nabla R$, with $v_0$ as its dominating vertex
and $v_0,v_1$ as its attachment vertices.
On the other hand, Definition \ref{def1.1} gives $H_{0,\gamma}=K_1\nabla R_{\gamma-1}$,
in which $v_0$ is also the dominating vertex and $v_0, v_1$ are its attachment vertices.

\begin{lem}\label{lem7.10}
We have $(R,v_1)\cong(R_{\gamma-1},v_{1}),$
and hence $H_0\cong H_{0,\gamma}$ by an isomorphism preserving
the ordered attachment vertices $v_0,v_1$.
\end{lem}

\begin{proof}
By Lemma~\ref{lem7.9}, after interchanging $v_2$ and $v_3$
if necessary, the sequence via \eqref{eq62} reduces $R$ to
a triangle in $\gamma-1$ removals.
We show that reversing the defining construction of
$R_{\gamma-1}$ produces the same sequence of marked pairs.

For $k=1,\ldots,\gamma-1$, the construction of $R_k$ from
$R_{k-1}$ adds the new vertices
$w_{2k+2},w_{2k+3}$ and the edges of the path
$ w_{2k+1}w_{2k+2}w_{2k+3}w_{2k-2}.$
It replaces the ordered triple
$(w_{2k-2},w_{2k},w_{2k+1})$ of $R_{k-1}$ by
$(w_{2k},w_{2k+2},w_{2k+3})$ in $R_k$.
Thus, in the reverse direction, the last two vertices of
the ordered triple of $R_k$ are removed.

Suppose that the ordered triple
$(w_{2k},w_{2k+2},w_{2k+3})$ of $R_k$ carries labels
$(\ell_1,\ell_2,\ell_3)$.
When $w_{2k+2}$ and $w_{2k+3}$ are removed,
Definition~\ref{def7.1} assigns label $\ell_2$ to
$w_{2k+1}$ and label $\ell_3$ to $w_{2k-2}$, while
$w_{2k}$ retains label $\ell_1$.
Hence, the ordered triple
$(w_{2k-2},w_{2k},w_{2k+1})$ of $R_{k-1}$ carries labels
$(\ell_3,\ell_1,\ell_2).$

By Definition \ref{def1.1}, the final ordered triple of $R_{\gamma-1}$ is
$(w_{2\gamma-2},w_{2\gamma},w_{2\gamma+1}),$
which carries labels $(1,2,3)$.
Reading the construction backward from $R_{\gamma-1}$ to
$R_0$, the successive ordered triples therefore carry labels
\[
 (1,2,3),\ (3,1,2),\ (2,3,1),\
 (1,2,3),\ldots.
\]
At each reverse step, the last two vertices of the current
ordered triple are removed.
Their unordered label pairs are consequently
$23,\ 12,\ 13,\ 23,\ldots,$
which is exactly the sequence in \eqref{eq62}.

Both reduction processes end with a triangle whose vertices
carry labels $1,2,3$.
There is a unique label-preserving isomorphism between these
two labeled triangles.
We now reverse the two removal sequences simultaneously.
At each step, part (iii) of Lemma \ref{lem7.1} shows that the inverse
expansion is uniquely determined once the correspondence
between the three labeled marked vertices is fixed.
Hence, the label-preserving isomorphism of the two triangles
extends inductively through all $\gamma-1$ inverse
expansions.

We thereby obtain an isomorphism
$\varphi:R\longrightarrow R_{\gamma-1}$
that maps the vertex carrying label $1$ in $R$ to the vertex
carrying label $1$ in the final ordered triple of $R_{\gamma-1}$.
Thus $\varphi(v_1)=w_{2\gamma-2},$ where $w_{2\gamma-2}$ is denoted by $v_1$ in the definition of $R_{\gamma-1}$,
and hence $(R,v_1)\cong(R_{\gamma-1},v_1)$.

Finally, extend $\varphi$ by mapping the dominating vertex
$v_0$ of $H_0=K_1\nabla R$ to the dominating vertex $v_0$
of $H_{0,\gamma}=K_1\nabla R_{\gamma-1}$.
This gives an isomorphism
$H_0\cong H_{0,\gamma}$ preserving the ordered attachment
vertices $v_0,v_1$.
\end{proof}

\begin{proof}[{\bf Proof of Theorem \ref{thm1.2} for $\gamma\geq5$}]
Let $a^*=\lceil\frac{n-2\gamma-4}{2}\rceil$ and $b^*=\lfloor\frac{n-2\gamma-4}{2}\rfloor.$
By Lemma \ref{lem7.3}, we have
$K_2\nabla H_{0,\gamma}(a^*,b^*)\in\mathbb G(n,\gamma)$.
By Lemmas \ref{lem7.5} and \ref{lem7.10}, the base graph $H_0$ of $H^*$ is isomorphic to $H_{0,\gamma}$,
with the ordered attachment vertices $v_0,v_1$.
Lemma \ref{lem5.3} implies $\{a,b\}=\{a^*,b^*\}$.
This does not yet determine their assignment when $a^*\ne b^*$.

If $a^*=b^*$, then $a=a^*$ and $b=b^*$.
Suppose that $a^*>b^*$.
The base $H_{0,\gamma}$ depends only on $\gamma$,
has minimum degree at least two, and satisfies
$d_{H_{0,\gamma}}(v_0)=2\gamma+1>3=d_{H_{0,\gamma}}(v_1)$.
By Lemmas \ref{lem7.2} and \ref{lem7.3}, $H_{0,\gamma}$ has a spanning path from $v_0$ to $v_1$,
and both assignments of the pendant paths are feasible.
If the longer path were attached at $v_1$,
part (ii) of Lemma \ref{lem5.3} would imply that
exchanging the two pendant paths strictly increases the
spectral radius, contradicting the maximality of $\rho(G^*)$.
Thus the longer path is attached at $v_0$, and hence $a=a^*$ and $b=b^*$.
Therefore, $G^*\cong K_2\nabla H_{0,\gamma}(a^*,b^*)$.
This completes the proof.
\end{proof}

\end{document}